\documentclass [11 pt, letterpaper, twoside, openright]{book}
\usepackage{amsfonts,amssymb,amscd,xy}
\usepackage[leqno]{amsmath}
\usepackage{mathrsfs}
\usepackage{amssymb}

\usepackage{amsmath}

\usepackage{hyperref}

\usepackage{amsxtra}

\usepackage{amsthm}
\usepackage{tabularx}
\usepackage{epsfig,amssymb}
\usepackage{bm} 
\usepackage{verbatim} 

\usepackage{hyperref}

\usepackage{verbatim} 

\makeatletter

\DeclareMathAlphabet{\mathcalligra}{T1}{calligra}{m}{n}

\numberwithin{equation}{section}

\newcommand{\isoto}{\overset{\!\sim}{\lra}}

\def\Aut{{\mathrm{Aut}}}
\def\uaut{\underline{\Aut}}

\newcommand{\lra}{\longrightarrow}

\def\isom{{\rm Isom}}

\def\hex{\underline{\rm Hex}}

\newcommand{\bj}{{\rm Ob}}
\newcommand{\cart}{{\rm Cart}}

\newcommand{\uinn}{\underline{{\rm Int}}}

\newcommand{\inn}{{\rm int}}

\newcommand{\aut}{{\rm Aut}}

\def\uisom{\underline{\rm Isom}}

\def\uhom{\underline{\rm Hom}}

\def\be{\kern -.1em}
\def\le{\kern 0.03em}
\def\lbe{\kern -.025em}

\newcommand{\Hom}{ \mathrm{Hom}}

\def\e{\kern 0.08em}

\newcommand{\ccc}{{  \lower0.09ex\hbox{{\text{\large$\circ$}}}}}
\newcommand{\dd}{-\!\!\!}
\newcommand{\gir}{{\ \dd\dd\dd\ \ccc\ }}

\newcommand{\ggto}[1]{\overset{#1}{\gir}}

\newtheorem{lemma}{Lemma}[section]
\newtheorem{theorem}[lemma]{Theorem}
\newtheorem{proposition-definition}[lemma]{Proposition-Definition}
\newtheorem{corollary}[lemma]{Corollary}
\newtheorem{proposition}[lemma]{Proposition}
\theoremstyle{definition}
\newtheorem{definition}[lemma]{Definition}

\theoremstyle{remark}
\newtheorem{remark}[lemma]{Remark}
\newtheorem{remarks}[lemma]{Remarks}
\newtheorem{example}[lemma]{Example}
\newtheorem{examples}[lemma]{Examples}

\begin{document}

\input xy     
\xyoption{all}

\title{Cohomology with values in a sheaf of crossed groups over a site}

\author{Raymond Debremaeker}
\date{}

\maketitle

\topmargin -1cm

\smallskip

\chapter*{Introduction}

Giraud's cohomology \cite{14} has two shortcomings.

\begin{enumerate}
\item[(1)] The $H^{2}$ is not functorial. Indeed, given a morphism of sheaves of groups $u\colon A\to B$ one only gets, in general, a {\it relation}
\[
H^{2}(A)\ggto{u^{(2)}} H^{2}(B).
\]
\item[(2)] A short exact sequence of sheaves of groups
\[
1\to A\overset{\!u}{\to}B\overset{\!v}{\to}C\to 1
\]
does not induce an exact cohomology sequence. Indeed, Giraud's $H^{2}(A)$ is not sufficiently large to contain all the obstructions to lifting a $C$-torsor to $B$. This is the reason why Giraud was forced to introduce a new set $O(v)$ which contains all the obstructions.
\end{enumerate}

We have set ourselves the goal of eliminating these two shortcomings in order to obtain a truly functorial cohomology theory. We have been mainly inspired by Dedecker \cite{4}, who has overcome the difficulties that arise in connection with the non-abelian 2-cohomology of a topological space 
by choosing a suitable category of coefficients. In the first chapter we introduce a number of concepts: sites, toposes, stacks on a site, objects with group actions, torsors.
In chapter \ref{cg} we show how Giraud arrives at the 2-cohomology groups with values in a sheaf of groups using gerbes and how a band is associated to a gerbe. In addition, we indicate where
the defects of his theory lie and in what special cases they vanish. In chapter \ref{bun} the obstructions to lifting a torsor are subjected to a new study. This discussion gives rise to the concept of an $(A,\Pi)$-gerbe. We prove a functoriality theorem to the effect that, using an $(f, q)$-morphism $(A,\Pi)\to (A^{\prime}, \Pi)$, we can construct an $(A^{\prime}, \Pi)$-gerbe out of an $(A,\Pi)$-gerbe. The new $H^{2}$ is defined in Chapter \ref{four}, where it is proven that it is functorial and that to a short exact sequence of sheaves of crossed groups on a site there corresponds an exact cohomology sequence.  Finally, in Chapter \ref{rel} the new $H^{2}$ is compared to the $H^{2}$ of Giraud, which is denoted by $H^{2}_{g}$ in that Chapter.

If $P$ is a presheaf on $E/S$, we denote by $P(f)$ the value of $P$ at the {\em object} $f\colon T\to S$ of $E/S$.  The result of the action of an element $a$ on an element $p$ is denoted by $p*a$ or $a*p$ depending on whether the action occurs along the right or left side.

I would like to thank my supervisor Prof. A. Warrinnier for giving me the opportunity to carry out this work. His constructive criticism has motivated me to complete this work despite the difficult conditions under which it had to be carried out. Finally, I would like to thank Prof. L. De Greve because he has given me the possibility to stay at this University.

\bigskip
\bigskip

\hskip 12cm R. Debremaeker

\tableofcontents

\chapter{Stacks on a site}\label{ss}

\section{Sites, toposes, stacks on a site}

\subsection{Sites}

\begin{definition} Let $C$ be a category. A {\em sieve $\mathcal R$ of the category $C$} is a subclass $\mathcal R$ of ${\rm Ob}(C)$ such that, for every arrow $m\colon x\to y$ in $C$ whose target $y$ belongs to $\mathcal R$, the source $x$ also belongs to $\mathcal R$.
Sometimes we will identify $\mathcal R$ with the full subcategory of $C$ whose class of objects is $\mathcal R$.
\end{definition}

Let $x$ be an object of $C$.  In what follows, the sieves of the category $C/x$ will be called {\em the sieves of $x$}.

\begin{definition} Let $\mathcal R$ be a sieve of $x$ and let $f\colon y\to x$ be a morphism in $C$. The class $\mathcal R^{\e f}$ of all those arrows $m\colon z\to y$ in $C$ such that
the composition $f\circ m\colon z\to x$ is in $\mathcal R$ is a sieve of $y$. We call it the inverse image of $\mathcal R$ under $f$.	
\end{definition}

Using the concept of sieve of an object we can now define the concept of topology on a category.

\begin{definition} We have a topology $J$ on a category $C$ if to every object $x$ of $C$ there corresponds a class $J(x)$ of sieves of $x$ such that the following three conditions hold:
\begin{enumerate}
\item[(i)] {\it Stability under base change}. For every object $x$ of $C$, every sieve
$\mathcal R\in J (x)$, every morphism $f\colon y\to x$ in $C$, the sieve $\mathcal R^{f}$ belongs to $J(y)$.
\item[(ii)] {\it Local character}.  Let $\mathcal R$ and $\mathcal R^{\prime}$ be two sieves of $x$.  If $\mathcal R\in J(x)$ and if, for every $f\colon y\to x$ in $\mathcal R$, the sieve $(\mathcal R^{\prime}\e)^{f}$ belongs to $J(y)$, then $\mathcal R^{\prime}$ belongs to $J(x)$.
\item[(iii)] For every object $x\in C$, $C/x$ belongs to $J(x)$.
\end{enumerate}	
\end{definition}

The sieves belonging to the set $J(x)$ are called {\it covering sieves of $x$} or also {\it refinements of $x$}.

\begin{definition} A category $C$ equipped with a topology is called a site. The category $C$ is called the underlying category of the site.	
\end{definition}

Let $x$ be an object of $C$. A family of morphisms $m_{i}\colon x_{i}\to x$, $i\in J$, in $C$ is called a {\it covering} if the sieve generated by this family is a covering sieve. By a family of topological generators of the site $C$ we mean a set $G$ of objects from $C$ such that every object of $C$ is the target of a covering family of morphisms of $C$ whose sources belong to $G$. To be able to give a correct definition of terms such as category of sets, presheaves of sets on a site, topos of sheafs of sets on a site, we need the concept of {\it universe} \cite[I\e]{26}. In this paragraph we will denote a universe by the letter $U$ or $V$.

\begin{definition}	Let $U$ be  a universe.  By a {\it $U$-site $C$} we mean a site $C$ such that:
\begin{enumerate}
\item[\rm (i)] the underlying category $C$ is a $U$-category, i.e., for every pair of objects $(x, y)$ in $C$ the set $\Hom(x, y)$ is isomorphic to an element of $U$, and
\item[\rm (ii)] there exists a family of topological generators $x_{i}$, $i\in J$, where $J\in U$.  We call $(x_{i}), i\in J$, a $U$-small (or also small) family of generators.
\end{enumerate}
\end{definition}

\subsection{Toposes}

\begin{definition} A {\it $U$-topos $E$} is a category $E$ for which there exists a site $C\in U$ such that $E$ is equivalent to the category $\widetilde{C}$ of $U$-sheaves of sets on $C$.
\end{definition}

\begin{remark} When we have a $U$-topos $E$ we will always consider it endowed with its {\it canonical topology}. Thus considered $E$ is then a $U$-site.
\end{remark}

\smallskip

Now we observe that the category $\widetilde{C}$ of sheaves on a site $C\in U$ has the following four properties:
\begin{enumerate}
\item[(a)] finite projective limits exist in $\widetilde{C}$
\item[(b)] direct sums indexed by an element of $U$ exist in  $\widetilde{C}$ and are disjoint and universal
\item[(c)] equivalence relations are universally effective
\item[(d)] $\widetilde{C}$ contains a family of generators indexed by an element of $U$
\end{enumerate}

The following theorem states that a $U$-topos is nothing but a $U$-category having the above four properties.

\begin{theorem} {\rm (Giraud)} Let $E$ be a $U$-category. The following statements are equivalent:
\begin{enumerate}
\item[\rm (i)] $E$ is a $U$-topos.
\item[\rm (ii)] $E$ has properties (a), (b), (c) and (d) above.
\item[\rm (iii)] Every $U$-sheaf on $E$ relative to the canonical topology is representable and $E$ has a small family of generators.
\item[\rm (iv)] There exist a $U$-site $C$ such that $E$ is equivalent to  the category $\widetilde{C}$ of $U$-sheaves on $C$.
\end{enumerate}
\end{theorem}

\begin{examples}\indent
\begin{enumerate}
\item {\it Topos associated to a topological space.}

Let $X$ be a small topological space and let ${\rm Ouv}(X)$ be the category of open subsets of $X$ with the canonical topology. The topos of $U$-sheaves on the $U$-site ${\rm Ouv}(X)$ is denoted by
\[
{\rm Top}(X).
\]
Here we can observe that the $U$-sheaves ${\rm Ouv}(X)$ are, in fact, nothing else than the sheaves of sets on the topological space $X$ in the sense of Godement.

According to \cite[pp.~110-111]{15}, a sheaf of sets on $X$ is the same thing as an {\it \'etale space $\widetilde{\mathcal F}$ over $X$.} In fact, to a given \'etale space $X^{\prime}$ over $X$, one can associate its sheaf of sections. Conversely, Godement shows that given a sheaf $\mathcal F$ on $X$ one  can construct an  \'etale space $\widetilde{\mathcal F}$ over $X$ such that the sheaf of sections of $\widetilde{\mathcal F}$ is up to isomorphism equal to the original sheaf $\mathcal F$ \cite[Theoreme 1.2.1, p. 111]{15}. This shows that the topos ${\rm Top}(X)$ is equivalent to the category of \'etale spaces over $X$.

\item {\it The topos $U$-{\rm Ens} of $U$-sets.}

Let $U$-Ens be the category of sets belonging to $U$. When we take $X$ to be a topological space {\it reduced to one point}, then we have the following functor
\[
{\rm Top}(X)\to \text{$U$-{\rm Ens}},\, F\mapsto F(X).
\]
This functor is an equivalence of categories, whence $U$-Ens is a $U$-topos.

\item {\it The topos $\widehat{C}$.}

Let $C$ be a $U$-small category, i.e., ${\rm Ob}(C)$ and ${\rm Mor}(C)$ are isomorphic to elements of $U$. Then one can consider the category of sheaves on $C$ with respect to the chaotic topology. This shows that $\widehat{C}$ is a $U$-topos.
\end{enumerate}
\end{examples}

\begin{definition} Let $E$ and $E^{\prime}$ be two $U$-toposes. By a {\it morphism of toposes from $E$ to $E^{\prime}$} we understand an ordered $3$-tuple $(u_{*},u^{*}, \varphi\e)$, where $u_{*}\colon E\to E^{\prime}$ and $u^{*}\colon E^{\prime}\to E$ are functors and $\varphi\colon \Hom_{E}(u^{*}(-),-)\isoto \Hom_{E^{\prime}}(-,u_{*}(-))$ is an adjunction isomorphism . Moreover, we require that $u_{*}$ commute with finite projective limits.

$u_{*}$ is called the {\it direct image functor} and $u^{*}$ is called the {\it inverse image functor} of the morphism of toposes $u$.
\end{definition}

\begin{example} Let $C$ be a small site, i.e., ${\rm Ob}(C)$ and ${\rm Mor}(C)$ are isomorphic to elements of $U$.  To $C$ one can associate two toposes $\widehat{C}$ and $\widetilde{C}$. Between the two we have the well-known associated sheaf functor
\[
a\colon \widehat{C}\to\widetilde{C}
\]	
Since $a$ commutes with finite projective limits and the  inclusion functor $i: \widetilde{C} \to \widehat{C}$ has a right adjoint,  we can interpret $a$ as the inverse image functor of a morphism of toposes $p\colon \widetilde{C} \to \widehat{C}$. Thus we have $p^{*}=a$ and $p_{*}=i$.
\end{example}

\subsection{Stacks on a site}

\begin{definition} Let $E$ be a site. An {\it $E$-stack} $F$ is a fibered $E$-category $F$ such that, for every $S\in{\rm Ob}(E)$ and every refinement $\mathcal{R}$ of $S$, the restriction functor
\[
\underline{\rm Cart}_{\e E}(E/S, F\e)  \to \underline{\rm Cart}_{\e E}(\mathcal{R}, F\e)
\]
is an {\it equivalence}. If this restriction functor is only {\it fully faithful}, then we say that $F$ is an {\it $E$-prestack}.
\end{definition}

\begin{remarks} \indent
\begin{enumerate}
\item In the previous definition $E/S$ is regarded as a fibered $E$-category via the source functor
\[
E/S\to E,\, (T\to S)\mapsto T
\]
Further, we interpret the covering sieve $\mathcal{R}$ as a full subcategory of $E/S$. Thus $\mathcal{R}$ is a fibered $E$-category via the following composition of functors
\[
\mathcal{R}\to E/S\to E.
\]
\item By $\underline{\rm Cart}_{\e E}(E/S, F)$ we mean the {\it category} of cartesian $E$-functors from $E/S$ to $F$. The {\it set} of cartesian $E$-functors from $E/S$ to $F$ will be denoted by ${\rm Cart}_{E}(E/S, F\e)$.
\end{enumerate}
\end{remarks}

\subsubsection{Practical meaning of {\it prestack} and {\it stack}}

Let $F$ be a prestack on $E$, $\mathcal R$ a refinement of an object $S$ of $E$, $\lambda, \mu$ two cartesian $E$-functors from $E/S$ to $F$.

We set $x=\lambda(S,1_{S})$, $y=\mu(S,1_{S})$,
\[
x_{i}=\lambda(S_{i}, s_{i}), y_{i}=\mu(S_{i}, s_{i})
\]
for every $s_{i}\colon S_{i}\to S$ in $\mathcal R$. Assume that coherent family $\{m_{i}\colon x_{i}\to y_{i}\}$, $i\in J$, of $S_{i}$-morphisms is given. Then the restriction functor $\underline{\rm Cart}_{\e E}(E/S, F)  \to \underline{\rm Cart}_{\e E}(\mathcal{R}, F)$ is fully faithful if
there exists {\it one, and only one,} $S$-morphism $m\colon x\to y$ in $F$ such that $m\circ\lambda(s_{i})=\mu(s_{i})\circ m_{i}$ for all $i\in J$. We say that {\it morphisms in $F$ can be glued together}.

\begin{proposition} Let $E$ be a site and $F$ a fibered $E$-category. In order for $F$ to be a prestack, it is necessary and sufficient that, for every $S\in{\rm Ob}(E)$ and every pair of objects $(x, y)$ in $F_{S}$, the presheaf $\uhom_{\e S}(x,y)$ be a sheaf on $E/S$.
\end{proposition}

\begin{remark} The {\it set} of $S$-morphisms from $x$ to $y$ will we denoted by $\Hom_{S}(x,y)$.
\end{remark}

If $F$ is a {\it stack} on $E$, then it is certainly a prestack and morphisms can be glued together in $F$. Let $x_{i}\in {\rm Ob}(F_{S_{i}})$, $s_{i}\colon S_{i}\to S\in\mathcal R$ be a {\it coherent} family of objects of $F$, in other words, it consists of a cartesian $E$-functor  $\lambda\colon\mathcal R\to F$ such that $\lambda(S_{i},s_{i})=x_{i}$ for every $s_{i}\colon S_{i}\to S\in\mathcal R$. Since the functor $\underline{\rm Cart}_{\e E}(E/S, F)  \to \underline{\rm Cart}_{\e E}(\mathcal{R}, F)$ is an equivalence, there exists a cartesian $E$-functor $\lambda^{*}\colon E/S\to F$ such that the restriction
$\lambda^{*}\vert_{\e\mathcal R}$ is isomorphic to $\lambda$. This means that there exists an object $x\in F_{S}$ which is locally isomorphic to the given objects $x_{i}$. More concisely,  we say that {\it the objects in $F$ can be glued together}. We conclude that an $E$-stack is a fibered $E$-category in which morphisms and objects can be glued together.

\begin{examples} In the examples that follow, the notation is that of Giraud \cite[II]{14}.
\begin{enumerate}
	
\item ${\rm Fl}(E)$.

Let $E$ be a category in which {\it finite fibered products} exist. We define a category ${\rm Fl}(E)$ over $E$. If $S\in{\rm Ob}(E)$, then we take as fiber category ${\rm Fl}(E)_{S}$ the category of objects over $S$. Thus we have ${\rm Fl}(E)_{S}=E/S$. Let $f\colon T\to S$ be any morphism in $E$, $x\colon X\to S$ an object of ${\rm Fl}(E)_{S}$ and $y\colon
Y\to T$ an object in ${\rm Fl}(E)_{T}$. As $f$-morphisms from $y\colon Y\to T$ to $x\colon X\to S$ we take all those morphisms $m\colon Y\to X$ in $E$ which satisfy $x\circ m=f\circ y$. Since finite fibered products exist in $E$, ${\rm Fl}(E)$ is a fibered $E$-category.

{\it If $E$ is a topos, then ${\rm Fl}(E)$ is a stack} \cite[II, 3.4.8.2]{14}.

\item {\it The split stack of sheaves of sets on $E$: ${\rm FAISCIN}(E)$.}

Let $E$ be a site. Then we have a contravariant functor
\[
{\rm FAISCIN}(E)\colon E^{\e\rm op}\to ({\rm Cat})
\]
which associates to every $S\in{\rm Ob}(E)$ the category $\widetilde{E/S}$ and to every arrow $f\colon T\to S$ the functor $\widetilde{E/S}\to\widetilde{E/T}$ which is obtained by composition with the functor $E/f\colon E/T\to E/S$.

Since projective limits indexed by a category that is an element of $U$ exist in $U$-Ens, we conclude from \cite[II, 3.4.4]{14} that ${\rm FAISCIN}(E)$ is a {\it stack} on $E$. Finally, note that there exists an equivalence of categories.
\[
\widetilde{E}\isoto\varprojlim({\rm FAISCIN}(E)/E\e)
\]
Here $\varprojlim({\rm FAISCIN}(E)/E)$ is the {\it category of cartesian sections} of ${\rm FAISCIN}(E)$.

\item {\it The split stack of sheaves of groups on $E$: ${\rm FAGRSC}(E)$.}

Let $E$ be a site. We write $U$-Gr for the category of groups belonging to $U$ and ${\rm Fagr}(E)$ for the category of sheaves on $E$ with values in $U$-Gr.

We define a split $E$-category by means of the contravariant functor
\[
{\rm FAGRSC}(E)\colon E^{\e\rm op}\to ({\rm Cat})
\]
which associates to every $S\in{\rm Ob}(E)$ the category ${\rm Fagr}(E/S)$ and to every arrow $f\colon T\to S$ the functor ${\rm Fagr}(E/S\e)\to{\rm Fagr}(E/T\e)$ which is obtained by composition with the functor $E/f\colon E/T\to E/S$. According to \cite[II, 3.4.4]{14}, ${\rm FAGRSC}(E)$ is a stack. Also here we have an equivalence of categories
\[
{\rm Fagr}{E}\isoto\varprojlim({\rm FAGRSC}(E)/E\e).
\]
\end{enumerate}
\end{examples}

\section{Objects with group actions. Torsors}

From now on we will work in a fixed universe $U$. This will not be mentioned again. When we speak of a site, a topos, a presheaf, a sheaf, we actually mean a $U$-site, a $U$-topos, a $U$-presheaf, a $U$-sheaf.

\subsection{Objects with group actions in a category}

Let $E$ be a category, $G$ a group object in $\widehat{E}$, i.e., a presheaf with values in the category of groups, $P$ any object in $\widehat{E}$.

\begin{definition} We say that $P$ is a {\it left} (respectively, {\it right}) {\it $G$-object} if, for every object $S\in E$, the set $P(S)$ is endowed with a structure of left (respectively, right) $G(S)$-set in such a way that, for every arrow $f\colon T\to S$ in $E$, the map $P(f)\colon P (S)\to P (T)$ is compatible with the group homomorphism $G(f)\colon G(S)\to G (T)$. Equivalently, there exists a morphism
\[
\mu\colon G\times P\to P \quad \text{(respectively, $\mu\colon P\times G\to P\e$)}
\]
such that, for every $S\in\bj(E)$, the set $P(S)$ is endowed with a structure of $G(S)$-set. 
\end{definition}

The previous definition enables us to define the concept of an object with a group action in any category.

\begin{definition} Let $E$ be a category and $A$ an object in $E$. We say that $A$ is a {\it group object} in $E$ if, for every object $X$ in $E$, a group law on the set $A(X)=\Hom(X,A)$ is given such that, for every arrow $f\colon Y\to X$ in $E$, the map $A(f)\colon A(X)\to A(Y)$ is a morphism of groups. Thus, to say that $A$ is a group object in the category $E$ is equivalent to saying that $A$ is an object in $E$ for which $h_{A}$ is a sheaf of groups on $E$. If $E$ has a final object $e$ and finite products exist in $E$, then the above is equivalent to giving morphisms
\[
\mu\colon A\times A\to A, i\colon A\to A, \nu\colon e\to A
\]
such that, if $\phi\colon A\to e$ is the final morphism of $A$, then the following diagrams commute:
\[
\xymatrix{A\ar[ddrr]_{1_{\lbe A}}\ar[r]^{(1_{A},\phi)}\ar[d]_{(\phi,1_{A})}& A\times e\ar[r]^{1_{A}\times\nu}& A\times A \ar[dd]^{\mu}\\
e\times A\ar[d]_{\nu\times 1_{A}}&&\\
A\times A\ar[rr]^{\mu}&& A,}
\]
\[
\xymatrix{A\ar[dd]_{(1_{A},i)}\ar[dr]^{\phi}\ar[rr]^{(i,1_{A})}&& A\times A \ar[dd]^{\mu}\\
& e\ar[dr]^{\nu}&\\
A\times A\ar[rr]_{\mu}&& A}
\]
and
\[
\xymatrix{A\times A\times A\ar[d]_{\mu\times 1_{A}}\ar[rr]^{1_{A}\times \mu}&& A\times A\ar[d]^{\mu}\\
A\times A\ar[rr]_{\mu}&& A}
\]
\end{definition}

\begin{definition} Let $A$ be a group object and $S$ an arbitrary object in a  category $E$. 
We will say that $S$ is a {\it left} (respectively, {\it right}) {\it $A$-object} if $h_{S}$ is a left (respectively, right) $h_{A}$-object in $\widehat{E}$. When finite products and a final object $e$ exist in $E$, then we can say that a structure of left $A$-object on an object $S$ of $E$ is equivalent to giving a morphism in $E$
\[
\lambda\colon A\times S\to S
\]
such that the following diagrams commute
\[
\xymatrix{A\times A\times S\ar[d]_{1_{A}\times\lambda}\ar[rr]^{\tau\times 1_{S}}&& A\times S\ar[d]^{\lambda}\\
A\times S\ar[rr]_{\lambda}&& S}
\]
and
\[
\xymatrix{e\times S\ar[drr]_{{\rm pr}_{2}}\ar[rr]^{(\nu, 1_{S})}&& A\times S\ar[d]^{\lambda}\\
&& S.}
\]
Here $\tau\colon A\times A\to A$ is the morphism that defines the group structure on $A$ and $\nu\colon e\to A$ is the unit section of $A$. We say that $S$ is an object with left group action.

Notation: $(S, A,\lambda)$ or, more briefly, $(S, A)$.

Similar considerations apply to objects endowed with right group actions.
\end{definition}

The previous definition shows that a functor that commutes with finite projective limits transforms an object with a group action into an object with a group action.

\begin{example} Let $E$ be a site, $G$ a sheaf of groups on $E$. Thus $G$ is a group object in $\widetilde{E}$. Let $P$ be a sheaf of sets on $E$, so that $P\in{\rm Ob}(\widetilde{E})$. A left action of $G$ on $P$ is equivalent to giving, for each object $S$ of $E$, a left action of the group $G(S)$ on the set $P(S)$ such that, for every arrow $f\colon T\to S$ in $E$, the map $P(f)\colon P(S)\to P(T)$ is compatible with the group homomorphism $G(f)\colon G(S)\to G(T)$.

This is also the same as giving a morphism of sheaves of groups
\[
\mu\colon G\to\aut(P).
\]	
\end{example}

\begin{remark} In the following, when we discuss objects with group actions without mentioning left or right, then a {\it right} action must be assumed (as in Giraud \cite{14}).
\end{remark}

\subsubsection{Notation}

The category of objects with group actions in $E$ is denoted by
\[
{\rm Oper}(E)
\]
If $G$ is a group in the category $E$, then a right (or left) {\it $G$-object of $E$} is an object with a right (or left) group action where the acting group is equal to $G$. The category of right (respectively, left) $G$-objects in $E$ will de denoted by
\[
{\rm Oper}(E; G)\quad (\text{respectively, } {\rm Oper}(E; G^{\e\rm o}))
\]

\subsubsection{The fibered $E$-category ${\rm Oper}(E)$.}\label{216}

We assume here that finite fiber products exist in $E$. Then we define an $E$-category ${\rm Oper}(E)$ as follows:

\begin{itemize}
\item the objects with projection $S\in{\rm Ob}(E)$ are the objects with group actions in $E/S$. The morphisms with projection $1_{S}$ are those morphisms in $E/S$ which are compatible with the structures on both ends. Thus we have ${\rm Oper}(E)_{S}={\rm Oper}(E/S)$.

\item the morphisms with projection $f\colon T\to S$ from an object $(P^{\e\prime},G^{\e\prime})\in {\rm Oper}(E)_{T}$ to an object $(P, G)\in {\rm Oper}(E)_{S}$ are all those $f$-morphisms from $(P^{\e\prime},G^{\e\prime})$ to $(P, G)$ that respect the structures.

\end{itemize}

Let $(P, G)$ be an object of ${\rm Oper}(E)_{S}$ and $f\colon T\to S$ an arrow in $E$.
Since finite fiber products exist in E, we can form the following fiber products:
\[
\xymatrix{P^{\e T}=P\times_{S}T\ar[d]\ar[r]& P\ar[d]\\
T\ar[r]^{f}& S}
\]
and
\[
\xymatrix{G^{\e T}=G\times_{S}T\ar[d]\ar[r]& G\ar[d]\\
T\ar[r]^{f}& S}
\]
Then $(P^{\e T}, G^{\e T})\in {\rm Oper}(E)_{T}$ is an inverse $f$-image of $(P, G)$.
It follows, in addition, that a composition of two cartesian morphisms is cartesian.

The $E$-category ${\rm Oper}(E)$ is thus {\it fibered}.

When $E$ is a site, the fibered $E$-category ${\rm Oper}(E)$ is not, in general, a stack on $E$. But if $E$ is a {\it topos}, then ${\rm Oper}(E)$ is a {\it stack on $E$} \cite[III, 1.1.7]{14}.

\begin{remark} Let $G$ be a group object in $E$. Similarly to the previous construction, we can define a fibered $E$-category ${\rm Oper}(E;G)$ whose fiber over $S$ is equal to the category
${\rm Oper}(E/S;G\times S)$ of objects in $E/S$ endowed with a $(G\times S)$-action.
\end{remark}

\subsection{Torsors in a topos}

\begin{definition}\label{dt} Let $E$ be a topos and $S$ an object in $E$.  A {\it torsor in $E$ over $S$} is an $S$-object of $E$ with group action $(P, G, m)$ such that:
\begin{enumerate}
\item[\rm (i)] the morphism $P\to S$ is an epimorphism, and
\item[\rm (ii)] the morphism $u\colon P\times_{S}G\to P\times_{S}P, (p,g)\mapsto (p,m(p,g)),$ is an isomorphism.
\end{enumerate}
\end{definition}

What do these conditions mean?

Condition (i): according to \cite[III, 1.4.1.1]{14}, the requirement means that $P\to S$ is an epimorphism and there exists an epimorphic family $\{S_{i}\to S\}$, $i\in J$, such that $\Hom_{\e S}(S_{i},P)\neq\emptyset$ for all $i\in J$.

Now we know that, in regard to the canonical topology of a topos, a family is epimorphic if, and only if, it is covering. Thus condition (i) states that {\it $P$ has a section locally.}

Condition (ii) expresses the fact that $G$ acts {\it freely and transitively} on $P$. When only this condition is satisfied, we say that $P$ is a {\it pseudo-torsor.}

By a {\it torsor of $E$} we mean an $e$-torsor, where $e$ is the final object of the topos $E$. In this sense, we can say that an $S$-torsor of $E$ is the same thing as a torsor of $E/S$. If we want to indicate the operating group, then we will speak of a $G$-torsor on $S$, where $G$ represents the operating group.

\begin{proposition} The $G$-object $G_{d}$ is a torsor.
\end{proposition}
\begin{proof} By $G_{d}$ we mean the right $G$-object obtained by letting $G$ act on itself via right translations. The morphism $G_{\le d}\to S$ is an epimorphism since, for every $X\to S$ in $E/S$, $\Hom_{S}(X, G_{\le d})$ is a group and therefore has at least one element. Note that the set of all arrows of $E$ with target $S$ is a covering sieve of $S$ for the canonical topology of the topos, whence it is an epimorphic family.

Also, condition (ii) is fulfilled since argument by argument we get the operation of a group on itself via right translations and this is a free transitive action.
\end{proof}

\begin{definition} The $G$-torsor $G_{d}$ is called the {\it trivial torsor}.
\end{definition}

\begin{proposition}\label{t1} Being a torsor is a local property.
\end{proposition}
\begin{proof} We need to show that $P$ is a $G$-torsor as soon as this is so locally.

The condition that $P\to S$ be an epimorphism is certainly local since this expresses the fact that $P$ has a section locally. That $P\times_{S}G\to P\times_{S}P$ is an isomorphism is also a local condition since ${\rm Fl}(E)$ is an $E$-stack because $E$ is a topos.
\end{proof}

\begin{proposition}\label{t2} The property of being a torsor is stable under base changes.
\end{proposition}
\begin{proof} The final morphism $P\to S$ remains an epimophism after base change since in a topos the epimorphisms are universal.
The isomorphism $P\times_{S}G\to P\times_{S}P$ remains so after a base change because of the functorial nature of the base change operations.
\end{proof}

\subsubsection{The stack of torsors on a topos $E$}

As we have seen in \ref{216}, the $E$-category ${\rm OPER}(E)$ is a stack on $E$.
We consider now the full subcategory ${\rm TORS}(E)$ whose objects with projection $S\in{\rm Ob}(E)$ are the $S$-torsors of $E$. By Propositions \ref{t1} and \ref{t2}, we conclude that ${\rm TORS}(E)$ is an $E$-stack. We call ${\rm TORS}(E)$ the {\it stack of torsors on $E$}. If $G$ is a group of $E$ then we can, in a similar fashion starting with ${\rm OPER}(E;G)$, define the stack of $G$-torsors. We denote it by ${\rm TORS}(E;G)$.

The {\it category} of torsors of $E$ is denoted by ${\rm Tors}(E)$.

\begin{proposition} \label{ltor} Let $S$ be an object of a topos $E$ and let $G$ be a group of $E$.
\begin{enumerate}
\item[\rm (i)] In order for a pseudo-torsor $P$ under $G$ over $S$ to be $G$-isomorphic to $G_{d}$, it is necessary and sufficient that $P$ have a section.
\item[\rm (ii)] In order for a $G$-object $P$ to be a $G$-torsor, it is necessary and sufficient that $P$ be locally isomorphic to the trivial torsor $G_{d}$.
\end{enumerate}
\end{proposition}
\begin{proof}\indent
\begin{enumerate}
\item[\rm (i)] According to \cite[III, 1.2.7 (i bis)]{14}, we have a canonical bijection
\[
\Hom_{G}(G_{d},P)\isoto P(S)=\Hom_{S}(S,P)
\]
Here $\Hom_{S}(S,P)$ is the set of sections of $P$. Since $P$ is a pseudo-torsor, every $G$-morphism $G_{d}\to P$ is a $G$-isomorphism.  Thus (i) is proved.

\item[\rm (ii)] Let $P$ be a $G$-torsor.  Then $P$ is a $G$-torsor locally and certainly also a pseudo-torsor.  Since $P\to S$ is an epimorphism, $P$ has a section locally and so, by (i), $P$ is locally $G$-isomorphic to the trivial $G$-torsor $G_{d}$. Conversely, assume that $P$ is locally isomorphic to $G_{d}$. Then $P$ is a $G$-torsor locally. By Proposition \ref{t1}, we conclude that $P$ is a $G$-torsor.	

\end{enumerate}
\end{proof}

\begin{proposition}\label{228} In the stack ${\rm TORS}(E;G)$, the following properties hold.
\begin{enumerate}
\item[\rm (i)] Any two $G$-torsors are locally $G$-isomorphic.
\item[\rm (ii)] Every $G$-morphism of $G$-torsors is an isomorphism.
\end{enumerate}
\end{proposition}
\begin{proof} Assertion (i) is clear when one notes that a $G$-torsor is locally isomorphic to the trivial torsor $G_{d}$. To prove (ii), we note that, since ${\rm TORS}(E;G)$ is a stack, it is enough to prove that every $G$-morphism is locally an isomorphism. But our two $G$-torsors locally reduce to the trivial $G$-torsor $G_{d}$ and every morphism $G_{d}\to G_{d}$ is an isomorphism. 
\end{proof}

\begin{remark} Later, we will express (i) and (ii)  more succinctly by saying that ${\rm TORS}(E;G)$ is a {\it gerbe} on $E$.
\end{remark}

\begin{definition} Let $P$ be an $A$-torsor of the topos $E$. Then we can consider the following sheaf of groups
\[
{\rm ad}(P)=\uaut_{\le A}(P\e)	
\]
Since $E$ is a topos, the indicated sheaf determines a group object of $E$ that we denote by the same symbol. We call this group object the {\it adjoint group of the $A$-torsor $P$}.
\end{definition}

\subsection{Torsors on a site}

We will show now how the concept of {\it torsor on a site} can be defined using the concept of {\it torsor on a topos} of Definition \ref{dt}.

\begin{definition} Let $E$ be a site.
\begin{enumerate}
\item[\rm (i)] A torsor on $E$ is a torsor of the topos $\widetilde{E}$.
\item[\rm (ii)] Let $S$ be an object of $E$. An $S$-torsor on $E$ is an $\varepsilon(S)$-torsor of the topos $\widetilde{E}$. Here $\varepsilon\colon E\to \widetilde{E}$ is the functor that maps an object $S\in E$ to the object $a(h_{S})\in \widetilde{E}$.
\item[\rm (iii)] Let $G$ be a sheaf of groups on $E$ or, in other words, a group object of $\widetilde{E}$. A $G$-torsor on $E$ is a $G$-torsor of the topos $\widetilde{E}$.
\end{enumerate}
\end{definition}

When we say that $P$ is a $G$-torsor on the site $E$, we mean that
\begin{itemize}
\item $G$ is a sheaf of groups on $E$.
\item $P$ is a sheaf of sets on $E$.
\item $G$ acts freely and transitively on $P$.
\item The final morphism $P\to e$ is an epimorphism, i.e., $P$ locally has a section.
\end{itemize}

\begin{definition} We write ${\rm Tors}(E)={\rm Tors}(\widetilde{E})$ and call ${\rm Tors}(E)$ the {\it category of torsors on $E$.} Further, we set ${\rm TORS}(E)={\rm TORS}(\widetilde{E})\times_{\widetilde{E}}E$ and we call ${\rm TORS}(E)$ the {\it stack of torsors on $E$}.

If $G$ is a sheaf of groups on $E$, then we define similarly
${\rm Tors}(E;G)={\rm Tors}(\widetilde{E};G)$ and ${\rm TORS}(E;G)={\rm TORS}(\widetilde{E};G)\times_{\widetilde{E}}E$.
\end{definition}

The relation between the category and the stack of torsors on the site $E$ is indicated by the following equivalences of categories \cite[III, 1.7.1.3]{14}:
\[
{\rm Tors}(E)\overset{\!\approx}{\lra}\varprojlim({\rm TORS}(E)/E\e)
\]
\[
{\rm Tors}(E;G)\overset{\!\approx}{\lra}\varprojlim({\rm TORS}(E;G)/E\e)
\]
Finally, we can introduce two {\it split} $E$-stacks, namely the split $E$-stack ${\rm TORSC}(E)$  of torsors on $E$ and the split $E$-stack ${\rm TORSC}(E;G)$ of $G$-torsors on $E$. 

${\rm TORSC}(E)$ is determined by the contravariant functor from $E$ to ${\rm Cat}$ which maps an object $S\in E$ to the category ${\rm Tors}(E/S)$.

${\rm TORSC}(E;G)$ is the split $E$-stack defined by the contravariant functor $S\mapsto {\rm Tors}(E/S, G^{S}\,)$, where $ G^{S}$ is the restriction of $G$ to the site $E/S$.

According to \cite[II, 3.4.8]{14}, we have {\it $E$-equivalences of stacks}:
\[
{\rm TORS}(E)\isoto{\rm TORSC}(E)
\]
\[
{\rm TORS}(E;G)\isoto{\rm TORSC}(E;G)
\]

\section{Contracted product of torsors}
\subsection{Contracted product}
Let $G$ be a group object of the topos $T$, $P$ a {\it right} $G$-object and $Q$ a {\it left} $G$-object of $T$. Then we define a diagonal action of $G$ on $P\times Q$
\[
d\colon (\le P\times Q\le)\times G\to P\times Q, \quad(p,q,g)\mapsto (p\cdot g, g^{-1}\cdot q).
\]
Since we work in a topos $T$, we can consider the cokernel of the following pair of morphisms
\[
(\le P\times Q\le)\times G\,{{{} {\rm pr}_{\lbe 1}\atop \longrightarrow}\atop{\longrightarrow \atop \!d {}}}\,  P\times Q
\]
This cokernel is none other than the quotient of $P\times Q$ by the equivalence relation defined by the diagonal action of $G$ on $P\times Q$

\begin{definition} The {\it contracted product $P\be {{{} G\atop \bigwedge}\atop }\lbe Q$} of $P$ with $Q$ over $G$ is the quotient object of $P\times Q$ by the diagonal action of the group $G$.
\end{definition}
 
\begin{proposition} Let $A$ be a group object of $T$ which acts on $P$ on the left in a compatibe way with $G$, i.e., via $G$-automorphisms. Then $A$ also acts on $P\be {{{} G\atop \bigwedge}\atop }\lbe Q$ on the left.
\end{proposition}
\begin{proof} Since $A$ acts on $P$ compatibly with the right action of $G$, one can also transfer the action of $A$ to the product $P\times Q$ in such a way that this action is compatible with the equivalence relation defined by the diagonal action. We therefore obtain an action of $A$ on the quotient $P\be {{{} G\atop \bigwedge}\atop }\lbe Q$.
\end{proof}

\begin{remark} A similar result holds with regard to the second factor $Q$.
\end{remark}

\begin{proposition}\label{pgg} For each $G$-object $P$ of the topos $T$, we have an isomorphism
\[
P\isoto P\be {{{} G\atop \bigwedge}\atop }\lbe G.
\]	
This isomorphism is compatible with the structures of left $\uaut_{\le G}(P)$-object and right $G$-objects on both sides.
\end{proposition}
\begin{proof} In this proposition $G$ is regarded as a left $G$-object via the action of $G$ on itself by left translations.

We will show that following composition is an isomorphism
\[
P\overset{\!{\rm Id}_{P}\times e}{\lra} P\times G\overset{\!q}{\lra} P\be {{{} G\atop \bigwedge}\atop }\lbe G.
\]
Here $e\colon P\to G$ is the unit morphism and $q\colon P\times G\to P\be {{{} G\atop \bigwedge}\atop }\lbe G$ is the projection of the quotient object.
If $m\colon P\times G\to P$ is the morphism that defines the right action of $G$ on $P$ then $m\circ {\rm pr}_{\lbe 1}=m\circ d$. Thus there exists a unique morphism $n\colon P\be {{{} G\atop \bigwedge}\atop }\lbe G\to P$ such that $n\circ q = m$. If we set $n^{\prime}= q\circ ({\rm Id}_{P}\times e)$, then we can verify that $n\circ n^{\e\prime}={\rm Id}_{P}$ and $n^{\e\prime}\circ n=
{\rm Id}_{P\be {{{} G\atop \bigwedge}\atop }\lbe G}$. Thus $q\circ (\e{\rm Id}_{P}\times e)$ is an isomorphism.
\end{proof}

\begin{proposition} Let $u\colon G\to G^{\e\prime}$ be a morphism of groups of $T$, $m\colon P\to P^{\e\prime}$ a $u$-morphism of objects with right actions and $n\colon Q\to Q^{\e\prime}$ a $u$-morphism of objects with left actions. Then the product morphism
\[
m\times n\colon P\times Q\to P^{\e\prime}\times Q^{\e\prime}
\]
determines by passing to the quotient a morphism
\[
m\wedge n\colon P\be {{{} G\atop \bigwedge}\atop }\lbe Q\to P^{\e\prime} {{{} G^{\e\prime}\atop \bigwedge}\atop }\lbe Q^{\e\prime}
\]
\end{proposition}
\begin{proof} We consider the following diagram
\[
\xymatrix{(\le P\times Q\e)\times G\hspace{-.1cm} \ar[d]_{(m\times n\le)\times u}&\hspace{-.5cm} {{{} {\rm pr}_{\lbe 1}\atop \xrightarrow{\makebox[1.5cm]{}}}\atop{\xrightarrow{\makebox[1.5cm]{}}\atop \!d {}}}\hspace{.3cm}& \hspace{-1.4cm}\ar[d]\hspace{1cm} P\times Q\ar[d]_{m\times n}\ar[r]^{\hspace{-.4cm}q} & P {{{} G\atop \bigwedge}\atop } Q\\
(\le P^{\e\prime}\times Q^{\e\prime}\e)\times G^{\e\prime}\hspace{-.1cm} &\hspace{-.5cm} {{{} {\rm pr}_{\lbe 1}\atop \xrightarrow{\makebox[1.5cm]{}}}\atop{\xrightarrow{\makebox[1.5cm]{}}\atop \!d^{\e\prime} {}}}\,& \hspace{-.2cm} P^{\e\prime}\times Q^{\e\prime}\ar[r]^{\hspace{-.25cm}q^{\e\prime}}& P^{\e\prime} {{{} G^{\le\prime}\atop \bigwedge}\atop }\lbe Q^{\e\prime}.}
\]
We have $(m\times n\le)\circ {\rm pr}_{\lbe 1}={\rm pr}_{\lbe 1}\circ((m\times n\le)\times u\le)$. Because of the way the diagonal action is defined and because of the fact that $m$ and $n$ are $u$-morphisms, we also have $(m\times n\le)\circ d=d^{\,\prime}\circ((m\times n\le)\times u\le)$. It follows that $(q^{\e\prime}\circ(m\times n))\circ {\rm pr}_{1}=(q^{\e\prime}\circ(m\times n))\circ d$. Thus we have a unique morphism $m\wedge n\colon P\be {{{} G\atop \bigwedge}\atop }\lbe Q\to P^{\e\prime} {{{} G^{\e\prime}\atop \bigwedge}\atop }\lbe Q^{\e\prime}$ such that $(m\wedge n\le)\circ q=q^{\e\prime}\circ(m\times n\le)$.
\end{proof}

\begin{proposition} {\rm (Associativity of the contracted product)}\label{assoc}
Let $P$ be a right $G$-object, $Q$ a left $G$-object on which $H$ acts on the right compatibly with the action of $G$ and $R$ a left $H$-object. Then we have a canonical isomorphism
\[
\big(\e P\be {{{} G\atop \bigwedge}\atop }\lbe Q\big){{{} H\atop \bigwedge}\atop }\lbe R\isoto P\be {{{} G\atop \bigwedge}\atop }\big(\e Q\be {{{} H\atop \bigwedge}\atop }\lbe R\big)
\]
\end{proposition}
\begin{proof} By applying twice the definition of the contracted product, we find that $\big(\e P\be {{{} G\atop \bigwedge}\atop }\lbe Q\big){{{} H\atop \bigwedge}\atop }\lbe R$ is the quotient of $\big(\e P\be\times  Q\big)
\times R$ by the equivalence relation induced by the diagonal actions of $G$ and $H$.

$P\be {{{} G\atop \bigwedge}\atop }\big(\e Q\be {{{} H\atop \bigwedge}\atop }\lbe R\big)$ is nothing but the quotient of $P\times  (Q \times R)$ by the equivalence relation induced by the diagonal actions of $G$ and $H$. We now have a canonical isomorphism
\[
\big(\e P\be\times  Q\big)
\times R\isoto P\times  (Q \times R)
\]
which is compatible with the two equivalence relations.  It thus follows that this isomorphism induces an isomorphism of the quotient objects:
\[
\big(\e P\be {{{} G\atop \bigwedge}\atop }\lbe Q\big){{{} H\atop \bigwedge}\atop }\lbe R\isoto P\be {{{} G\atop \bigwedge}\atop }\big(\e Q\be {{{} H\atop \bigwedge}\atop }\lbe R\big).
\]
\end{proof}

\begin{proposition} The contracted product $P\be {{{} G\atop \bigwedge}\atop }\lbe Q$ is stable under bases changes.
\end{proposition}
\begin{proof} The contracted product $P\be {{{} G\atop \bigwedge}\atop }\lbe Q$ is a cokernel of the pair of morphisms $
(\le P\times Q\le)\times G\,{{{} {\rm pr}_{\lbe 1}\atop \longrightarrow}\atop{\longrightarrow \atop \!d {}}}\,  P\times Q$
and is thus an inductive limit. In a topos inductive limits are universal, so a base change of a contracted product is a contracted product.
\end{proof}

\subsection{Extension of the structural group}

Let $u\colon F\to G$ be a morphism of group objects of the topos $T$ and let $P$ be an $F$-object. The operation of extension of the structural group consists in regarding (via $u$) the right $F$-object $P$ as a right $G$-object $^{u}\be P$ constructed in such a way that this operation is left adjoint to the operation of restriction of the structural group.

How does this work?

The group $G$ acts on $G_{d}$ on the left via left translations, so $F$ also acts on $G_{d}$ on the left via $u\colon F\to G$. Therefore we can form the contracted product of $P$ with $G_{d}$. Because of the way $G$ acts on $G_{d}$ from the right, we see that $P\be {{{} F\atop \bigwedge}\atop }\lbe G_{d}$ is a {\it right $G$-object}.

\begin{definition} \label{321} We set $^{u}\be P=P\be {{{} F\atop \bigwedge}\atop }\lbe G_{d}$ and say that $^{u}\be P$ is derived from $P$ by extending its structural group $F$ to $G$.

The composition of $({\rm Id}_{P}, e)\colon P\to P\times G_{d}$ with the projection $P\times G_{d}\to P\be {{{} F\atop \bigwedge}\atop }\lbe G_{d}$ yields a $u$-morphism denoted as follows
\[
Pu\colon P\to ^{u}\!\!\be P.
\]
For every $G$-object $Q$, the map 
\[
\Hom_{G}(^{u}\be P, Q)\to \Hom_{F}(P, Q^{u}),\, n\mapsto n\circ Pu,
\]
is a {\it bijection}. Here $Q^{u}$ is the $F$-object obtained from the $G$-object $Q$ by restriction of the structural group.

This shows that the operation of extending the structural group $P\to ^{u}\be P$ is {\it left adjoint} to the operation of restricting the structural group $Q\to Q^{u}$. The adjunction morphism is $Pu\colon P\to (^{u}\be P)^{u}$.
\end{definition}

\begin{proposition} The operation of extending the structural group sends torsors to torsors.
\end{proposition}
\begin{proof} If $P$ is an $F$-torsor and $u\colon F\to G$ is a morphism of group objects of $T$, then we must show that $^{u}\be P=P\be {{{} F\atop \bigwedge}\atop }\lbe G_{d}$ is a $G$-torsor. According to Proposition \ref{ltor}(ii), it suffices to show that $^{u}\be P$ is locally isomorphic to the trivial torsor $G_{d}$. We know that $P$ is locally isomorphic to the trivial $F$-torsor $F_{d}$. By Proposition \ref{pgg}, $F_{d}\be {{{} F\atop \bigwedge}\atop }\lbe G_{d}$ is $G$-isomorphic to $G_{d}$. Thus locally we have a $G$-isomorphism $^{u}\be P=P\be {{{} F\atop \bigwedge}\atop }\lbe G_{d}\isoto 
G_{d}$ whence $^{u}\be P$ is a $G$-torsor.	
\end{proof}

\subsection{Contracted products in the topos $\widetilde{E}$ of sheaves of sets on a site $E$}

Let $E$ be a site, $G$ a sheaf of groups on $E$, $P$ a right $G$-object and $Q$ a left $G$-object of $\widetilde{E}$.

$P\times Q$ is the sheaf product on $E$, i.e., it associates with each object $S\in E$ the set $P(S)\times Q(S)$.

The diagonal action of $G$ on $P\times Q$ is such that, for each object $S\in E$, we have a diagonal action of the group $G(S)$ on the set $(P \times Q)(S)= P(S)\times Q(S)$.  We can denote it as follows:
\[
d(S)\colon (P(S)\times Q(S))\times G(S)\to P(S)\times Q(S),\, ((p,q),g)\mapsto (pg,g^{-1}q).
\]
The contracted product $P\be {{{} G\atop \bigwedge}\atop }\lbe Q$  is the quotient sheaf of $P\times Q$ by the equivalence relation $\mathcal R_{d}$ induced by the diagonal action $d$.

To define the quotient sheaf we first consider the quotient presheaf $K$, i.e., the contravariant functor from $E$ to the category of sets ${\rm Ens}$ which assigns, to every object $S$ of $E$, the quotient set $P(S)\times Q(S)/\mathcal R_{d(S)}$. The projection from $P\times Q$ to $K$
\[
\mu\colon P\times Q\to K
\]
is an {\it epimorphism} of presheaves of sets on $E$. If $(p, q)\in (P \times Q)(S)$, then $\mu(p, q)=[(p, q)]$ is the class of $(p, q)$ modulo $\mathcal R_{d(S)}$.

Then we consider the {\it sheaf associated with $K$} and this is then the quotient sheaf of $P\times Q$ modulo $\mathcal R_{d}$. Thus we have $P\be {{{} F\atop \bigwedge}\atop }\lbe Q=a(K)$. In this way we obtain a {\it bicovering} morphism of presheaves \cite[II]{26}
\[
\lambda\colon K\to a(K)=P\be {{{} F\atop \bigwedge}\atop }\lbe Q
\]
What do we mean when we say that $\lambda$ is {\it bicovering}?

Let $x$ be an element of $(P\be {{{} G\atop \bigwedge}\atop }\lbe Q)(S)$. Then there exists a refinement $V$ of $S$ such that for every $s_{i}\colon S_{i}\to S\in V$ there exists an element $x_{i}\in K(S_{i})$ such that $\lambda_{S_{i}}(x_{i})=x^{\e s_{i}}$. If $x_{i}^{\e\prime}\in K(S_{i})$ is also an element with $\lambda_{S_{i}}(x_{i}^{\e\prime})=x^{\e s_{i}}$, then  $x_{i}$ and $x_{i}^{\e\prime}$ agree locally. Now we have $K(S_{i})=(P(S_{i})\times Q(S_{i}))/\mathcal R_{d(S_{i})}$ so that the element $x_{i}$ is in some equivalence class $[(p_{i},q_{i})]$. {\it In what follows we will state this briefly by saying that $x$ is locally determined by $(p_{i},q_{i})$ or even that $(p_{i},q_{i})$ is a local representative of $x$ relative to $\lambda,\mu$}.  From the foregoing it is clear what is the meaning of the expressions {\it locally determined by} and {\it local representative of} in the more general case in which one has a {\it covering} morphism $M\to N$ of presheaves on $E$.

\subsection{Twisting}
Let $E$ be a site, $C$ an $E$-stack and $G$ a sheaf of groups on $E$.

\begin{definition} Let $S$ be an object of $E$. By an {\it $S$-$G$-object with operators from $C$} we mean a pair $(x, u)$ with $x\in{\rm Ob}(C_{S})$ and $u\colon G^{S}\to\uaut_{\e S}(x)$ a morphism of sheaves of groups on $E/S$. If $(x, u)$ and $(y, v)$ are $S$-$G$-objects with operators from $C$, then we write $\Hom_{S,G}(x,y)$ for the set of all those $S$-morphisms $x\to y$ in $C$ which are compatible with the actions of $G$ on $x$ and $y$.
\end{definition}

We can now define an $E$-stack whose fiber over $S$ is none other than the category of $S$-$G$-objects with operators from $C$. This stack is called the {\it $E$-stack of $G$-objects with operators from $C$.}

\medskip

Notation: ${\rm OPER}(G; C)$ 

\medskip

Furthermore, we set
\[
{\rm Oper}(G; C)=\varprojlim({\rm OPER}(G; C)/E\e)
\]
and call it the {\it category} of $G$-objects with operators from $C$. 
To see now what an object with operators from $C$ in the stack $C$ really is, we have to consider what a cartesian section of ${\rm OPER}(G; C)$ represents.
So a cartesian section (thus in effect an object of ${\rm Oper}(G; C)$) can be interpreted as a pair $(s, u)$, where $s\colon E\to C$ is a cartesian section of $C$ and $u\colon G\to\aut(s)$ is a morphism of sheaves on groups on $E$.

\subsubsection{The functor ${\rm Tw}$}

Let $T$ be the full subcategory of ${\rm TORSC}(E; G)$ where, in each fiber,  we take as object the trivial torsor $G_{d}$.

We then have a functor
\[
s\colon T\times_{E}{\rm OPER}(G; C)\to C
\]
which assigns to an object $(G_{d}, (x,u))$ the object $x\in C_{S}$.

\begin{theorem} There exist a morphism of stacks
\begin{equation}\label{3.4.2}
{\rm Tw}\colon {\rm TORSC}(E; G)\times_{E}{\rm OPER}(G; C)\to C
\end{equation}
and an isomorphism of morphisms of stacks
\[
i\colon s\isoto {\rm Tw}\be\mid_{\e T\times_{E}\e{\rm OPER}(G;\e C)}.
\]
\end{theorem}
\begin{proof} The definition of $T$ shows that the inclusion functor
\[
T\to {\rm TORSC}(E; G)
\]
is bicovering.  Therefore the following inclusion functor is also bicovering
\[
T\times_{E}{\rm OPER}(G; C)\to {\rm TORSC}(E; G)\times_{E}{\rm OPER}(G; C).
\]
This makes it possible to consider ${\rm TORSC}(E; G)\times_{E}{\rm OPER}(G; C)$ as a stack associated to $T\times_{E}{\rm OPER}(G; C)$ \cite[II, 2.1.3]{14}.
By the definition of associated stack \cite[II, 2.1.1]{14}, we have an equivalence of categories
\[
\underline{\cart}_{\e E}({\rm TORSC}(E; G)\times_{E}{\rm OPER}(G; C),C)\isoto\underline{\cart}_{\e E}(T\times_{E}{\rm OPER}(G; C),C).
\]
Now $s$ is an object of $\underline{\cart}_{\e E}(T\times_{E}{\rm OPER}(G; C),C)$. Thus there exists an object ${\rm Tw}$ in  $\underline{\cart}_{\e E}({\rm TORSC}(E; G)\times_{E}{\rm OPER}(G; C),C)$ such that ${\rm Tw}\mid_{\e T\times_{E}{\rm OPER}(G; C)}\,\,\isoto\,\, s$.
\end{proof}

\subsubsection{Twisted objects}

Let $P$ be a $G$-torsor on $E/S$ and $X=(x, u)$ an $S$-$G$-object with operators from the $E$-stack $C$. 

\begin{theorem}\indent
\begin{enumerate}
\item[\rm (i)] There exist an object ${}^{P}\!x$ in $C_{S}$ and also a $G$-morphism
\[
m_{P}\colon P\to\uisom_{\e S}(x,{}^{P}\!x)
\]
\item[\rm (ii)] If $({}^{P}\!x,m_{P})$ is a pair as in {\rm (i)}, then there exists a {\rm unique} $S$-isomorphism
\[
j_{x}\colon {}^{P}\!x\isoto {\rm Tw}(P,X)
\]
such that the following composition
\[
P\underset{(1)}{\isoto}\uisom_{\e G}(G_{d},P)\underset{(2)}{\isoto}
\uisom_{\e S}(\e{\rm Tw}(G_{d},X),{\rm Tw}(P,X))\underset{(3)}{\isoto}\uisom_{\e S}(x,{}^{P}\!x)
\]
is equal to $m_{P}\colon P\to\uisom_{\e S}(x,{}^{P}\!x)$.
\end{enumerate}
Here the arrow {\rm (1)} is defined using {\rm \cite[III, 1.2.7(i)]{14}}, arrow {\rm (2)} is induced by the morphism of stacks ${\rm Tw}$ and arrow {\rm (3)} is obtained by composition with the isomorphisms $i_{x}\colon x\isoto {\rm Tw}(G_{d},X)$ \eqref{3.4.2} and $j_{x}\colon {}^{P}\!x\isoto {\rm Tw}(P,X)$.
\end{theorem}
\begin{proof} 
(i) We set ${}^{P}\!x= {\rm Tw}(P, X)$ and take for $m_{P}$ the following composition
\[
P\underset{(1)}{\isoto}\uisom_{\e G}(G_{d},P)\underset{(2)}{\isoto}
\uisom_{\e S}(\e{\rm Tw}(G_{d},X),{\rm Tw}(P,X))\underset{(3^{\e\prime})}{\isoto}\uisom_{\e S}(x,{\rm Tw}(P, X)),
\]
where the arrow $(3^{\e\prime})$ is obtained by composition with $i_{x}\colon x\to {\rm Tw}(G_{d},X)$. Since this composition is a $G$-morphism, we obtain the required pair $({}^{P}\!x,m_{P})$.

(ii) Let $({}^{P}\!x,m_{P})$ be any pair that satisfies {\rm (i)}. The composition $(3^{\e\prime})\circ (2)\circ (1)$ determines a $G$-morphism
\[
n_{P}\colon P\to \uisom_{\e S}(x, {\rm Tw}(P,X))
\]
Since the operation of extending the structural group is left-adjoint to the operation of restricting the structural group and $P\be {{{} G\atop \bigwedge}\atop }\lbe \uaut_{\e S}(x)$ and $\uisom_{\e S}(x, {\rm Tw}(P,X))$ are
$\uaut_{\e S}(x)$-torsors, $n_{P}$ determines an $\uaut_{\e S}(x)$-isomorphism
\begin{equation}\label{1}
P\be {{{} G\atop \bigwedge}\atop }\lbe \uaut_{\e S}(x)\isoto \uisom_{\e S}(x, {\rm Tw}(P,X))
\end{equation}
Similarly, the given $G$-morphism $m_{P}\colon P\to\uisom_{\e S}(x,{}^{P}\!x)$ induces an $\uaut_{\e S}(x)$-isomorphism
\begin{equation}\label{2}
P\be {{{} G\atop \bigwedge}\atop }\lbe \uaut_{\e S}(x)\isoto \uisom_{S}(x,{}^{P}\!x).
\end{equation}
From \eqref{1} and \eqref{2} we obtain an isomorphism
\[
\uisom_{\e S}(x,{}^{P}\!x)\isoto\uisom_{\e S}(x, {\rm Tw}(P,X)).
\]
According to \cite[III, 2.2.6]{14}, we may conclude that there exists an $S$-isomorphism $j_{x}\colon {}^{P}\!x\isoto {\rm Tw}(P,X)$. It is not difficult to check that this isomorphism satisfies the prescribed condition from which its uniqueness follows.
\end{proof}

\begin{definition}
Let $P$ be a $G$-torsor on $E/S$ and $X=(x, u)$ an $S$-$G$-object with operators from a stack $C$. By a {\it twisted object of $X$ by $P$} we understand a pair $({}^{P}\!x,m_{P})$, where ${}^{P}\!x\in{\rm Ob}(C/S)$ and $m_{P}\colon P\to\uisom_{\e S}(x,{}^{P}\!x)$ is a $G$-morphism.
\end{definition}

\begin{remark} The $G$-morphism $m_{P}\colon P\to\uisom_{S}(x,{}^{P}\!x)$ implies that the
sheaf $\uisom_{\e S}(x,{}^{P}\!x)$ has a local section so that $x$ and ${}^{P}\!x$
are {\it locally isomorphic}. Thus the twisted object of $X= (x, u)$ by $P$ is locally isomorphic to $x$.
\end{remark}

\subsubsection{A few examples of twisting}

\begin{enumerate}
\item[(a)]  {\it Twist of a sheaf of sets}.

A sheaf of sets $X$ on a site $E$ on which a sheaf of groups $A$ acts on the left may be viewed as an object with operators from the $E$-stack  ${\rm FAISCIN}(E)$.

If $P$ is an $A$-torsor on $E$ then one can consider the twisted object of $X$ by $P$.

Now we have the following morphism of sheaves of sets on $E$
\[
\lambda\colon P\to\uisom(X,P\be {{{} A\atop \bigwedge}\atop }\lbe X)
\]
defined by the formula  $\lambda(p)(x)= (\e p, x)^{*}$, where $(\e p, x)^{*}$ is the
image of $(\e p, x)$ under $P\times X\to P\be {{{} A\atop \bigwedge}\atop }\lbe X$.
This morphism $\lambda$ is an $A$-morphism and therefore defines on $P\be {{{} A\atop \bigwedge}\atop }\lbe X$ a structure of {\it twisted object of $X$ by $P$}.

\item[(b)]  {\it Twist of a sheaf of groups.}

If we let a sheaf of groups $A$ act on itself on the left via {\it inner automorphisms}, then we can consider $A$ as an $A$-object with operators from the $A$-stack ${\rm FAGRSC}(E)$. It is therefore possible to twist $A$ by means of a an $A$-torsor $P$ on the site $E$. To find the twisted sheaf of groups we consider the following morphism of sheaves of sets
\[
\mu\colon P\to\uisom_{\e\rm GR}(A,{\rm ad}(P))
\]
defined by the formula $\mu(\e p)(a)(\e p\cdot a^{\prime})=p\cdot(aa^{\prime})$, where $p\in P(S), a, a^{\prime}\in A(S), S\in{\rm Ob}(E)$. This morphism $\mu$ is an $A$-morphism and therefore endows the object ${\rm ad}(\e P\le)$ with the structure of a twisted object of $A$ by $P$ (in the stack of sheaves of groups on $E$). 
\end{enumerate}

\begin{remark}	The role of twisted object of $A$ by $P$ is also played by $P\be {{{} A\atop \bigwedge}\atop }\lbe A$ when $A$ acts on the left on $A$ via {\it inner automorphisms}.

Indeed, we can certainly consider $A$ as a left $A$-object in the stack of sheaves of {\it sets} on $E$. But then we know that the twisted object of $A$ by $P$ is none other than the contracted product $P\be {{{} A\atop \bigwedge}\atop }\lbe A$ and that thanks to the $A$-morphism
\[
\lambda\colon P\to\uisom(A,P\be {{{} A\atop \bigwedge}\atop }\lbe A),\, \lambda(p)(a)=(p,a)^{*}.
\]
Now we know that, since $A$ acts on the left via inner automorphisms, we
have a group structrue on $P\be {{{} A\atop \bigwedge}\atop }\lbe A$ induced by that on $A$.  Therefore we may regard $P\be {{{} A\atop \bigwedge}\atop }\lbe A$ as an object in the stack of sheaves of {\it groups} on $E$. Now $\lambda(p)$ is a morphism of groups since
\[
\lambda(p)(a\cdot a^{\prime})=(p, a\cdot a^{\prime})^{*}=(p, a)^{*}\cdot(p, a^{\prime})^{*}=\lambda(p)(a)\cdot \lambda(p)(a^{\prime})
\]
The $A$-morphism $\lambda$ thus determines an $A$-morphism
\[
\lambda^{\prime}\colon P\to\uisom_{\e\rm GR}(A,P\be {{{} A\atop \bigwedge}\atop }\lbe A)
\]
and therefore endows the object $P\be {{{} A\atop \bigwedge}\atop }\lbe A$ with the structure of a twisted object of $A$ by $P$ in the stack of sheaves of groups on $E$. It thus follows that we have an isomorphism
\[
{\rm ad}(P)\isoto P\be {{{} A\atop \bigwedge}\atop }\lbe A.
\] 	
\end{remark}

\chapter{The cohomology of Giraud}\label{cg}

We will show how Giraud is led to define his 2-cohomology in terms of gerbes and how he comes to understand the band of a gerbe.
A number of constructions and properties of bands will be discussed. The $H^{2}$ will be defined and this will also allow us to indicate where the shortcomings of his cohomology theory lie.

\section{Cohomology in dimensions $0$ and $1$}

\subsection{Definitions of $H^{\le 0}$ and $H^{1}$}

Let $E$ be a site and  let $A$ be a sheaf of groups on $E$.

\subsubsection{$H^{0}$}

We set $H^{\le 0}(E; A) = \varprojlim A$.

Sometimes we will use the abbreviated notation $H^{\le 0}(A)$. If $e$ denotes the {\it final} sheaf on $E$ then we have a canonical isomorphism
\[
H^{\le 0}(E; A)\isoto \Hom(e,A)
\]

A given morphism of sheaves of groups $u\colon A\to B$ induces a map from $H^{\le 0}(E; A)$ to $H^{\le 0}(E; B)$ that will be denoted as follows
\[
u^{(0)}\colon H^{\le 0}(A)\to H^{\le 0}(B)
\]

This map $u^{(0)}$ is defined using the universal character of a projective limit.

\subsubsection{$H^{1}$}

We define $H^{1}(E; A)$ as the set of isomorphism classes of $A$-torsors on $E$.  This is a pointed set; the base point is the class that contains the trivial $A$-torsor $A_{d}$. Here, also, we will sometimes use the shortened notation $H^{1}(A)$. Using a morphism of sheaves of groups $u\colon A\to B$, we can construct a map
\[
u^{(1)}\colon H^{\le 1}(A)\to H^{\le 1}(B)
\]
using the operation of extension of the structural group via $u$. Thus, if $[P]$ is an element of $H^{1}(A)$, then $u^{(1)}([P])=[{}^{u}\be P]$, where ${}^{u}\be P=P\be {{{} A\atop \bigwedge}\atop }\lbe B_{d}$.

\subsection{The coboundary map}

\subsubsection{The inverse image of a section}

Let $E$ be a site, $A$ a sheaf of groups and $X$ a sheaf of sets on $E$. A right action $m\colon X \times  A \to X$ of A on $X$ determines an equivalence relation on $X$. The quotient object $X/A$ is the cokernel of the pair of maps
\[
X\times A\,{{{} {\rm pr}_{\lbe 1}\atop \longrightarrow}\atop{\longrightarrow \atop \! m {}}}\,  X.
\]
This follows from the fact that $(X\times A, ({\rm pr}_{\lbe 1}, m))$ represents the equivalence relation on $X$ and from the definition of quotient object of an object modulo an equivalence relation.

If we denote the cokernel of the pair of maps $({\rm pr}_{\lbe 1}, m)$ by $q\colon X\to X/A$, then $X/A$ is the sheaf associated with the presheaf
\[
S\mapsto F(S)=X(S)/A(S)
\]
and $q$ is the following composition
\[
X\overset{p}{\lra} F \overset{\varphi}{\lra} a(F)=X/A,
\]
where $p\colon X\to F$ is equal to the projection of the quotient sheaf and $\varphi\colon F\to a(F)$ is the canonical map of the presheaf to its associated sheaf.

\begin{definition} Let $s\colon e\to X/A$ be a {\it section} of $X/A$. 
Since the composition of $s$ with the final morphism $X/A\to e$ is equal to ${\rm Id}_{e}$, $s$ is a {\it monomorphism}.

The inverse image of $s$ under $q$ is obtained by performing the following pullback in $\widetilde{E}$
\[
\xymatrix{q^{-1}(s)\ar@{-->}[d]_{s^{\prime}}\ar@{-->}[r]^{q^{\prime}}& e\ar[d]^{s}\\
X\ar[r]^{q}& X/A}
\]	
\end{definition}

\begin{lemma} If $s\colon e\to X/A$ is a section of $X/A$, then the inverse image
$q^{-1}(s)$ is a {\rm homogeneous} $A$-space, i.e., the action of $A$ on $X$ is transitive.
\end{lemma}
\begin{proof} The morphism $m\colon X\times A\to X$ determined by the action of $A$ on $X$ induces a unique morphism
\[
m^{\prime}\colon q^{-1}(s)\times A\to q^{-1}(s)
\]
such that $q^{\prime}\circ  m^{\prime}=\emptyset$  (final morphism $q^{-1}(s)\times A\to e$) and $s^{\prime}\circ  m^{\prime}=m\circ(s^{\prime}\times 1_{A})$. This morphism $m^{\prime}$ defines an action of $A$ on $q^{-1}(s)$ 
	
$q^{-1}(s)$ will be a homogeneous $A$-space if we can show that the quotient object $q^{-1}(s)/A$ is isomorphic to the final object $e$ \cite[III, 3.1.1.2]{14}. From the preceding we know that $q^{-1}(s)/A$ is the cokernel of the following pair of morphisms
\[
q^{-1}(s)\times A\,{{{} {\rm pr}_{\lbe 1}\atop \longrightarrow}\atop{\longrightarrow \atop \! m^{\prime} {}}}\, q^{-1}(s).
\]
Now consider the following diagram
\begin{equation}\label{cart}
\xymatrix{	q^{-1}(s)\times A\ar[d]_{s^{\prime}\times 1_{A}} \hspace{-.1cm} &\hspace{-.5cm} {{{} {\rm pr}_{\lbe 1}\atop \xrightarrow{\makebox[1.5cm]{}}}\atop{\xrightarrow{\makebox[1.5cm]{}}\atop \!m^{\prime} {}}}\hspace{.3cm}& \hspace{-1.4cm}\hspace{1cm} q^{-1}(s)\ar[d]_{s^{\prime}}\ar[r]^{q^{\e\prime}} & e\ar[d]_{s}\\
X\times A\hspace{-.1cm} &\hspace{-.5cm} {{{} {\rm pr}_{\lbe 1}\atop \xrightarrow{\makebox[1.5cm]{}}}\atop{\xrightarrow{\makebox[1.5cm]{}}\atop \!m {}}}\,& \hspace{-.2cm} X\ar[r]^{\hspace{-.3cm}q}& X/A.}
\end{equation}

One verifies directly that the following pullback squares are in $\widetilde{E}$:
\[
\xymatrix{q^{-1}(s)\times A\ar[d]_{s^{\prime}\times 1_{A}}\ar[rr]^{{\rm pr}_{\lbe 1}}&& q^{-1}(s)\ar[d]_{s^{\prime}}\\
X\times A\ar[rr]^{{\rm pr}_{\lbe 1}}&& X}
\]
and
\[
\xymatrix{q^{-1}(s)\times A\ar[d]_{s^{\prime}\times 1_{A}}\ar[rr]^{m^{\prime}}&& q^{-1}(s)\ar[d]_{s^{\prime}}\\
X\times A\ar[rr]^{m}&& X.}
\]	
The right-hand square in \eqref{cart} is also a pullback.

It still means that the base change $e\overset{s}{\lra} X/A$ of the anti-cone
\[
\xymatrix{X\times A\ar[dr]\hspace{-.1cm} &\hspace{-.5cm} {{{} {\rm pr}_{\lbe 1}\atop \xrightarrow{\makebox[4.8cm]{}}}\atop{\xrightarrow{\makebox[4.8cm]{}}\atop \!m {}}}&   \hspace{-.1cm}X\ar[dl]^{q}\\
&\hspace{-.1cm} X/A &.}
\]
is
\[
\xymatrix{q^{-1}(s)\times A\ar[dr]\hspace{-.1cm} &\hspace{-.5cm} {{{} {\rm pr}_{\lbe 1}\atop \xrightarrow{\makebox[4.8cm]{}}}\atop{\xrightarrow{\makebox[4.8cm]{}}\atop \!m^{\prime} {}}}&   \hspace{-.1cm}q^{-1}(s)\ar[dl]^{q^{\prime}}\\
&\hspace{-.1cm} e &.}
\]
Since $(X/A, q)$ is a cokernel of $({\rm pr}_{\lbe 1}, m)$ and in the topos $\widetilde{E}$ inductive limits are universal, we conclude that $(e, q^{\e\prime}\e)$ is a cokernel of the pair $({\rm pr}_{\lbe 1}, m^{\e\prime}\e)$.  Since $q^{-1}(s)/A$ is also a cokernel of this pair, we obtain that $q^{-1}(s)/A$ is isomorphic to $e$. Thus $q^{-1}(s)$ is a homogeneous $A$-space.

\end{proof}

{\bf Special case}: if $A$ acts freely on $X$, then the inverse image under $q$ of a section $s\colon e\to X/A$ is an {\it $A$-torsor}.

\subsubsection{The coboundary map}

Let $1\to A\overset{u}{\lra} B\overset{v}{\lra} C\to 1$ be a short exact sequence of sheaves of groups on $E$. Then we know, in particular, that $C$ is isomorphic to the quotient sheaf $B/A$. Thus the inverse image under $v$ of a section $s$ of $C$ is an $A$-torsor since $A$ acts freely on $B$. Therefore we have a map
\[
d\colon H^{\le 0}(C)\to H^{\le 1}(A), \, s\mapsto d(s)=[v^{-1}(s)].
\]
Finally, we can verify that the following sequence of pointed sets is {\it exact}
\[
1\to H^{\le 0}(A)\overset{u^{(0)}}{\lra}H^{\le 0}(B)\overset{v^{(0)}}{\lra}H^{\le 0}(C)
\overset{d}{\lra}H^{\le 1}(A)  \overset{u^{(1)}}{\lra} H^{\le 1}(B)  \overset{v^{(1)}}{\lra}H^{\le 1}(C)
\]
For more details we refer to \cite[III]{14}.

\section{Bands on a site}

\subsection{Obstructions to the lifting of a torsor}

To determine how $H^{2}$ should be defined we proceed in the classical way.
We consider a short exact sequence of sheaves of groups on $E$
\[
1\to A\overset{u}{\lra} B\overset{v}{\lra} C\to 1
\]
and, in addition, also a $C$-torsor $P$. Now we look for the obstruction to lifting $P$ to $B$ via the epimorphism $v\colon B\to C$. Giraud, inspired by Grothendieck, finds  the obstructions in an {\it $E$-category $K(P)$}. For every object $S$ of $E$, the fiber of $K(P)$ over $S$ is nothing but the category of liftings of $P^{S}$ to $B$, i.e., the category whose objects are pairs $(Q, \lambda)$, where $Q$ is a $B$-torsor on $E/S$ and $\lambda$ is a $v$-morphism from $Q$ to $P^{S}$.
He further noted that $K(P)$ is a stack on $E$ because the liftings of $P$ localize and glue together. In addition, this stack has the following properties:
\begin{enumerate}
\item[(1)] There exists a refinement $\mathcal R$ of $E$ such that ${\rm Ob}(K(P)_{S})\neq\emptyset$ for every $S$ in $\mathcal R$, in other words, $K(P)$ is locally nonempty. This is so because locally there exists a $v$-morphism from $B_{d}$ to $P$.
\item[(2)] Every $S$-morphism is an isomorphism, in other words, the fibers are {\it groupoids}. This is clear when one notes that a $B$-morphism between $B$-torsors is an isomorphism.
\item[(3)] Any two objects from $K(P)_{S}$ are locally isomorphic. Indeed, we have seen that two $B$-torsors are locally isomorphic.
\end{enumerate}

All of this is expressed briefly by saying that $K(P)$ is a {\it gerbe} on $E$. A gerbe is called {\it trivial} if it has a section. It is then clear that $K(P)$ is trivial exactly when $P$ can be lifted to $B$. Thus one can use $K(P)$ as an obstruction to the lifting. The idea is that $K(P)$ should determine an element of $H^{2}$, but then the question arises of how these sheaves of groups are related to the gerbe $K(P)$. One can quite easily see that, for every object $(Q, \lambda)$ in $K(P)$, $\aut_{S}(Q, \lambda)$ is {\it locally} isomorphic to $A$ and that these local isomorphisms differ only by inner automorphisms of $A$.

It follows that these local isomorphisms will not glue in the stack of sheaves of groups on $E$, but they will do so in a new stack which is obtained from the stack of sheaves of groups by annihilating the inner automorphisms. Giraud calls this new stack the {\it stack of bands on $E$}. We see, therefore, that, in general, it is not possible to associate a sheaf of groups to a gerbe but a new object called a {\it band} on $E$ (lien sur $E$).

\subsection{Bands on a site}

\subsubsection{The $E$-stack ${\rm LIEN}(E)$.}

We leave behind the split $E$-stack of sheaves of groups on the site $E$. To begin with, we associate to $E$ a {\it split $E$-prestack ${\rm LI}(E)$}, in other words, we define a contravariant functor
\[
{\rm LI}(E)\colon E^{\e \circ}\to({\rm Cat})
\]
\begin{itemize}
\item effect on the {\it objects}.

If $S\in{\rm Ob}(E)$, then ${\rm LI}(E)(S)$ is the category with the same objects as ${\rm FAGRSC}(E)_{S}$. For every pair $F, G\in{\rm Ob}({\rm FAGRSC} (E)_{S})$, we define the set of $S$-morphisms ${\rm Hex}(F,G)$ as follows
\[
{\rm Hex}(F,G)=H^{0}(S,\hex(F,G))
\]
Here $\hex(F,G)$ is the sheaf associated to the following presheaf on $E/S$:
\[
T\overset{f}{\to} S\to G(T)\backslash\Hom(F^{\e T},G^{\e T})/F(T)
\]
where $G(T)$ and $F(T)$ act via inner automorphisms.

\item effect on the {\it morphisms}

Let $f\colon T\to S$ be an arrow of $E$. Then ${\rm LI}(E)(f)$ is the following functor
\[
{\rm LI}(E)(S)\to {\rm LI}(E)(T), F\mapsto F^{\e T}=\text{the restriction of $F$ to $E/T$}
\]
Thus we have a split $E$-prestack ${\rm LI}(E)$ and, in addition, a {\it morphism of split $E$-categories}
\begin{equation}\label{tag}
{\rm FAGRSC}(E)\to {\rm LI}(E).
\end{equation}

This morphism is the identity on the objects and further, sends every
pair of morphisms of sheaves of groups that differ from each other by an inner automorphism to the same morphism in ${\rm LI}(E)$. Thus the inner automorphisms are annihilated under this morphism. 
\end{itemize}      
	
Finally, we consider the {\it associated stack} of the split $E$-prestack ${\rm LI}(E)$, which we denote by
\[
{\rm LIEN}(E).
\]
We have seen in the definition of associated stack that we have a bicovering functor
\begin{equation}\label{tag2}
{\rm LI}(E)\to {\rm LIEN}(E).
\end{equation}
Since ${\rm LI}(E)$ is already a prestack, bicovering here means that this functor is {\it fully faithful and locally surjective on objects} \cite[II, 1.4.5]{14}. By composing \eqref{tag} and \eqref{tag2}, we obtain a morphism  of split $E$-stacks
\[
{\rm lien}(E)\colon {\rm FAGRSC}(E)\to{\rm LIEN}(E)
\]
 
\begin{definition}\indent
\begin{itemize}
\item The split $E$-stack ${\rm LIEN}(E)$ is called the stack of bands on $E$.
\item By a band on $E$ we mean a cartesian section of the stack of bands ${\rm LIEN}(E)$. Thus we define the category of bands on $E$ as follows.
\[
{\rm Lien}(E)=\varprojlim({\rm LIEN}(E)/E)
\]

\item By taking cartesian sections of the morphism of stacks 
${\rm lien}(E)\colon {\rm FAGRSC}(E)\to{\rm LIEN}(E)$ and recallling the equivalence ${\rm Fagr}(E)\isoto\varprojlim({\rm FAGRSC}(E)/E)$ we obtain a functor denoted by ${\rm lien}(E)$:
\[
{\rm lien}(E)\colon {\rm Fagr}(E)\to {\rm Lien}(E).
\]

\item For every object $S$ of $E$, by a band on $S$ we mean an object of the fiber category ${\rm LIEN}(E)_{S}$.

The category ${\rm LIEN}(E)_{S}$ is called the {\it category of bands on $S$.}

\end{itemize}
	
\end{definition}

{\bf{Some terminology}}

Let $f\colon T\to S$ be an arrow of $E$ and $L$ a band on $S$, i. e., $L\in{\rm Ob}({\rm LIEN}(E)_{S})$. By analogy with the case of sheaves of groups, the inverse image of $L$ under $f$ will be called the restriction of $L$ to $T$.

Notation: $L^{T}$ or also $L^{f}$.

Thus if $L$ is a band on $E$, in other words, a cartesian section of ${\rm LIEN}(E)$, then the value of $L$ at $S$, $S\in{\rm Ob}(E)$, is called the restriction of $L$ to $S$.

Notation: $L^{S}$ or $L(S)$.

$L^{S}$ can also be interpreted as the restriction of the section $L$ to $E/S$.

\subsection{Local properties of bands}

\subsubsection{Representability of bands and morphisms of bands}

Let $A$ be a sheaf of groups on a site $E$. Via the functor ${\rm lien}(E)$, we associate to $A$ a band on $E$ denoted by ${\rm lien}(A)$. We say that ${\rm lien}(A)$ is the band determined by $A$.

Let $L$ be a band on $E$.

By a {\it representative of $L$} we mean a pair $(A, a)$ where $A$ is a sheaf of groups on $E$ and $a\colon {\rm lien}(A)\isoto L$ is an isomorphism of bands. We say that a band $L$ {\it representable} on $E$ if it has a representative.

By a {\it representative of a morphism of bands $f\colon L\to M$} we mean a $5$-tuple $(A,a,B,b,g)$ where $(A,a)$ and $(B,b)$ are representatives of $L$ and $M$ and $g\colon B\to A$ is a morphism of sheaves of groups such that 
$b\circ{\rm lien}(g)= f \circ a$.

Since the functor ${\rm lien}(E)\colon {\rm FAGRSC}(E)\to{\rm LIEN}(E)$ is bicovering we know that it is locally surjective on objects and morphisms. This means that {\it every band is locally representable and this is also the case for every morphism of bands.}

\subsubsection{Abelian bands}

An abelian band is a band that is represented by a sheaf of abelian groups.

In this context, we recall:

\begin{enumerate}
\item[\rm (i)]  In order for a band to be represented by a sheaf of abelian groups, it is necessary and sufficient that this be so locally.
\item[\rm (ii)] The stack of abelian bands ${\rm LIAB}(E)$ is $E$-equivalent to the split stack of sheaves of abelian groups on $E$.  So abelian bands are practically the same thing as sheaves of abelian groups \cite[IV, 1.2.3]{14}.
\end{enumerate}

\subsubsection{Unit bands and unit morphisms}

The band represented by the trivial sheaf of groups is a unit object in ${\rm LIEN}(E)$, i.e., it is both an initial and a final object. This band is called {\it the unit band.}

Notation: 1

Thus if $L$ and $M$ are bands, we will denote by
\[
L\overset{1}{\to}M
\]
the {\it unit morphism}, i.e., the composition of the initial morphism $1\to M$ and the final morphism $L\to 1$

\subsubsection{Injective, surjective, central and normal morphisms of bands}

A morphism of bands $u\colon L\to M$ is called injective, surjective, central or normal if there exist a refinement $\mathcal R$ of $E$ and, for every $S\in{\rm Ob}(\mathcal R)$, a representative $u^{*}\colon A\to B$ of the restriction $u^{S}$ of $u$ to $S$ which is a monomorphism of sheaves of groups, or an epimorphism, or that the image of $A$ lies in the center of $B$ or is such that $u^{*}(A)$ is a normal subgroup of $B$.

These conditions are stable under localization and are also local in nature.

\subsection{Some additional concepts}

We will show that a number of concepts from the category of sheaves of groups can be transferred to the category of bands.

\subsubsection{Cokernel of a morphism of bands}

By a cokernel of a morphism of bands $u\colon L\to M$ we mean a difference cokernel of the pair
\[
L{{{} 1\atop \longrightarrow}\atop{\longrightarrow \atop \! u {}}}\, M.
\]
It is shown in \cite[IV, 1.3.1]{14} that the cokernel of a morphism of bands exists. To do this, one shows that the functor ${\rm lien}(E\e)\colon{\rm Fagr}(E\e)\to{\rm Lien}(E\e)$ transforms cokernels into cokernels and further notes that a cokernel exists as soon as this is so locally.

\subsubsection{Direct product of bands}

Finite direct products exist in the category of bands and also in the fibers of the stack of bands. It is shown that the functor ${\rm lien}(E)\colon{\rm Fagr}(E)\to{\rm Lien}(E)$ commutes with finite products.  From this and from the existence of finite products of sheaves of groups one can show the existence of finite products of bands.

Let $(L\times M\overset{\!{\rm pr}_{1}}{\lra}L, L\times M\overset{\!{\rm pr}_{2}}{\lra}M)$ be a direct product of bands. By the universal property of a direct product, there exists a unique morphism ${\rm inj}_{1}\colon L\to  L \times M$ such that ${\rm pr}_{1}\circ  {\rm inj}_{1}=1_{L}$  and  ${\rm pr}_{2}\circ  {\rm inj}_{1}=1$.

One defines ${\rm inj}_{2}\colon M\to  L \times M$ similarly.

\subsubsection{Centralizer of a morphism of bands. Center of a band}

If $u\colon A\to B$ is a morphism of groups, then its centralizer is
\[
C_{u}=\{b\in B\mid b\cdot u(a)=u(a)\cdot b\,\,\text{for all}\, a,b\in A \}
\]
Thus formulated, the concept of centralizer does not transfer to bands. In order to make this possible, we note that $C_{u}$ represents the contravariant functor $\mathcal C_{u}$ that associates to every group $X$ the set
\[
\mathcal C_{u}(X)=\{x\in\Hom(A\times X,B)\mid x\circ{\rm inj}_{1}=u\}
\]
This is why we define the centralizer of a morphism of bands
\[
u\colon L\to M
\]
to be a pair $(C_{u},c_{u})$, where $C_{u}$ is a band and $c_{u}\colon L\times C_{u}\to M$ is a morphism of bands such that for every band $X$ the map
\[
c\colon \Hom(X, C_{u})\to \Hom(L\times X, M), x\mapsto c(x)=c_{u}\circ(\e {\rm id}_{L}\times x)
\]
induces a {\it bijection}
\[
\Hom(X, C_{u})\isoto \mathcal C_{u}(X).
\]

{\bf{Properties}} \cite[IV, 1.5.2]{14}

\begin{itemize}
\item A morphism of bands $u\colon L\to M$ always has a centralizer.
\item For every object $S$ of $E$, the restriction of $(C_{u},c_{u})$ to $S$ is a centralizer of the restriction of $u$ to $S$.
\item The functor ${\rm lien}(E)\colon {\rm Fagr}(E)\to {\rm Lien}(E)$ transforms a centralizer into a centralizer.
\item In order for $(C_{u},c_{u})$ to be a centralizer of $u$ it is necessary and sufficient that it be so locally.
\end{itemize}

\subsubsection{Center of a band}

By the center of a band $L$ we mean the centralizer of ${\rm Id}_{L}\colon L\to L$. Notation: $(C_{L},\gamma_{L}\colon C_{L}\to L)$.

{\bf{Some properties}}.

\begin{itemize}
\item The center $C_{L}$ of a band $L$ is an abelian band.
\item A morphism of bands $u\colon L\to M$ is central only if $\gamma_{u}\colon C_{u}\to M$ is an isomorphism.
\item If $u\colon L\to M$  is an epimorphism, then $C_{M}\isoto C_{u}$.
\end{itemize}

\subsubsection{Contracted product of two bands}

Let $u\colon C\to L$ and $v\colon C\to M$ be two morphisms of bands.  The contracted product of $L$ with $M$ over $C$ is the cokernel of the pair
\[
C\,{{{} (u,1)\atop \longrightarrow}\atop{\longrightarrow \atop (1,\e v){}}}\,  L\times M
\]

Notation: $L\be {{{} C\atop \bigwedge}\atop }\lbe M$

The contracted product of bands does not always exist because one cannot prove that every pair of morphisms of bands has a difference cokernel.

However, this contracted product does exist if one of the two morphisms, for example $v$, is {\it central}. In that case, the cokernel of the pair $((u, 1), (1, v))$ is none other than the cokernel of $(u,v^{-1})\colon C\to L\times M$.

Thus the contracted product exists if $v$ or $u$ is central. The contracted product is stable under localization and the functor ${\rm lien}(E)\colon{\rm Fagr}(E)\to{\rm Lien} (E)$ transforms contracted products into contracted products \cite[IV, 1.6.1.3]{14}.

\section{Gerbes. Band of a gerbe.}

\subsection{Gerbes on a site}

\begin{definition} A {\it gerbe $G$ on a site $E$} is an $E$-stack $G$ in groupo\"ids $G$ such that the following two conditions hold.
\begin{enumerate}
\item[(i)] There exists a refinement $\mathcal R$ of $E$ such that, for every $S\in{\rm Ob}(\mathcal R)$, the fiber category $G_{S}$ is nonempty.
\item[(ii)] For every $S\in{\rm Ob}(E)$ and every pair $x, y\in{\rm Ob}(G_{S})$, $x$ and $y$ are locally isomorphic.
\end{enumerate}
\end{definition}

\begin{remarks}\indent
\begin{enumerate}
\item[(a)] Recall that, by a refinement $\mathcal R$ of $E$, we mean a saturated subset $\mathcal R$ of ${\rm Ob}(E)$ such that, for every $S\in{\rm Ob}(E)$, $\{X\to S\mid \forall\, X\in\mathcal R\}$ is a covering sieve of $S$. If $E$ has a final object, then a refinement $\mathcal R$ of $E$ is nothing but a covering sieve of this final object.	
\item[(b)] Conditions (i) and (ii) together mean that the projection $G\to E$ is covering.	
\item[(c)] Condition (i) is certainly satisfied if $G$ has a section. If a gerbe $G$ has a section, then it is called a {\it trivial} gerbe.
\end{enumerate}
\end{remarks}

\subsubsection{Morphisms of gerbes}
A {\it morphism of gerbes} is a morphism of stacks on $E$ whose source and target are gerbes.  One defines in an analog manner the notion of {\it morphism of morphisms of gerbes}.

\begin{example}
Let $A$ be a sheaf of groups on a site $E$. According to Proposition \ref{228},  the $E$-stack ${\rm TORS}(E; A)$ is a gerbe on $E$.  This is a {\it trivial gerbe} since it has a section, namely the section determined by the trivial $A$-torsor $A_{d}$.	
\end{example}

\subsection{Band of a gerbe}

\subsubsection{Action of a band on a stack}

Let $F$ be a stack on the site $E$ and let $F^{\e \rm cart}$ \cite[II, 3.5.1]{14} be the underlying stack in groupoids on $E$. The objects of $F^{\rm cart}$ are the same as those of $F$ and the morphisms are the {\it cartesian} $E$-morphisms of $F$. This implies that the fibers of $F^{\rm cart}$ are grupoids. We have a morphism of stacks
\begin{equation}\label{aut}
\aut(F)\colon F^{\e\rm cart}\to {\rm FAGRSC}(E)
\end{equation}
that associates with $x\in{\rm Ob}(F_{S})$ a sheaf of groups $\uaut_{S}(x)$ on $E/S$ and with an $S$-morphism $m\colon x\isoto y$ in $ F^{\e\rm cart}$ the morphism of sheaves of groups ${\rm Int}(m)\colon \uaut_{\e S}(x)\to \uaut_{\e S}(y)$.

By composing $\aut(F)$ with ${\rm lien}(E)$ we obtain a morphism of stacks
\[
{\rm liau}(F)\colon F^{\e\rm cart}\to {\rm LIEN}(E)
\]
Let $L$ be a band on $E$, i.e., cartesian section $L\colon E\to {\rm LIEN}(E)$ of the stack of bands on $E$.  We write $f\colon F^{\e\rm cart}\to E$ for the projection from $F^{\e\rm cart}$ to $E$.

\begin{definition} An {\it action of $L$ on $F$} is a morphism of morphisms of stacks
\[
a\colon L\circ f\to {\rm liau}(F)
\]
\end{definition}

This means that, for every object $S$ of $E$ and every object $x$ in the fiber $F_{S}$, a morphism of bands on $S$
\[
a(x)\colon L(S)\to {\rm lien}(\uaut_{\e S}(x))
\]
is given in such a way that the family $\{a(x)\mid \forall S\in {\rm Ob}(E), \forall x\in {\rm Ob}(F_{S})\}$ is coherent.

\subsubsection{Morphisms of stacks with actions by bands}

Let $(L,a, F\e)$ and $(M, b, G\e)$ be stacks with actions by bands.
\[
m\colon F\to G
\]
a morphism of stacks, and 
\[
u\colon L\to M
\]
a morphism of bands on $E$.

We say that $m$ and $u$ are {\it compatible} or that $m$ is a {\it $u$-morphism} if the following condition holds:
\[
{\rm liau}(m)\circ a=(b\ast m^{\rm cart})\circ (u\ast f\e)
\]
This condition expresses that, for every $S\in {\rm Ob}(E)$ and every $x\in {\rm Ob}(F_{S})$, the following square must commute
\[
\xymatrix{L(S)\ar[d]_{a(x)}\ar[r]^{u(S)}& M(S)\ar[d]_{b(m(x))}\\
{\rm lien}(\uaut_{\e S}(x))\ar[r]^{\mu_{x}}& {\rm lien}(\uaut_{\e S}(m(x)))}.
\]
Here $\mu_{x}\colon \uaut_{\e S}(x)\to \uaut_{\e S}(m(x))$ is the morphism of sheaves of groups induced by $m\colon F\to G$.

\subsubsection{Band of a gerbe}

For an arbitrary $E$-stack, we have introduced the concept of action of a band.  When we see this in more detail for a particular type of stack, namely the gerbes, we find that under all possible actions of bands on a gerbe there exists one (unique up to isomorphism)
such that the structural morphism $a$ is an {\it isomorphism}.

For the proof we refer to \cite[IV, 2.2.2.3]{14}

\begin{definition} If $(L,a)$ is a band which acts on a gerbe $F$ so that $a$ is an isomorphism, then we say that $L$ is {\it the band of $F$} or also that $F$ is {\it bound by $L$}.	
\end{definition}

That $F$ is bound by $L$ means that for every $S\in {\rm Ob}(E)$ and every $x\in{\rm Ob} (F_{S})$ we have an isomorphism of bands on $E/S$
\[
a(x)\colon L(S)\isoto {\rm lien}(\uaut_{\e S}(x))
\]
such that the family $\{a(x)\mid \forall S\in {\rm Ob}(E), \forall x\in {\rm Ob}(F_{S})\}$
is compatible with the above constraints and for every $S$-isomorphism $i\colon x\to y$ of $F$ it must hold that the morphism of sheaves of groups ${\rm Int}(i)\colon \uaut_{\e S}(x)\to \uaut_{\e S}(y)$ represents the identity morphism of $L(S)$.

It is further proved that if $m\colon F\to G$ is a morphism of gerbes, $(L, a)$ is a band of $F$ and $(M, b)$ a band of $G$, then there exists unique morphism of bands
\[
u\colon L\to M
\]
such that $m$ is a $u$-morphism \cite[IV, 2.2.3]{14}.

\begin{definition} We say that the morphism of gerbes $m$ is {\it bound} by the morphism of bands $u\colon L\to M$.
\end{definition}

If follows from the definition of band of a gerbe and the definition of $u$-morphism that a morphism of gerbes $m\colon F\to G$ is bound by a morphism of bands $u\colon L\to M$ if, and only if, for every $S\in {\rm Ob}(E)$ and every $x\in {\rm Ob}(F_{S})$, the morphism of sheaves of groups $\uaut_{\e S}(x)\to \uaut_{\e S}(m(x))$ represents the restriction of $u$ to $S$.

\smallskip

{\bf Some additions}

\begin{itemize}
\item An $E$-gerbe $G$ is called an {\it abelian gerbe} if it is bound by an abelian band.
\item Let $F$ be a gerbe, $L$ the band of $F$ and $s$ a section of $F$.
If $A= \uaut(s)$ is the sheaf of automorphisms of $s$, then we have a canonical isomorphism of bands on $E$:
\[
L\isoto{\rm lien}(A).
\]
\item Let $m\colon F\to G$ be a morphism of gerbes and let $u\colon L\to M$ be the morphism of bands such that $m$ is bound by $u$. Then one has
\begin{enumerate}
\item[(i)] $m$ is faithful if and only if $u$ is injective.
\item[(ii)] $m$ is a covering if and only if $u$ is surjective.
\item[(iii)] $m$ is fully faithful if and only if $u$ is an isomorphism. In this case $m$ is even an equivalence \cite[IV, 2.2.6]{14}
\end{enumerate}
\end{itemize}

\subsection{The gerbe of liftings of a section}
In what follows we collect information which will later allow us to define the connecting morphism $d\colon H^{1}\to H^{2}$.

\subsubsection{The $E$-category $K(s)$}

Let $F$ and $G$ be two fibered $E$-categories, $m\colon F\to G$ a cartesian $E$-functor and $s\colon E\to G$ a cartesian section of $G$. We define an $E$-category $K(s)$ as follows:
\begin{itemize}
\item the objects with projection $S\in{\rm Ob}(E)$: all pairs $(a, \alpha)$ with 
$a\in{\rm Ob}(F_{S})$ and $\alpha\in\isom_{\e S}(m(a),s(S))$
\item the morphisms with source $(b, \beta)$, target $(a,\alpha)$ and projection $f\colon T\to S$: all $f$-morphisms $\psi\colon b\to a$ from $F$ for which $s(f\e)\circ\beta=\alpha\circ m(\psi\e)$
\item the composition of morphisms: via the composition in $F$.
\end{itemize}

\begin{theorem}\indent
\begin{enumerate}
\item[\rm (i)] $K(s)$ is a fibered $E$-category.
\item[\rm (ii)] We have a faithful and conservative cartesian $E$-functor 
\[
k(s)\colon K(s)\to F, (a,\alpha)\mapsto a
\]
as well as an $E$-isomorphism of functors
\[
\kappa(s)\colon m\circ k(s)\isoto s^{K(s)},\, (a,\alpha)\mapsto \alpha.
\]
Here $s^{K(s)}$ is the following composition
\[
K(s)\overset{\rm{proj.\,on}\, E}{\lra} E\overset{s}{\to} G
\]
\item[\rm (iii)] If $i\colon s\isoto t$ is an isomorphism of sections of $G$, then we have an $E$-functor
\[
K(i)\colon K(s)\to K(t), (a,\alpha)\mapsto (a,i(S)\circ \alpha),
\]
characterized by the following conditions
\begin{eqnarray*}
k(t)\circ K(i)&=& k(s)\\
\kappa(t)\ast K(i)&=& i^{K(s)}\circ \kappa(s)
\end{eqnarray*}
In addition,  $K(i)$ is an $E$-isomorphism. 
\end{enumerate}
\end{theorem}
\begin{proof} Since the morphisms of $K(s)$ are morphisms of $F$ and $F$ is fibered, $K(s)$ is fibered as well. The proof of (ii) and (iii) amounts to a series of verifications that are carried out without difficulty. 
\end{proof}

In the following statement we will examine when the fibered $E$-category $K(s)$ is a stack and when a gerbe.

\begin{theorem}\indent
\begin{enumerate}
\item[\rm (i)] If $F$ is a stack and $G$ is a prestack, then $K(s)$ is a stack.
\item[\rm (ii)] If, moreover, $m\colon F\to G$ is covering and conservative, then $K(s)$ is a gerbe.
\end{enumerate}
\end{theorem}
\begin{proof} An object of $K(s)$ is a pair $(a, \alpha)$ with $a\in{\rm Ob}(F_{S})$ and $\alpha\colon m(a)\isoto s(S)$ an $S$-isomorphism in $G$. This shows that the objects of $K(s)$ will glue together as soon as the objects in $F$ and the morphisms in $G$ do. Since the morphisms of $K(s)$ are defined by thiose of $F$, we see that the morphisms of $K(s)$ glue together as this is the case in $F$. Thus (i) is proved. If $m$ is conservative, then so also is $m\circ k(s)$.  It follows that $K(s)$ is a stack in groupoids. Since $m$ is covering, we can lift $s$ locally and it follows that $K(s)$ is a gerbe.
\end{proof}

\begin{definition}\indent
\begin{enumerate}
\item[\rm (i)] If $K(s)$ is a gerbe, it is called the {\it gerbe of liftings of $s$} relative to $m\colon F\to G$.
\item[\rm (ii)] The {\it category} of liftings of $s$: $\mathcal K_{s}(E)$.

$\mathcal K_{s}(E)$ is the category with objects the pairs $(t,\tau)$ where $t\colon E\to F$ is a cartesian section of $F$ and $\tau\colon m\circ t\isoto s$ is an isomorphism of cartesian sections. The description of the objects also shows what the morphisms are.
\end{enumerate}
\end{definition}

Now we have a canonical isomorphism from the category of cartesian sections of the gerbe $K(s)$ to the category of liftings of the section $s$
\[
\varprojlim(K(s)/E)\,\isoto\,\mathcal K_{s}(E), \sigma\mapsto (k(s)\circ\sigma,\e\kappa(s)\ast\sigma)
\]
 
{\bf Fact:} in order for the section $s$ of $G$ to be liftable to $F$ via $m$, it is necessary and sufficient that the gerbe $K(s)$ of liftings of $s$ be {\it trivial}.

\subsubsection{The band of the gerbe $K(s)$}\label{333}

\begin{definition} A sequence of bands
\[
1\to L\overset{u}{\to} M\overset{v}{\to} N\to 1
\]
is said to be {\it exact} if $u\colon L\to M$ is injective and normal and $v\colon M\to N$ is a cokernel of $u$.
\end{definition}

Equivalently, this sequence of bands is locally represented by an exact sequence of sheaves of groups.

\begin{proposition} Let $m\colon F\to G$ be a morphism of gerbes bound by an {\rm epimorphism} of bands $v\colon M\to N$. Let $s\colon E\to G$ be a cartesian section of $G$ and $u(s)\colon L (s)\to M$ the morphism of bands which binds the morphism of gerbes $k(s)\colon K(s)\to F$. Then we have
\begin{enumerate}
\item[\rm (i)] the sequence of bands
\[
1\to L(s)\overset{u(s)}{\to} M\overset{v}{\to} N\to 1
\]
is exact.
\item[\rm (ii)] the sheaf of groups $\uaut(s)$ of automorphisms of $s$ acts on $L(s)$.
\item[\rm (iii)] for every $\uaut(s)$-torsor $P$ we have a canonical isomorphism of bands $L({}^{P}s)\isoto {}^{P}L(s)$.  
\end{enumerate}
\end{proposition}
\begin{proof} \cite[IV, 2.5.5]{14}.
\end{proof}

\subsubsection{Application}

We can now apply to a particular situation. Let                  
\[
1\to A\overset{a}{\to} B\overset{b}{\to} C\to 1
\]  
be a short exact sequence of sheaves of groups on $E$ and let $P$ be a $C$-torsor. The morphism of gerbes
\[
{\rm TORS}(E,b)\colon {\rm  TORS}(E, B)\to{\rm TORS}(E, C)
\]
is bound by ${\rm lien}(b)\colon {\rm lien}(B)\to {\rm lien}(C)$ and this is an epimorphism of bands since $b$ is an epimorphism of sheaves of groups.

A $C$-torsor $P$ in fact amounts to a section of the gerbe ${\rm TORS}(E, C)$.  According to the previous proposition, the gerbe of liftings {\it $K(P)$ is bound by ${}^{P}{\rm lien}(A)$}. If $u(P)\colon {}^{P}{\rm lien}(A)\to {\rm lien}(B)$ is the morphism of bands that binds the morphism of gerbes
$k(P)\colon  K(P)\to {\rm TORS}(E, B)$, then the following sequence of bands is
{\it exact}
\[
1\to {}^{P}{\rm lien}(A)\overset{u(P)}{\to}{\rm lien}(B)
\overset{{\rm lien}(b)}{\to} {\rm lien}(C)\to 1
\]

\section{The 2-cohomology}

\subsection{Definition of the $H^{2}$}

Let $E$ be a site and let $L$ be a band on $E$.

$H^{\le 2}(E; L\e)$ is the set of classes of $L$-equivalent $L$-gerbes. Instead of $H^{\le 2}(E; L\e)$, we also use the shorter notation $H^{\le 2}(L\e)$.

A class is called {\it neutral} if one of its representatives has a section.  The set of neutral elements is denoted by $H^{\le 2}(L\e)^{\prime}$. A band $L$ is called {\it realisable} if $H^{\le 2}(L)$ is nonempty.

\subsubsection{$2$-cohomology with values in a sheaf of groups $A$}

Let $A$ be a sheaf of groups on $E$. Then we set
\[
H^{2}(E; A)=H^{2}(E; {\rm lien}(A)).
\]
We also use the shorter notation $H^{\le 2}(A\e)$.

Amongst the neutral elements of $H^{\le 2}(A\e)$ there exists a special element, namely the class of the gerbe ${\rm TORS}(E; A)$ of $A$-torsors on $E$. This particular neutral element will be called the {\it unit element}.

\subsubsection{Functoriality of the $H^{2}$}

Let $u\colon L\to M$ be a morphism of bands on $E$.
How can we associate a map $u^{(2)}\colon H^{\le 2}(L\e)\to H^{\le 2}(M\e)$?

For $H^{1}$ this was possible thanks to the operation of extension of the structural group.  One might wonder whether here one can obtain the desired map by an analogous construction of extension of the structural band.

Giraud showed that this will not be possible if the natural morphism $C_{M}\to C_{u}$ is not an isomorphism \cite[IV, 2.3.17]{14}.

In general we will see that no map can be defined from $H^{2}(L)$ to $H^{2}(M)$. As a result, Giraud was only able to associate to a morphism of bands $u\colon L\to M$ a {\it relation} from $H^{2}(L)$ to $H^{2}(M)$. Two elements $p\in H^{2}(L)$ and $q\in H^{2}(M)$ are related if there exist representatives $P$ of $p$ and $Q$ of $q$ as well as a $u$-morphism $P\to Q$. We will write $p\ggto{u} q$.

In this way a relation is defined between the sets $H^{2}(L)$ and $H^{2}(M)$
which we denote by
\[
H^{2}(L)\ggto{u^{(2)}}H^{2}(M)
\]

\begin{theorem} Let $u\colon L\to M$ be a morphism of bands on $E$ such that the natural morphism $C_{M}\to C_{u}$ is an {\rm isomorphism}. Then the relation
$u^{(2)}\colon H^{2}(L)\to H^{2}(M)$ is a {\rm function}.
\end{theorem}
\begin{proof} \cite[IV, 3.1.5.3]{14}. The proof is based on a number of extra properties associated with the gerbe ${\rm HOM}_{u}(P, Q)$ of $u$-bound morphisms. But from the proof itself one learns nothing regarding the actual reason why, in that particular case, a function can be defined from $H^{2}(L)$ to $H^{2}(M)$.
\end{proof}

\begin{remarks}\indent
\begin{enumerate}
\item[(a)] The function $u^{(2)}\colon H^{2}(L)\to H^{2}(M)$ maps neutral elements of $H^{2}(L)$ to neutral elements of $H^{2}(M)$.
\item[(b)] Let $L\overset{u}{\to}M \overset{v}{\to}N$ be morphisms of bands such that $u$, $v$ and $v\circ u$ satisfy the conditions of the previous theorem. Then we have
\[
v^{(2)}\circ u^{(2)}=(v\circ u)^{(2)}.
\]
\end{enumerate}
\end{remarks}

{\it Special cases where the condition  $C_{M}\isoto C_{u}$ is satisfied.}
\begin{enumerate}
\item[(i)] If $u\colon L\to M$ is an {\it epimorphism}, then $C_{M}\isoto C_{u}$. Indeed, since the bands form a stack, it suffices to prove that $C_{M}\isoto C_{u}$ {\it locally}. But by the representability this is reduced to the case of sheaves of groups, where it is obvious that the assertion is true.

\item[(ii)] If $M$ is an abelian band then both $C_{M}$ and $C_{u}$ are isomorphic to $M$. It follows that $C_{M}\isoto C_{u}$.
\end{enumerate}

\subsection{The second coboundary}

Consider the following short exact sequence of sheaves of groups on $E$:
\[
1\to A\overset{a}{\to} B\overset{b}{\to} C\to 1
\]  
How can we go from $H^{1}(C)$ to $H^{2}(A)$?

Let $p=[\e P\e]$ be an arbitrary element of $H^{1}(C)$. The gerbe $K(P)$ of liftings of the $C$-torsor $P$ to $B$ is bound by ${}^{P}{\rm lien}(A)$.
Consequently, if we were to proceed in the traditional way, i.e., map $p$ to the class that contains the obstruction to lifting a representative $P$ of $p$, we would generally not end up in $H^{2}(A)$. After all, $H^{2}(A)$ is the set of all classes of ${\rm lien}(A)$-equivalent gerbes. We would, indeed, end up in $H^{2}(A)$ if $B$ acted trivially on ${\rm lien}(A)$, since then ${}^{P}{\rm lien}(A)$ would be isomorphic to ${\rm lien}(A)$. From the above it appears
that, in general, $H^{2}(A)$ is not sufficient to include all the obstructions to lifting a torsor on $C$ to $B$. Giraud is therefore forced to introduce an ad-hoc set
\[
O(b)=N(b)/R.
\]
Here $N(b)$ is the set of triples $(K, L, u)$, where $K$ is a gerbe, $L$ is its band and $u\colon L\to{\rm lien}(B)$ a morphism of bands such that
\[
1\to L\overset{u}{\to} {\rm lien}(B)\overset{{\rm lien}(b)}{\to} {\rm lien}(C)\to 1
\] 
is exact. $R$ is an equivalence relation on the set $N(b)$.    
Two elements $(K, L, u)$ and $(K^{\e\prime}, L^{\e\prime}, u^{\e\prime})$ of $N(b)$ are related under $R$ if there exists a morphism of gerbes $m\colon K\to K^{\e\prime}$ such that the morphism $\alpha\colon L\to L^{\e\prime}$ that binds $m$ satisfies the condition $u^{\e\prime}\circ\alpha=u$.
A class $c\in O(b)$ is called {\it neutral} if there exists a representative $(K, L, u)$ of $c$ such that $K$ is trivial. The set of neutral classes is denoted by $O(b)^{\e\prime}$. The class of $({\rm TORS}(E;A),{\rm lien}(A), {\rm lien}(a))$ is called the unit class. By \ref{333}, we obtain a map of pointed sets
\[
d\colon H^{1}(C)\to O(b)
\]
which sends the class of the $C$-torsor $P$ to the class of $(K(P), L(P), u(P))$. Here $K(P)$ is the gerbe of liftings of $P$ with respect to ${\rm TORS}(E; b)$, $L(P)$ is the band of the gerbe $K(P)$ and $u(P)\colon L(P)\to{\rm lien}(B)$ is the morphism that binds $k(P)\colon K(P)\to {\rm TORS}(E;B)$.

On the other hand, we have a relation
\[
O(b)\ggto{a^{(2)}} H^{2}(B).
\]
An element $c\in O(b)$ is related to an element $q\in H^{2}(B)$ if there exist representatives $(K, L, u)$ of $c$ and $Q$ of $q$ as well as a ${\rm lien}(a)$-morphism $K\to Q$. We denote it by $c\ggto{a} q$.

The cohomology sequence
\[
1\to H^{0}(A)\to\dots\to H^{1}(B)\overset{b^{(1)}}{\to}H^{1}(C)\overset{d}{\to} O(b)\ggto{a^{(2)}} H^{2}(B)\overset{b^{(2)}}{\to}H^{2}(C)
\]
is {\it exact} in the following sense.
\begin{enumerate}
\item[\rm (i)] $x\in H^{1}(C)$ belongs to ${\rm Im}\, b^{(1)}$ if, and only if, $d (x)\in O(b)^{\prime}$.
\item[\rm (ii)] In order for $x\in O(b)$ to correspond to the unit class of $H^{2}(B)$, it is necessary and sufficient that $x$ belongs to the image of $d$.
\item[\rm (iii)] Let $x\in H^{2}(B)$.  For $b^{(2)}(x)$ to be neutral, it is necessary and sufficient that $x$ corresponds with a $y\in O(b)$ \cite[IV, 4.2.8]{14}.
\end{enumerate}

\chapter{$(A,\Pi\e)$-gerbes on a site}\label{bun}

The obstruction to the lifting of a torsor is subjected to a new study. The result leads through generalization to the concept of $(A, \Pi)$-gerbe.

Finally, we prove here the theorem from which the functoriality of the new $H^{2}$ will follow.

\section{Obstructions to the lifting of a torsor}
Let
\[
1\to A\overset{\!u}{\to}B\overset{\!v}{\to}C\to 1
\]
be a short exact sequence of sheaves of groups on the site $E$. For a given $C$-torsor $P$, we consider the gerbe $K(P)$ of liftings of $P$ to $B$. Then we have a morphism of gerbes on $E$
\[
\nu\colon K(P)\to{\rm TORS}(E;\uinn(B)),
\]
where $\uinn(B)$ denotes the sheaf of inner automorphisms of $B$. Indeed, let $(Q,\lambda)$ be an object of $K(P)_{S}$ with $S\in{\rm Ob}(E)$. Then we have a presheaf of sets on $E/S$
\[
\uisom_{\e Q}(B,\uaut_{\e S}(Q))\colon (E/S)^{\circ}\to({\rm Sets})
\]
This presheaf $\uisom_{\e Q}(B,\uaut_{\e S}(Q))$ associates with every object
$t\colon T\to S$ of $E/S$ the set of isomorphisms $\sigma_{q}\colon B^{\e t}\isoto
\uaut_{\e S}(Q)^{t}$ determined by the element $q\in Q(t)$. Here we indicate the action of $\sigma_{q}$ by the formula
\[
\sigma_{q}(b)(q\cdot b^{\le\prime})=(q\cdot b)\cdot b^{\le\prime}.
\]
$\uinn(B)$ acts freely and transitively on the presheaf $\uisom_{\e Q}(B,\uaut_{\e S}(Q))$ and the final morphism $\uisom_{\e Q}(B,\uaut_{\e S}(Q))\to e$ is covering since $Q$ is locally nonempty.
Let us work with the associated sheaf functor $a$. We find that $a(\uisom_{\e Q}(B,\uaut_{\e S}(Q)))$ is a {\it sheaf} on which $\uinn(B)$ acts freely and transitively and the final morphism $a(\uisom_{\e Q}(B,\uaut_{\e S}(Q)))\to e$ is an {\it epimorphism}. Thus $a(\uisom_{\e Q}(B,\uaut_{\e S}(Q)))$ is an {\it $\uinn(B)$-torsor}.

We therefore set
\[
\nu(Q,\lambda)=a(\uisom_{\e Q}(B,\uaut_{\e S}(Q))).
\]
The effect of $\nu$ on the morphisms is clear when one considers that $\inn(m)\circ\sigma_{q}=\sigma_{m(q)}$ for every morphism $m\colon (Q,\lambda)
\to (Q^{\e\prime}, \lambda^{\prime})$. We write $-\be {{{} \uinn(B)\atop \bigwedge}\atop }\lbe A\colon {\rm TORSC}(E;\uinn(B))\to {\rm FAGRSC}(E)$ for the cartesian $E$-functor that maps an $\uinn(B)$-torsor $P$ to the sheaf of groups $P\be {{{} \uinn(B)\atop \bigwedge}\atop }\lbe A$.

Let us now, for an arbitrary object $(Q,\lambda)$ in $K(P)_{S}$, examine more closely the group structure on $\nu(Q,\lambda)\be {{{} \uinn(B)\atop \bigwedge}\atop }\lbe A$.

Let $x,y$ be two elements of $(\nu(Q,\lambda)\be {{{} \uinn(B)\atop \bigwedge}\atop }\lbe A)(1_{S})$, where $S\in{\rm Ob}(E)$.
Since $\nu(Q,\lambda)$ is an $\uinn(B)$-torsor, we can always find an element $\sigma\in \nu(Q,\lambda)$ locally. This element determines an isomorphism of sheaves of sets $i_{\le\sigma}\colon A\isoto \nu(Q,\lambda)\be {{{} \uinn(B)\atop \bigwedge}\atop }\lbe A$ given by the formula $i_{\le\sigma}(a) = (\sigma, a)^{*}$. Here $(\sigma, a)^{*}$ is the image of $(\sigma, a)$ under the morphism
\[
\nu(Q,\lambda)\times A\to\nu(Q,\lambda)\be {{{} \uinn(B)\atop \bigwedge}\atop }\lbe A.
\]
The isomorphism of sheaves of sets $i_{\le\sigma}$ enables us to transfer the group structure on $A$ to $\nu(Q,\lambda)\be {{{} \uinn(B)\atop \bigwedge}\atop }\lbe A$. Associated to the elements $x$ and $y$, there exist unique elements $a_{x},a_{y}\in A$ such that  $i_{\le\sigma}(a_{x}) = (\sigma, a_{x})^{*}=x$ and $i_{\le\sigma}(a_{y}) = (\sigma, a_{y})^{*}=y$. Taken locally, the composition $x\cdot y$ is then equal to $(\sigma, a_{x}\le a_{y})^{*}$. The local group structure thus obtained on $\nu(Q,\lambda)\be {{{} \uinn(B)\atop \bigwedge}\atop }\lbe A$ is independent of the $\sigma\in \nu(Q,\lambda)$ used.
Consequently, all these local group structures determine a global group structure on $\nu(Q,\lambda)\be {{{} \uinn(B)\atop \bigwedge}\atop }\lbe A$. The composition $x\cdot y$ is then that element of  $(\nu(Q,\lambda)\be {{{} \uinn(B)\atop \bigwedge}\atop }\lbe A)(1_{S})$ that is locally determined
by $(\sigma, a_{x}a_{y})$. Furthermore, we have an isomorphism of morphisms of stacks
\[
k\colon {\rm Aut}(K(P))\isoto (-\be {{{} \uinn(B)\atop \bigwedge}\atop }\lbe A)\circ\nu,
\]
where ${\rm Aut}(K(P))$ was defined in \eqref{aut}.

Let $(Q,\lambda)$ be an object of $K(P)_{S}$. We have an $\uinn(B)$-morphism
\[
\uisom_{\e Q}(B,\uaut_{\e S}(Q))\to\uisom(A,\aut_{\e S}(Q,\lambda)),\, \sigma_{q}\mapsto \sigma_{q}\!\mid_{A}.
\]
This induces an $\uinn(B)$-morphism
\begin{equation}\label{usi}
\nu(Q,\lambda)\to \uisom(A,\aut_{\e S}(Q,\lambda)).
\end{equation}
We also have an $\uinn(B)$-morphism
\[
\uisom_{\e Q}(B,\uaut_{\e S}(Q))\to\uisom\big(A,\nu(Q,\lambda)\be {{{} \uinn(B)\atop \bigwedge}\atop }\lbe A\e\big),\,\sigma_{q}\mapsto i_{\le\sigma_{\be q}^{\le *}}
\]
where $\sigma_{\be q}^{\e *}$ is the image of $\sigma_{\be q}$ via $\uisom_{\e Q}(B,\uaut_{\e S}(Q))\to a(\uisom_{\e Q}(B,\uaut_{\e S}(Q)))$. Consequently, there is also again an induced $\uinn(B)$-morphism
\begin{equation}\label{tro}
\nu(Q,\lambda)\to \uisom\big(A,\nu(Q,\lambda)\be {{{} \uinn(B)\atop \bigwedge}\atop }\lbe A\e\big)
\end{equation}
From \eqref{usi} and \eqref{tro} we conclude that both $\uaut_{\e S}(Q,\lambda)$ and 
$\nu(Q,\lambda)\be {{{} \uinn(B)\atop \bigwedge}\atop }\lbe A$ play the role of the twist of $\nu(Q,\lambda)$ by $A$. Thus there exists a {\it natural isomorphism} of sheaves of groups
\[
k_{(Q,\lambda)}\colon \uaut_{\e S}(Q,\lambda)\isoto \nu(Q,\lambda)\be {{{} \uinn(B)\atop \bigwedge}\atop }\lbe A
\]
such that the following triangle commutes
\begin{equation}\label{tri}
\xymatrix{\nu(Q,\lambda)\ar[dr]_(.35){(3.1.2)}\ar[r]^(.35){(3.1.1)} &  \uisom(A,\aut_{\e S}(Q,\lambda))\ar[d]_(.35){\wr}^(.35){\uisom\big(A,\, k_{(Q,\lambda)}\big)}\\
&\uisom\big(A,\nu(Q,\lambda)\be {{{} \uinn(B)\atop \bigwedge}\atop }\lbe A\big).
}
\end{equation}
Finally, the following condition holds:

for every $S\in{\rm Ob}(E)$ and every $(Q,\lambda)\in {\rm Ob}(K(P)_{S})$ we have, for each element $x\in (\nu(Q,\lambda)\be {{{} \uinn(B)\atop \bigwedge}\atop }\lbe A)(1_{S})$ and every local representative $(\sigma,a)$ of $x$,
\[
\nu\big(k_{(Q,\lambda)}^{-1}(x)\big)(\sigma)=\sigma\circ\inn(a)
\]
Indeed, the element $x$ determines via $k_{(Q,\lambda)}^{-1}$ an $S$-automorphism $k_{(Q,\lambda)}^{-1}(x)\colon (Q,\lambda)\isoto (Q,\lambda)$. Applying the functor $\nu$ we obtain an associated $\uinn(B)$-automorphism 
\[
\nu\big(k_{(Q,\lambda)}^{-1}(x)\big)\colon \nu(Q,\lambda)\isoto \nu(Q,\lambda).
\]
From \eqref{tri} we conclude that
$k_{(Q,\lambda)}^{-1}(x)$ is locally equal to $\sigma_{\be q}(a)$, where $\sigma_{\be q}$ is a local representative of $\sigma$. Therefore $\nu(k_{(Q,\lambda)}^{-1}(x))(\sigma)$ is the element that is locally determined by $\inn(\sigma_{\lbe q}(a))\circ \sigma_{\lbe q}$.

Now observe that, for every $b\in B$, we have

\[
(\inn(\sigma_{\lbe  q}(a))\circ \sigma_{\lbe  q})(b)=\sigma_{\lbe  q}(a)\circ \sigma_{\lbe  q}(b)\circ \sigma_{\lbe  q}(a)^{-1}=\sigma_{\lbe  q}(aba^{-1})= (\sigma_{\lbe  q}\circ \inn(a))(b).
\]
Thus we have $\inn(\sigma_{\lbe q}(a))\circ \sigma_{\lbe q}=\sigma_{\lbe q}\circ \inn(a)$.

Since the element $\nu\big(k_{(Q,\lambda)}^{-1}(x)\big)(\sigma)$ is locally determined 
by $\inn(\sigma_{\lbe q}(a))\circ \sigma_{\lbe q}$ and $\sigma\circ \inn(a)$ is locally determined by $\sigma_{\lbe q}\circ \inn(a)$, we conclude that $\nu\big(k_{(Q,\lambda)}^{-1}(x)\big)(\sigma)$ and $\sigma\circ \inn(a)$ are locally equal. Since $\nu(Q,\lambda)$ is a sheaf we conclude that they are globally equal.

{\bf Conclusions}:

The obstruction to the lifting of a $C$-torsor to $B$ under $v\colon B\to C$ is a $3$-tuple $(K(P),\nu,k)$, where 

\begin{enumerate}
\item[\rm (i)] $K(P)$ is the gerbe of liftings of $P$ to $B$,
\item[\rm (ii)] $\nu\colon K(P)\to{\rm TORS}(E;\uinn(B))$ is a morphism of gerbes on $E$,
\item[\rm (iii)] $k\colon {\rm Aut}(K(P))\isoto (-\be {{{} \uinn(B)\atop \bigwedge}\atop }\lbe A)\circ\nu$ is an isomorphism of morphisms of stacks.
\end{enumerate}

The following condition must be satisfied:

Let $(Q, \lambda)$ be an arbitrary $S$-object of $K(P)_{S}$, $S\in{\rm Ob}(E)$.

For every element $x\in(\nu(Q,\lambda)\be {{{} \uinn(B)\atop \bigwedge}\atop }\lbe A)(1_{S})$ and every local representative $(\sigma,a)$ of $x$, we must have 
\[
\nu\big(k_{(Q,\lambda)}^{-1}(x)\big)(\sigma)=\sigma\circ\inn(a).
\]

\section{$(A,\Pi\e)$-gerbes on a site}

\subsection{Sheaves of crossed groups on a site}

\subsubsection{The category of sheaves of crossed groups}

Let $E$ be a site. By a {\it sheaf of crossed groups} on $E$ we mean a $4$-tuple
\[
(A,\e\rho,\e \Pi,\e\phi)
\]
where $A$ and $\Pi$ are sheaves of groups on $E$, $\rho\colon A\to \Pi$ is a morphism of sheaves of groups and $\phi\colon \Pi\times A\to A$ is a morphism of sheaves of sets that defines a left action of $\Pi$ on $A$.

We require that for every object $S$ of $E$ the $4$-tuple
\[
(A(S),\e\rho(S),\e \Pi(S), \phi(S))
\]
be a crossed group and that for every arrow $f\colon  T\to S$ in $E$ 
the restriction morphisms $A(S)\to A(T)$ and $\Pi(S)\to \Pi(T)$ define a morphism of crossed groups from $(A(S), \Pi(S))$ to $(A(T), \Pi(T))$.

\begin{remark} In the preceding definition we have used the shorter notation
$(A(S), \Pi(S))$ to denote the crossed group $(A(S), \rho(S), \Pi(S), \phi(S))$.
In what follows we will also use the shorter notation $(A,\Pi)$ to denote the sheaf of crossed groups $(A,\e\rho,\e \Pi,\e\phi)$.	
\end{remark}

Let $(A,\Pi)$ and $(B,\Sigma)$ be two sheaves of  crossed groups on the site $E$.
A {\it morphism of  sheaves of crossed groups} from $(A, \Pi)$ to $(B, \Sigma)$
is a pair of morphisms of sheaves of groups
\[
(\e f\colon A\to B, \varphi\colon \Pi\to\Sigma)
\]
such that for every object $S$ of $E$, the pair
\[
(\e f(S)\colon A(S)\to B(S),\, \varphi(S)\colon \Pi(S)\to \Sigma(S))
\]
is a morphism of crossed groups from $(A(S), \Pi(S))$ to $(B(S), \Sigma(S))$.

The composition of morphisms of sheaves of crossed groups is defined componentwise.

Thus we see that the sheaves of crossed groups on $E$ and the morphisms of sheaves of crossed groups form a {\it category}. We denote it by
\[
{\rm Cogr}(E).
\]

\subsubsection{The {\it stack} of sheaves of crossed groups on $E$}

The contravariant functor from $E$ to ${\rm Cat}$ which associates to every object $S$ of $E$ the category ${\rm Crogr}(E/S)$ and to every arrow $f\colon T\to S$ in $E$ the functor ${\rm Crogr}(E/S)\to{\rm Crogr}(E/T)$ defined via the composition with $E/f$, defines a split $E$-category
\[
{\rm CROGRSC}(E)
\]
\begin{theorem} ${\rm CROGRSC}(E)$ is a stack on $E$.	
\end{theorem}
\begin{proof}\indent
\begin{enumerate}
\item[\rm (i)] The morphisms glue together in ${\rm CROGRSC}(E)$. After all, 
a coherent family $\{(f_{i},\varphi_{\e i})\colon i\in I\}$ of morphisms of sheaves of crossed
groups gives rise to two coherent families of morphisms of 
sheaves of groups $\{f_{i}\colon i\in I\}$ and $\{\varphi_{\e i}\colon i\in I\}$. These families glue together in the stack ${\rm FAGRSC}(E)$ of sheaves of groups on $E$ and therefore determine a pair of morphisms of sheaves of groups $(f, \varphi)$. The morphism $(f, \varphi)\colon (A, \Pi)\to (B,\Sigma)$ will be a morphism of sheaves of crossed groups if the squares
\[
\xymatrix{A\ar[r]\ar[d]_{f}& \Pi\ar[d]^{\varphi}\\
B\ar[r]& \Sigma}
\]
and
\[
\xymatrix{\Pi\times A\ar[r]\ar[d]_{\varphi\times f}& A\ar[d]^{f}\\
\Sigma\times B\ar[r]& B}
\]

commute. This is so because the above holds locally.

\item[(ii)] The objects glue together in ${\rm CROGRSC}(E)$. Let 
\[
\{(A_{\e i},\rho_{i}, \Pi_{i},\varphi_{i}\e)\colon i\in I\e\}
\]
be a coherent family of sheaves of crossed groups. The families $\{A_{\e i}\colon i\in I\}$ and $\{\Pi_{i}\colon i\in I\}$ are coherent families of sheaves of groups and therefore can be glued together in ${\rm FAGRSC}(E)$.   Thus we obtain the sheaves of groups $A$ and $\Pi$.

The family $\{\rho_{\e i}\colon A_{\e i}\to \Pi_{\e i}\colon i\in I\}$ is a coherent family of morphisms of sheaves of groups that can be glued together in ${\rm FAGRSC}(E)$ to yield a morphism of sheaves of groups
\[
\rho\colon A\to \Pi.
\]   
Now the coherent family of morphisms of sheaves of sets $\{\varphi_{i}\colon \Pi_{\e i}\times A_{\e i}\to A_{\e i}\colon i\in I\}$ is a coherent family of morphisms of sheaves of groups that can be glued together in ${\rm FAISCIN}(E)$ to yield a morphism of sheaves of groups
\[
\varphi\colon \Pi\times A\to A
\]
This $\varphi$ defines a left action of $\Pi$ on $A$ because this is so locally.
That the 4-tuple $(A,\rho, \Pi, \varphi)$ satisfies the two axioms of a crossed group structure follows likewise from the fact that this is so locally.
\end{enumerate}
\end{proof}

\begin{definition} We call ${\rm CROGRSC}(E)$ the split $E$-stack of sheaves of crossed groups on $E$.	
\end{definition}

\subsection{$(A,\Pi)$-gerbes}

The previous analysis of the obstructions to lifting a $C$-torsor $P$ to $B$ via the morphism  $v\colon B\to C$ brings us to the following definition of $(A, \Pi)$-gerbe on $E$.

\begin{definition} Let $(A,\Pi)$ be a sheaf of crossed groups on the site $E$.
By an {\it $(A,\Pi)$-gerbe on $E$} we mean a $3$-tuple $(G,\mu, j)$, where $G$ is a gerbe on $E$, $\mu\colon G\to{\rm TORSC}(E;\Pi)$ is a morphism of gerbes on $E$ and $j\colon \aut(G)\to (-\be {{{} \Pi\atop \bigwedge}\atop }\lbe A)\circ\mu$ is an isomorphism of morphisms of stacks on $E$.

The following condition must be satisfied:

Let $x\in{\rm Ob}(G_{S})$, $S\in{\rm Ob}(E)$.

For every element $a_{x}\in (\mu(x)\be {{{} \Pi\atop \bigwedge}\atop }\lbe A)(1_{S})$ and every local representative $(\alpha, a)$ of $a_{x}$, we have
\[
\boxed{\mu(\, j^{-1}(a_{x}))(\alpha)=\alpha\cdot\rho(a)}
\]
\end{definition}

\begin{remark} It is in part thanks to this condition that, given a morphism of sheaves of crossed groups $(A,\Pi)\to(A^{\e\prime},\Pi^{\e\prime})$, we will be able to map an $(A, \Pi)$-gerbe $(G,\mu, j)$ to an $(A^{\e\prime},\Pi^{\e\prime})$-gerbe via the given morphism.
\end{remark}

\begin{definition}\label{222} Let $(f,\varphi)\colon (A,\Pi)\to(A^{\e\prime},\Pi^{\e\prime})$ be a morphism of sheaves of crossed groups, $(G,\mu, j)$ an $(A, \Pi)$-gerbe and $(G^{\e\prime},\mu^{\e\prime}, j^{\e\prime})$ an $(A^{\e\prime},\Pi^{\e\prime})$-gerbe on $E$.

By an {\it $(f,\varphi)$-morphism} 
\[
\lambda\colon (G,\mu, j)\to (G^{\e\prime},\mu^{\e\prime}, j^{\e\prime}\e)
\]
we mean a morphism of gerbes on $E$

\[
\lambda\colon G\to G^{\e\prime}
\]
such that the following two conditions hold:
\begin{enumerate}
\item[(i)] The square
\[
\xymatrix{G\ar[d]_(.4){\mu}\ar[rrr]^{\lambda}&&& G^{\e\prime}\ar[d]^(.4){\mu^{\prime}}\\
{\rm TORSC}(E;\Pi)\ar[rrr]^{\!{\rm TORSC}(E;\varphi)}&&& {\rm TORSC}(E;\Pi^{\le\prime}\e).}
\]	
commutes up to an isomorphism $i\colon {\rm TORSC}(E;\varphi)\circ\mu\isoto\mu^{\e\prime}\circ \lambda$ 
\item[(ii)] Also the following square must commute
\[
\xymatrix{\aut(G\e)\ar[d]_(.4){j}\ar[r]^(.45){\aut(\lambda)}& \aut(G^{\e\prime}\e)\circ\lambda\ar[d]^(.4){j^{\e\prime}\ast\e\lambda}\\
(-\be {{{} \Pi\atop \bigwedge}\atop }\lbe A)\circ\mu\ar[r]^(.4){\omega}& ((-\be {{{} \Pi^{\e\prime}\atop \bigwedge}\atop }\lbe A^{\e\prime}\e)\circ\mu^{\e\prime}\e)\circ\lambda.}
\]	
Here $\aut(\lambda)$ is the morphism of morphisms of stacks that at an arbitrary
object $x\in G_{S}$ is equal to the morphism $\uaut_{\e S}(x)\to\uaut_{\e S}(\lambda(x))$ defined by the functoriality of $\lambda$

To explain the effect of $\omega$ we first note that for an arbitrary object $x\in G_{S}$ we have an evident $\varphi$-morphism from $\mu(x)$ to $\mu(x)\be {{{} \Pi\atop \bigwedge}\atop }\lbe \Pi^{\e\prime}$ that we will denote by $p_{\e x}\colon \mu(x)\to\mu(x)\be {{{} \Pi\atop \bigwedge}\atop }\lbe \Pi^{\e\prime}$ \ref{321}

$\omega_{x}$ is then the following composition
\[
\xymatrix{\mu(x)\be {{{} \Pi\atop \bigwedge}\atop }\lbe A\ar[rr]^(.4){p_{\le x}\wedge f}
&&(\,\mu(x)\be {{{} \Pi\atop \bigwedge}\atop }\lbe \Pi^{\e\prime}\,)\be {{{} \Pi^{\e\prime}\atop \bigwedge}\atop }\lbe A^{\e\prime}\ar[rr]^{i_{x}\e\wedge\e A^{\prime}}&&\mu^{\e\prime}(\lambda(x))\be {{{} \Pi^{\prime}\atop \bigwedge}\atop }\lbe A^{\prime}.\\
}
\]
\end{enumerate}

\end{definition}

\subsection{Some theorems}

\begin{theorem} Let $(G,\mu, j)$ be an $(A,\Pi)$-gerbe and $\lambda\colon F\isoto G$ an $E$-equivalence of gerbes on $E$. Then $F$ has a unique canonical structure of $(A,\Pi)$-gerbe.	
\end{theorem}
\begin{proof} We define a morphism of gerbes $\mu^{*}\colon F\to {\rm TORSC}(E;\Pi)$ by setting $\mu^{*}=\mu\circ\lambda$ and an isomorphism of morphisms of stacks $j^{*}\colon \aut(F)\isoto (-\be {{{} \Pi\atop \bigwedge}\atop }\lbe A)\circ \mu^{*}$ by setting $j^{*}=j\circ\aut(\lambda)$. The 3-tuple $(F, \mu^{*} , j^{*})$ satisfies the condition from the definition of $(A,\Pi)$-gerbe. After all, if $x\in{\rm Ob}(F_{S})$ with $S\in{\rm Ob}(E)$, then we have $\mu^{*}(x)\be {{{} \Pi\atop \bigwedge}\atop }\lbe A=\mu(\lambda(x))\be {{{} \Pi\atop \bigwedge}\atop }\lbe A$ and the following triangle commutes
\[
\xymatrix{\uaut_{\e S}(x)\ar[dr]_{\aut(\lambda)_{x}}\ar[rr]^(.45){j_{x}^{*}} &&  \mu^{*}(x)\be {{{} \Pi\atop \bigwedge}\atop }\lbe A=\mu(\lambda(x))\be {{{} \Pi\atop \bigwedge}\atop }\lbe A\\
&\uaut_{\e S}(\lambda(x))\ar[ur]_{j_{\lambda(x)}}&.
}
\]
If $\Theta$ is an arbitrary element of $(\mu^{*}(x)\be {{{} \Pi\atop \bigwedge}\atop }\lbe A)(1_{S})$, then we have     
\[
\mu^{*}(x)((\,j_{x}^{*})^{-1}(\Theta))=\mu(\lambda(\, (j^{*}_{x})^{-1}(\Theta)))=\mu(\,j_{\lambda(x)}^{-1}(\Theta)).
\]
It follows that $(F, \mu^{*} , j^{*})$ satisfies the required condition and therefore $\mu^{*}$ and $j^{*}$ determine a structure of $(A,\Pi)$-gerbe on $F$. 
\end{proof}

\begin{definition}
The $(A,\Pi)$-gerbe structure on $F$ in the previous theorem is said to be {\it  induced by that on $G$ via $\lambda$}. 
\end{definition}

\begin{theorem} Let $(G,\mu, j)$ be an $(A,\Pi)$-gerbe on the site $E$ and  $\lambda\colon F\isoto G$ an $E$-equivalence of gerbes.  The morphism $\lambda$ is an ${\rm id}_{(A,\Pi)}$-morphism if we endow $F$ with the $(A,\Pi)$-gerbe structure induced by $\lambda$.	
\end{theorem}
\begin{proof} First of all, we must consider whether the following square
\[
\xymatrix{F\ar[d]_(.4){\mu^{*}}\ar[rrr]^{\lambda}&&& G\ar[d]^(.4){\mu}\\
{\rm TORSC}(E;\Pi)\ar[rrr]^{\!{\rm TORSC}(E;\e 1_{\Pi})}&&& {\rm TORSC}(E;\Pi^{\le\prime}\e).}
\]	
commutes up to an isomorphism. This is the case since, for every $x\in{\rm Ob}(G_{S})$, $S\in{\rm Ob}(E)$, we have a natural isomorphism of $\Pi$-torsors on $S$
\[
i_{x}\colon \mu^{*}(x)\be {{{} \Pi\atop \bigwedge}\atop }\lbe \Pi\isoto \mu^{*}(x).
\]
Namely, $i_{x}$ is the inverse of the $\Pi$-isomorphism $p_{\e x}\colon \mu^{*}(x)\isoto \mu^{*}(x)\be {{{} \Pi\atop \bigwedge}\atop }\lbe \Pi$\, \ref{pgg}.

Next, we need to show that the square
\[
\xymatrix{\aut(F)\ar[d]_(.4){j^{*}}\ar[r]^{\aut(\lambda)}& \aut(G)\circ\lambda\ar[d]^(.4){j\e\ast\e\lambda}\\
(-\be {{{} \Pi\atop \bigwedge}\atop }\lbe A)\circ\mu^{*}\ar[r]^(.4){\omega}& (-\be {{{} \Pi\atop \bigwedge}\atop }\lbe A\e)\circ\mu\circ\lambda}
\]	
commutes. This means that, for every object $x$ in $F_{S}$, we must show that $\omega_{x}\circ j_{x}^{\e *}=j_{\lambda(x)}\circ\aut(\lambda)_{x}$. This equality holds because $j_{\lambda(x)}\circ\aut(\lambda)_{x}=j_{x}^{*}$ and $\omega_{x}$ is the following composition
\[
\xymatrix{\mu^{*}(x)\be {{{} \Pi\atop \bigwedge}\atop }\lbe A\ar[rr]^(.4){p_{\le x}\e\wedge\e A}
&&(\,\mu^{*}(x)\be {{{} \Pi\atop \bigwedge}\atop }\lbe \Pi)\be {{{} \Pi\atop \bigwedge}\atop }\lbe A\ar[rr]^{i_{\le x}\e\wedge\e A}&&\mu(\lambda(x))\be {{{} \Pi\atop \bigwedge}\atop }\lbe A\\
}
\]
where $i_{x}=p_{x}^{-1}$ so that  $(i_{x}\e\wedge\e A)\circ(p_{x}\e\wedge\e A)={\rm id}$.

\end{proof}

\begin{theorem} Let $(f,\varphi)\colon (A,\Pi)\to (B,\Sigma)$ and $(g,\psi)\colon (B,\Sigma)\to (C,\Gamma)$ be morphisms of sheaves of crossed groups on the site $E$. If $\delta\colon (G,\mu,j)\to (F,\nu,k)$ is an $(f,\varphi)$-morphism and $\lambda\colon(F,\nu,k)\to (H,\omega,l)$ is a $(g,\psi)$-morphism, then $\lambda\circ\delta\colon (G,\mu,j)\to (H,\omega,l)$ is a $(g,\psi)\circ (f,\varphi)$-morphism.
\end{theorem}
\begin{proof} That $\delta$ be an $(f,\varphi)$-morphism means that, for every $x\in{\rm Ob}(G_{S})$, we have a natural isomorphism 
\[
i_{x}\colon \mu(x)\be {{{} \Pi\atop \bigwedge}\atop }\lbe \Sigma\isoto\nu(\delta(x))
\]
such that the following square commutes
\[
\xymatrix{\uaut_{\e S}(x)\ar[d]_{j_{\le x}}\ar[rr]^(.45){\aut(\delta)_{x}}&& \uaut_{\e S}(\delta(x))\ar[d]^(.45){k_{\le\delta(x)}}\\
\mu(x)\be {{{} \Pi\atop \bigwedge}\atop }\lbe A\ar[rr]^(.4){\eta_{\le x}}&& \nu(\delta(x)){{{} \Sigma\atop \bigwedge}\atop }\lbe B
}
\]	
Here $\eta_{\le x}$ is the following composition
\[
\xymatrix{\mu(x)\be {{{} \Pi\atop \bigwedge}\atop }\lbe A\ar[rr]^(.4){p_{\le x}\e\wedge\e f}
&&(\,\mu(x)\be {{{} \Pi\atop \bigwedge}\atop }\lbe \Sigma\,)\be {{{} \Sigma\atop \bigwedge}\atop }\lbe B\ar[rr]^{i_{\le x}\e\wedge\e B}&&\nu(\delta(x))\be {{{} \Sigma\atop \bigwedge}\atop }\lbe B\\
}
\]
and $p_{\le x}\colon \mu(x)\to \mu(x)\be {{{} \Pi\atop \bigwedge}\atop }\lbe \Sigma$ is the evident $\varphi$-morphism \ref{321}.

That $\lambda$ be an $(g,\psi)$-morphism means that, for every $x\in{\rm Ob}(F_{S})$, we have a natural isomorphism 
\[
m_{z}\colon \nu(z)\be {{{} \Sigma\atop \bigwedge}\atop }\lbe \Gamma\isoto\omega(\lambda(z))
\]
such that the following square commutes
\[
\xymatrix{\uaut_{\e S}(z)\ar[d]_(.45){k_{z}}^(.45){\wr}\ar[rr]^(.45){\aut(\lambda)_{z}}&& \uaut_{\e S}(\lambda(z))\ar[d]^(.45){l_{\lambda(z)}}_(.45){\wr}\\
\nu(z)\be {{{} \Sigma\atop \bigwedge}\atop }\lbe B\ar[rr]^(.45){\alpha_{z}}&& \omega(\lambda(z)){{{} \Gamma\atop \bigwedge}\atop }\lbe C
}
\]
Here $\alpha_{z}$ is the following composition
\[
\xymatrix{\nu(z)\be {{{} \Sigma\atop \bigwedge}\atop }\lbe B\ar[rr]^(.4){q_{\le z}\e\wedge\e g}
&&(\,\nu(z)\be {{{} \Sigma\atop \bigwedge}\atop }\lbe \Gamma\,)\be {{{} \Gamma\atop \bigwedge}\atop }\lbe C\ar[rr]^{m_{z}\e\wedge\e C}&&\omega(\lambda(z)){{{} \Gamma\atop \bigwedge}\atop }\lbe C\\
}
\]
and $q_{z}\colon \nu(z)\to \nu(z)\be {{{} \Sigma\atop \bigwedge}\atop }\lbe \Gamma$ is the evident $\psi$-morphism.

For every object $x$ in $G_{S}$, we have the following isomorphisms
\[
\xymatrix{\mu(x)\be {{{} \Pi\atop \bigwedge}\atop }\lbe \Gamma\ar[rr]^(.4){c_{\le x}}_{\sim}
&&(\,\mu(x)\be {{{} \Pi\atop \bigwedge}\atop }\lbe \Sigma\,)\be {{{} \Sigma\atop \bigwedge}\atop }\lbe\Gamma\ar[rr]^{i_{\le x}\e\wedge\e \Gamma}_{\sim}&&\nu(\delta(x)){{{} \Sigma\atop \bigwedge}\atop }\lbe \Gamma\ar[rr]^{m_{\delta(x)}}_{\sim}&&\omega((\lambda\circ\delta)(x))\\
}
\]
where $c_{x}$ exists by Propositions \ref{pgg} and \ref{assoc}. We now set $n_{x}=m_{\delta(x)}\circ(i_{x}\wedge \Gamma)\circ c_{x}$  and thus obtain a natural isomorphism 
\[
n_{x}\colon \mu(x)\be {{{} \Pi\atop \bigwedge}\atop }\lbe \Gamma\isoto \omega((\lambda\circ\delta)(x))
\]
It remains now to show that the following square commutes
\[
\xymatrix{\uaut_{\e S}(x)\ar[d]_(.45){j_{\le x}}^(.45){\wr}\ar[rr]^(.4){\aut(\lambda\e\circ\e\delta)_{x}}&& \uaut_{\e S}((\lambda\circ\delta)(x))\ar[d]^(.45){l_{(\lambda(\delta(x))}}
_(.45){\wr}\\
\mu(x)\be {{{} \Pi\atop \bigwedge}\atop }\lbe A\ar[rr]^(.4){\beta_{\le x}}&& \omega((\lambda\circ\delta\e)(x)){{{} \Gamma\atop \bigwedge}\atop }\lbe C
}
\]
Here $\beta_{x}$ is the following composition
\[
\xymatrix{\mu(x)\be {{{} \Pi\atop \bigwedge}\atop }\lbe A\ar[rr]^(.4){r_{\le x}\e\wedge\e g\le\circ f}
&&(\,\mu(x)\be {{{} \Pi\atop \bigwedge}\atop }\lbe \Gamma\,)\be {{{} \Gamma\atop \bigwedge}\atop }\lbe C\ar[rr]^{n_{\le x}\e\wedge\e C}&&\omega((\lambda\circ\delta\e)(x)){{{} \Gamma\atop \bigwedge}\atop }\lbe C\\
}
\]
and $r_{x}\colon \mu(x)\to \mu(x)\be {{{} \Pi\atop \bigwedge}\atop }\lbe \Gamma$ is the evident $\psi\circ\varphi$-morphism. This is also equal to the composition
\[
\xymatrix{\mu(x)\ar[rr]^(.4){\ref{pgg}}_(.4){\sim}
&&\mu(x)\be {{{} \Pi\atop \bigwedge}\atop }\lbe \Pi\ar[rr]^{\mu(x)\e\wedge \e\psi\e\circ\e\varphi}&&\mu(x)\be {{{} \Pi\atop \bigwedge}\atop }\lbe \Gamma.\\
}
\]
It follows that $r_{x}=(\mu(x)\wedge \e\psi\e)\circ p_{x}$ because $p_{x}$ is the composition
\[
\xymatrix{\mu(x)\ar[r]^(.4){\sim}
&\mu(x)\be {{{} \Pi\atop \bigwedge}\atop }\lbe \Pi\ar[rr]^{\mu(x)\e\wedge\e \varphi}&&\mu(x)\be {{{} \Pi\atop \bigwedge}\atop }\lbe \Sigma.
}
\]
We obtain from the given data the commutative squares
\[
\xymatrix{\uaut_{\e S}(x)\ar[d]^(.45){j_{\e x}} _(.45){\wr}\ar[rr]^(.45){\aut(\delta)_{x}}&& \uaut_{\e S}(\delta(x))\ar[d]^(.45){k_{\e\delta(x)}}_(.45){\wr}\ar[rr]^(.45){\aut(\lambda)_{\delta(x)}}&&\uaut_{\e S}(\lambda(\delta(x)))\ar[d]^(.43){l_{\lambda(\delta(x))}}
_(.4){\wr}\\
\mu(x)\be {{{} \Pi\atop \bigwedge}\atop }\lbe A\ar[rr]^(.4){\eta_{\e x}}&&
\nu(\delta(x))\be {{{} \Sigma\atop \bigwedge}\atop }\lbe B \ar[rr]^{\alpha_{\le\delta(x)}}&& \omega(\lambda(\delta(x))){{{} \Gamma\atop \bigwedge}\atop }\lbe C
}
\]
The only thing we still have to prove is that $\alpha_{\le\delta(x)}\circ\eta_{\le x}=\beta_{\le x}$, i.e.,
\[
(m_{\le\delta(x)}\wedge C\e)\circ(q_{\le\delta(x)}\wedge g\e)\circ (i_{\le x}\wedge B\e)\circ (p_{\le x}\wedge f\e)=(n_{\le x}\wedge C\e)\circ (r_{\le x}\wedge g\circ f\e).
\]
The right-hand side is equal to $((m_{\le\delta(x)}\circ(i_{x}\wedge\Gamma)\circ c_{x})\wedge C\e)\circ (r_{\le x}\wedge g\circ f\e)$ or even $(m_{\le\delta(x)}\wedge C\e)\circ(((i_{x}\wedge\Gamma\e)\circ c_{x}\e)\wedge C\e)\circ (r_{\le x}\wedge g\circ f\e)$.

Noting that $m_{\le\delta(x)}\wedge C$ is an isomorphism, it suffices to show that
\[
(q_{\le\delta(x)}\wedge g\e)\circ (i_{\le x}\wedge B\e)\circ (p_{\le x}\wedge f\e)=(((i_{x}\wedge\Gamma\e)\circ c_{x}\e)\wedge C\e)\circ (r_{\le x}\wedge g\circ f\e).
\]
Since $r_{\le x}\wedge g\circ f=((\mu(x)\wedge \psi\e)\wedge g\e)\circ (p_{\le x}\wedge f\e)$, it only remains to show that
\[
((i_{x}\wedge\Gamma\e)\wedge C\e)\circ(c_{x}\wedge C\e)\circ ((\mu(x)\wedge \psi\e)\wedge g\e)=(q_{\le\delta(x)}\wedge g\e)\circ(i_{x}\wedge B\e).
\]
This last verification is, in fact, the same as showing that the
following square commutes
\[
\xymatrix{\mu(x)\be{{{}\Pi\atop \bigwedge}\atop }\lbe \Sigma\ar[dd]_(.4){i_{\e x}}\ar[rrr]^(.45){\mu(x)\wedge\psi}&&&\mu(x)\be {{{} \Pi\atop \bigwedge}\atop }\lbe \Gamma\ar[d]^(.45){c_{\le x}}_(.45){\wr} \\
&&&(\mu(x)\be {{{} \Pi\atop \bigwedge}\atop }\lbe \Sigma)\be {{{} \Sigma\atop \bigwedge}\atop }\lbe \Gamma\ar[d]^{i_{\le x}\e\wedge\e\Gamma}\\
\nu(\delta(x))\ar[r]^(.45){d_{\le\delta(x)}}_(.43){\ref{pgg}}&\nu(\delta(x))\be {{{} \Sigma\atop \bigwedge}\atop }\lbe \Sigma\ar[rr]^(.5){\nu(\delta(x))\wedge\psi}&&
\nu(\delta(x))\be {{{} \Sigma\atop \bigwedge}\atop }\lbe \Gamma
}
\]
This is the case, as one can easily verify {\it locally} that $\beta_{\le x}=\alpha_{\delta(x)}\circ \eta_{\e x}$.
\end{proof}

\begin{theorem}\label{234} Let $(f,\varphi)\colon (A,\Pi\e)\to (B,\Sigma\e)$ be a morphism of sheaves of crossed groups on the site $E$ and $\lambda\colon (G,\mu,j)\to (F,\nu,k)$ an $(f,\varphi)$-morphism of gerbes on $E$. If $(f,\varphi)$ is an {\rm isomorphism} of sheaves of crossed groups, then $\lambda$ is an {\rm $E$-equivalence}.
\end{theorem}
\begin{proof} Regarding $\lambda$ as a morphism of gerbes, it suffices to check that $\lambda$ is fully faithful.

From the hypothesis we conclude that, for every $x\in{\rm Ob}(G_{S})$, we have a natural isomorphism
\[
i_{x}\colon \mu(x)\be{{{}\Pi\atop \bigwedge}\atop }\lbe \Sigma\isoto \nu(\lambda(x)).
\]
Further, we also have the commutative square
\[
\xymatrix{\uaut_{\e S}(x)\ar[d]_(.4){j_{\le x}}^(.4){\wr}\ar[rr]^(.45){\aut(\lambda)_{x}}&& \uaut_{\e S}(\lambda(x))\ar[d]^(.45){k_{\le\lambda(x)}}_(.45){\wr}\\
\mu(x)\be {{{} \Pi\atop \bigwedge}\atop }\lbe A\ar[rr]^(.45){\alpha_{\le x}}&& \nu(\lambda(x)){{{} \Sigma\atop \bigwedge}\atop }\lbe B
}
\]	
Here $\alpha_{x}$ is equal to the following composition
\[
\xymatrix{\mu(x)\be {{{} \Pi\atop \bigwedge}\atop }\lbe A\ar[rr]^(.4){p_{\le x}\e\wedge\e f}
&&(\,\mu(x)\be {{{} \Pi\atop \bigwedge}\atop }\lbe \Sigma\,)\be {{{} \Sigma\atop \bigwedge}\atop }\lbe B\ar[rr]^{i_{\le x}\e\wedge\e B}_{\sim}&&\nu(\lambda(x)){{{} \Sigma\atop \bigwedge}\atop }\lbe B\\
}
\]
and $p_{x}$ is equal to $\mu(x)\lra\mu(x)\be {{{} \Pi\atop \bigwedge}\atop }\lbe \Sigma$ (see Definition \ref{222}).

Since $\varphi\colon\Pi\to\Sigma$ is an isomorphism, $p_{x}$ also is an isomorphism and since $f\colon A\to B$ is an isomorphism, $p_{\le x}\wedge f$ is an isomorphism as well.  So $\alpha_{x}$ is an isomorphism. 
From the commutativity of the square it follows that
$\aut(\lambda)_{x}\colon \uaut_{\e S}(x)\to \uaut_{\e S}(\lambda(x))$ is an isomorphism of sheaves of groups on $E/S$. 
From this we can now further deduce that, for every pair of objects $x, y$ in $G_{S}$, the morphism
\[
\lambda(-)\colon\uisom_{\e S}(x,y)\to\uisom_{\e S}(\lambda(x),\lambda(y)), \,m \mapsto\lambda(m),
\]
is an isomorphism.

Since this is a morphism between sheaves of sets on $E/S$ and the sheaves of sets on a site form a {\it stack}, it suffices to check that this is locally an isomorphism.  Since $x$ and $y$ belong to a gerbe $G$, they are locally isomorphic. Let $m\colon x\isoto y$ be a local isomorphism. Using this, we have the following commutative square:
\[
\xymatrix{n\ar[d]&\uisom_{\e S}(x,y)\ar[d]^(.45){\wr}\ar[rr]^(.45){\lambda(-)}&& \uisom_{\e S}(\lambda(x),\lambda(y))&\lambda(m)\circ v\\
m^{-1}\circ n&\uaut_{\e S}(x)\ar[rr]^(.4){\aut(\lambda)_{x}}&& \uaut_{S}(\lambda(x))\ar[u]^(.45){\wr}&v\ar[u]
}
\]
Since we already know that $\aut(\lambda)_{x}\colon \uaut_{\e S}(x)\to \uaut_{\e S}(\lambda(x))$ is an isomorphism, from the commutativity we may conclude that $\lambda(-)\colon\uisom_{\e S}(x,y)\to\uisom_{\e S}(\lambda(x),\lambda(y))$ is an isomorphism.  We have shown that, fiber by fiber, $\lambda$ is fully faithful and this is sufficient to show that a cartesian $E$-functor is fully faithful.
\end{proof}

{\bf Special case}:

If $(G,\mu, j)$ and  $(F, \nu,  k)$ are two $(A, \Pi)$-gerbes and $\lambda\colon G\to F$ is an ${\rm id}_{(A,\Pi)}$-morphism, then $\lambda$ is an $E$-equivalence. We then say that $(G,\mu, j)$ and  $(F, \nu,  k)$ are {\it $(A, \Pi)$-equivalent gerbes} and we call $\lambda$ an {\it  $(A, \Pi)$-equivalence.}

\section{The functoriality theorem}

Given an $(A, \Pi)$-gerbe $(G,\mu, j)$ on the site $E$ and $(f,\varphi)\colon (A,\Pi)\to(A^{\e\prime},\Pi^{\e\prime})$ a morphism of sheaves of crossed groups on $E$, we want to show that we can construct out of $(G,\mu, j)$ an $(A^{\e\prime},\Pi^{\e\prime})$-gerbe $(G^{\e\prime},\mu^{\e\prime}, j^{\e\prime}\e)$. We may assume that $G$ is a {\it split}  $(A, \Pi)$-gerbe since Giraud has shown in \cite[Proposition 5.6]{13} that every fibered $E$-category is $E$-equivalent to a split $E$-category.

\subsection{Construction of the split $E$-prestack $G^{\e *}$}

We must construct a functor
\[
G^{\e *}\colon E^{\e\circ}\to ({\rm Cat})
\]
so that the presheaves of $S$-morphisms are sheaves.

\subsubsection{Effect of $G^{\e *}$ on the objects of $E$}
Let $S$ be an arbitrary object from $E$.

How can we define then the category $G^{\e *}(S)$?

\begin{itemize}
\item {\it Objects of $\e G^{\e *}(S)$}

We name as objects those of $G(S)$, i.e., ${\rm Ob}(G^{\e *}(S))={\rm Ob}(G(S))$.

\item {\it Morphisms of $\e G^{\e *}(S)$}

Let $x$ and $y$ be two objects of $G^{\e *}(S)$. In what follows we will use $\mu(x,A)$ as shorthand for $\mu(x)\be {{{} \Pi\atop \bigwedge}\atop }\lbe A$. 
The morphism $\varphi\colon\Pi\to\Pi^{\e\prime}$ enables us to consider $A^{\e\prime}$ as a left $\Pi$-object and consequently the contracted product 
$\mu(x)\be {{{} \Pi\atop \bigwedge}\atop }\lbe A^{\e\prime}$ can be formed. We denote the latter by $\mu(x,A^{\e\prime}\e)$.

The sheaf of groups $\uaut_{\e S}(x)$ acts on the right on the sheaf of sets $\uisom_{\e S}(x,y)$ by composition of morphisms. Thanks to the isomorphism $j_{\le x}\colon \uaut_{\e S}(x)\isoto \mu(x,A)$, $\uisom_{\e S}(x,y)$ is a right $\mu(x,A)$-object.  Because of the functoriality of the contracted product, we have a morphism $\mu(x,f)\colon \mu(x,A)\to \mu(x,A^{\e\prime}\e)$ which allows us to regard $\mu(x,A^{\e\prime}\e)$ as a left $\mu(x,A)$-object.
We define then
\[
\uisom_{\e S}^{*}(x,y)=\uisom_{\e S}(x,y)\be {{{} \mu(x,A)\atop \bigwedge}\atop }\lbe \mu(x,A^{\e\prime}\e)
\]

$\uisom_{\e S}^{*}(x,y)$ is a sheaf on $E/S$ and this will imply that we have a prestack $G^{\e*}$.

We define the {\it set of $S$-morphisms from $x$ to $y$} to be the set of sections of $\uisom_{\e S}^{*}(x,y)$.

We will denote these sets as $\isom_{\e S}^{*}(x,y)$ or sometimes as $\Hom_{\e S}^{*}(x,y)$.

\item {\it Composition of morphisms in  $\e G^{\e *}(S)$}

For every three objects $x, y$ and $z$, we have to define a map
\[
\isom_{\e S}^{*}(x,y)\times \isom_{\e S}^{*}(y,z)\to \isom_{\e S}^{*}(x,z)
\]
such that a number of axioms are satisfied.

When one considers that $\isom_{\e S}^{*}(x,y)$ is the set of sections of 
the sheaf $\uisom_{\e S}^{*}(x,y)$ and that the latter is the sheaf associated with the quotient presheaf $\uisom_{\e S}(x,y)\times  \mu(x,A^{\e\prime}\e)$  relative to the equivalence relation $\mathcal R_{\e xy}$ induced by the diagonal action of $\mu(x,A)$, then it is clear that it suffices to define a morphism
\[
\left(\frac{\uisom_{\e S}(x,y)\times  \mu(x,A^{\e\prime}\e)}{\mathcal R_{\e xy}}
\right)\times\left(\frac{\uisom_{\e S}(y,z)\times  \mu(y,A^{\e\prime}\e)}{\mathcal R_{\e yz}}\right)\to \frac{\uisom_{\e S}(x,z)\times  \mu(x,A^{\e\prime}\e)}{\mathcal R_{\e xz}}
\]
We define this morphism as follows:
\[
([(m,a^{\prime}_{x})],[(n,a^{\prime}_{y})])\mapsto [(m\circ n,\mu(m^{-1},A^{\e\prime}\e)(a^{\prime}_{y})\cdot a^{\prime}_{x})]
\]
Before moving on to the next point, where we will verify that all these data satisfy the axioms of a category, we must first show that this definition is {\it independent of the representatives used}. To this end, we take another set of representatives, namely.
\[
(m\circ j_{x}^{-1}(a_{x}),\mu(x,f)(a_{x}^{-1})\cdot a_{x}^{\e\prime}\e)\text{ and }
(n\circ j_{y}^{-1}(a_{y}),\mu(y,f)(a_{y}^{-1})\cdot a_{y}^{\e\prime}\e)
\]
According to our composition rule, we would obtain:
\[
(n\circ j_{y}^{-1}(a_{y})\circ m\circ j_{x}^{-1}(a_{x}),\mu(j_{x}^{-1}(a_{x}^{-1})\circ m^{-1},A^{\e\prime}\e)(\e\mu(y,f\e)(a_{y}^{-1})\cdot a_{y}^{\prime})\cdot\mu(x,f)(a_{x}^{-1})\cdot a_{x}^{\prime})
\]
Is this now equivalent to $(n\circ m,\mu(m^{-1},A^{\prime}\e)(a_{y}^{\prime})\cdot a_{x}^{\e\prime}\e)$?

First we will rewrite the first component of the new pair obtained by composition of the new representatives.  Since $j\colon \aut(G)\to (-\be {{{} \Pi\atop \bigwedge}\atop }\lbe A\e)\circ \mu$ is a morphism
of morphisms of stacks, we have the following commutative square:
\[
\xymatrix{\uaut_{\e S}(x)\ar[rr]^(.45){j_{\le x}}_(.45){\sim}&& \mu(x,A\e)\\
\uaut_{\e S}(y)\ar[u]^{\uinn(m^{-1})}\ar[rr]^(.45){j_{\le y}}_(.45){\sim}&&
\mu(y,A\e)\ar[u]_{\mu(m^{-1}, A)}
}
\]
Here we deduce
\[
(\, j_{x}\circ \uinn(m^{-1})\circ j_{y}^{-1})(a_{y})=\mu(m^{-1}, A)(a_{y})
\]
and thus
\[
m^{-1}\circ j_{y}^{-1}(a_{y})\circ m=j_{x}^{-1}(\e\mu(m^{-1}, A)(a_{y}))
\]
or even
\[
j_{y}^{-1}(a_{y})\circ m=m\circ j_{x}^{-1}(\e\mu(m^{-1}\be,\e A)(a_{y})).
\]
The first component $n\circ j_{y}^{-1}(a_{y})\circ m\circ j_{x}^{-1}(a_{x})$ is equal to
\[
n\circ m\circ j_{x}^{-1}(\e\mu(m^{-1}, A)(a_{y})\cdot a_{x}).
\]

Now we would still have to check that the second component is equal to                                                       
\[
\mu(x, f)(a_{x}^{-1}\cdot\mu(m^{-1}, A)(a_{y}^{-1}))\cdot \mu(m^{-1}, A^{\e\prime}\e)(a_{y}^{\e\prime})\cdot a_{x}^{\e\prime}.
\]
Is $\mu(j_{x}^{-1}(a_{x}^{-1})\circ m^{-1}, A^{\e\prime})(\mu(y,f)(a_{y}^{-1})\cdot a_{y}^{\e\prime})\cdot\mu(x,f)(a_{x}^{-1})$ equal to
\[
\mu(x,f)(a_{x}^{-1}\cdot \mu(m^{-1}, A)(a_{y}^{-1}))\cdot\mu(m^{-1}, A^{\e\prime}\e)(a_{y}^{\e\prime})?
\]
Since $\mu(x,f)$ is a morphism  of sheaves of groups the last element is
\[
\mu(x,f)(a_{x}^{-1})\cdot \mu(x,f)\circ \mu(m^{-1},A)(a_{y}^{-1})\cdot\mu(m^{-1},A^{\e\prime}\e)(a_{y}^{\e\prime}\e).
\]
Since $\mu(y,f)\circ \mu(m,A)=\mu(m,A^{\e\prime}\e)\circ \mu(x,f)$ this element equals
\[
\mu(x,f)(a_{x}^{-1})\cdot \mu(m^{-1},A^{\e\prime}\e)\circ \mu(y,f)(a_{y}^{-1})\cdot \mu(m^{-1},A^{\e\prime}\e)(a_{y}^{\e\prime}).
\]
Further, we have $\mu(\e j_{x}^{-1}(a_{x}^{-1})\circ m^{-1},A^{\e\prime}\e)(\e\mu(y,f)(a_{y}^{-1})\cdot a_{y}^{\e\prime}\e)\cdot\mu(x,f)(a_{x}^{-1})$ is equal to
\[
\mu(\e j_{x}^{-1}(a_{x}^{-1}),A^{\e\prime}\e)(\e\mu(m^{-1},A^{\e\prime}\e)\circ\mu(y,f)(a_{y}^{-1})\cdot\mu(m^{-1},A^{\e\prime}\e)(a_{y}^{\e\prime}\e)\cdot\mu(x,f)(a_{x}^{-1}).
\]
If we could still prove that
\[
\mu(\e j_{x}^{-1}(a_{x}^{-1}),A^{\e\prime}\e)=\inn(\e\mu(x,f)(a_{x}^{-1}))
\]
then we would be done.

Let $\Theta$ be any element of $\mu(x,A^{\e\prime}\e)=\mu(x)\be {{{} \Pi\atop \bigwedge}\atop }\lbe A^{\e\prime}$.

Is then $\mu(\e j_{x}^{-1}(a_{x}^{-1}),A^{\e\prime}\e)(\Theta)=\inn(\e\mu(x,f\e)(a_{x}^{-1}))(\Theta)$?

It suffices to show that both members are equal locally since $\mu(x,A^{\e\prime}\e)$ is a sheaf.  But locally $a_{x}$ is determined by a class $[(\alpha, a)]$ and $\Theta$ by a class $[(\alpha, a^{\e\prime}\e)]$ so
that $\mu(\e j_{x}^{-1}(a_{x}^{-1}),A^{\e\prime}\e)(\Theta)$ is locally determined by a class $[(\mu(\e j_{x}^{-1}(a_{x}^{-1})(\alpha), a^{\e\prime}\e)]$. Since
$G$ is an $(A,\Pi)$-gerbe, we have $[(\e\mu(\e j_{x}^{-1}(a_{x}^{-1})(\alpha), a^{\e\prime}\e)]=[(\alpha\cdot\rho(a^{-1}), a^{\e\prime}\e)]$.

On the other hand, $\inn(\e\mu(x,f)(a_{x}^{-1}))(\Theta)$ is locally determined by
\[
\begin{array}{rcl}
\inn([(\alpha,f(a^{-1}))])([(\alpha,a^{\e\prime}\e)])&=&[(\alpha,f(a^{-1}))]
\cdot [(\alpha,a^{\e\prime}\e)]\cdot [(\alpha,f(a))]\\
&=& [(\alpha,f(a^{-1})\cdot a^{\e\prime}\cdot f(a))]\\
&=&[(\alpha,\rho^{\e\prime}(f(a^{-1}))\cdot a^{\e\prime}\e)].
\end{array}
\]
Because of the diagonal action of $\Pi$ on $\mu(x)\times A^{\e\prime}$, $[(\alpha\cdot\rho(a^{-1}),a^{\e\prime})]$ is equal to $[(\alpha,\varphi(\rho(a^{-1}))\cdot a^{\e\prime}\e)]$ and since $(f,\varphi)$ is a morphism of sheaves of crossed groups the latter is equal to
$[(\alpha,\rho^{\e\prime}(f(a^{-1}))\cdot a^{\e\prime}\e)]$.

Thus $\mu(\e j_{x}^{-1}(a_{x}^{-1}),A^{\e\prime}\e)(\Theta)$ and $\inn(\e\mu(x,f)(a_{x}^{-1}))(\Theta)$ are equal locally and therefore  globally.

Hence we have shown that $\mu(\e j_{x}^{-1}(a_{x}^{-1}),A^{\e\prime}\e)=\inn(\mu(x,f)(a_{x}^{-1}))$ and also at once that our composition rule is independent of the chosen representatives.

\item Are the two category axioms satisfied?

\begin{enumerate}
\item[(i)] Associativity.

We must verify that the following equality holds:
\[
([(m, a_{x}^{\e\prime}\e)]\cdot [(n, a_{y}^{\e\prime}\e)])\cdot[(k, a_{z}^{\e\prime}\e)]=[(m, a_{x}^{\e\prime}\e)]\cdot([(n, a_{y}^{\e\prime}\e)]\cdot[(k, a_{z}^{\e\prime}\e)])
\]
\[
\begin{array}{rcl}
([(m, a_{x}^{\e\prime}\e)]\cdot [(n, a_{y}^{\e\prime}\e)])\cdot[(k, a_{z}^{\e\prime}\e)]&=&[(n\circ m, \mu(m^{-1}, A^{\e\prime}\e)(a_{y}^{\e\prime})\cdot a_{x}^{\e\prime})]\cdot [(k, a_{z}^{\e\prime}\e)]\\
&=&[(k\circ (n\circ m), \mu(m^{-1}\circ n^{-1}, A^{\e\prime}\e)(a_{z}^{\e\prime})\cdot \mu(m^{-1},A^{\e\prime})(a_{y}^{\e\prime})\cdot a_{x}^{\e\prime})]
\end{array}
\]
\[
\begin{array}{rcl}
[(m, a_{x}^{\e\prime}\e)]\cdot([(n, a_{y}^{\e\prime}\e)]\cdot[(k, a_{z}^{\e\prime}\e)])&=&[(m, a_{x}^{\e\prime}\e)]\cdot[(k\circ n,\mu(n^{-1},
A^{\e\prime})(a_{z}^{\e\prime}\e)\cdot a_{y}^{\e\prime}\e)]\\
&=&[((k\circ n)\circ m,\mu(m^{-1}, A^{\e\prime})(\mu(n^{-1},
A^{\e\prime})(a_{z}^{\e\prime}\e)\cdot a_{y}^{\e\prime}\e)\cdot a_{x}^{\e\prime})]\\
&=&[((k\circ n)\circ m,\mu(m^{-1}\circ n^{-1}, A^{\e\prime})(a_{z}^{\e\prime}\e)\cdot \mu(m^{-1},A^{\e\prime}\e)(a_{y}^{\e\prime}\e)\cdot a_{x}^{\e\prime})].
\end{array}
\]
Since $k\circ (n\circ m)=(k\circ n)\circ m$, the required equality holds.
\item[(ii)] Identity condition.

For every object $x$ of $G^{\e *}(S)$, $[({\rm id}_{x},1)]$ determines the identity morphism of $x$ in $G^{\e *}(S)$ since
\[
\begin{array}{rcl}
[({\rm id}_{x},1)]\cdot [(n, a_{y}^{\e\prime}\e)]&=&[(n\circ {\rm id}_{x}, \mu({\rm id}_{x}, A^{\e\prime}\e)(a_{y}^{\e\prime})\cdot 1)]\\
&=&[(n, a_{y}^{\e\prime}\e)]
\end{array}
\]
\[
\begin{array}{rcl}
[(m, a_{x}^{\e\prime}\e)]\cdot[({\rm id}_{y},1)]&=&[({\rm id}_{y}\circ m, \mu(m^{-1}, A^{\e\prime}\e)(1)\cdot a_{x}^{\e\prime})]\\
&=&[(m, a_{x}^{\e\prime}\e)]
\end{array}
\]
\end{enumerate}

\end{itemize}
  
\subsubsection{Effect of $G^{\e*}$ on the morphisms of $E$}

For every morphism $f\colon T\to S$ in $E$ we must define a functor $G^{\e*}(f\le)\colon G^{\e*}(S\le)\to G^{\e*}(T\le)$.

\begin{itemize}
\item Effect of $G^{\e*}(f\le)$ on the objects.

Let $x$ be an object of $G^{\e*}(S\le)$, i.e., an object of the fiber
$G_{S}$.

Then we define
\[
G^{\e*}\lbe(\e f\e)(x)=x^{f},
\]
where $x^{f}$ is the inverse $f$-image of $x$ with respect to the normalized splitting of $G$.

\item Effect of $G^{\e*}(f\le)$ on the morphisms.

For every pair of objects $x$ and $y$ in $G^{\e*}(S\le)$, we must define a map
from $\Hom_{\e S}^{*}(x,y)$ to $\Hom_{\e T}^{*}(G^{\e*}(f\e)(x),G^{\e*}(f\e)(y))$ or from $\uhom_{\e S}^{*}(x,y)(1_{S})$ to $\uhom_{\e T}^{*}(x^{f},y^{f})(1_{T})$. Here
\[
\uhom_{\e S}^{*}(x,y)=\uisom_{\e S}(x,y)\be {{{} \mu(x,A)\atop \bigwedge}\atop }\lbe \mu(x,A^{\e\prime}\e)
\]
and
\[
\uhom_{\e T}^{*}(x^{f},y^{f})=\uisom_{\e T}(x^{f},y^{f})\be {{{} \mu(x^{f},A)\atop \bigwedge}\atop }\lbe \mu(x^{f},A^{\e\prime}\e).
\]
For every arrow $f\colon T\to S$ in $E$ we have
\[
\uisom_{\e S}(x,y)^{f}=\uisom_{\e T}(x^{f},y^{f})
\]
where $\uisom_{\e S}(x,y)^{f}$ is the inverse $f$-image of $\uisom_{\e S}(x,y)$
with respect to the splitting of ${\rm FAISCIN}(E)$.

Since $\mu$ respects inverse images we have, by \cite[III, 1.3.1.1]{14},
\[
\mu(x^{f})\be {{{} \Pi\atop \bigwedge}\atop }\lbe A\isoto (\e\mu(x^{f})\be {{{} \Pi\atop \bigwedge}\atop }\lbe A)^{f}\quad\text{ and }\quad \mu(x^{f})\be {{{} \Pi\atop \bigwedge}\atop }\lbe A^{\e\prime}\isoto (\e\mu(x)\be {{{} \Pi\atop \bigwedge}\atop }\lbe A^{\e\prime}\e)^{f}.
\]
From these considerations and also from the functoriality of  
$\uisom_{\e S}(x,y)\be {{{} \mu(x,A)\atop \bigwedge}\atop }\lbe \mu(x,A^{\e\prime}\e)$ we obtain the following morphisms:
\[
(\e\uisom_{\e S}(x,y)\be {{{} \mu(x,A)\atop \bigwedge}\atop }\lbe \mu(x,A^{\e\prime}\e))(1_{S})\overset{(1)}{\lra}(\e\uisom_{\e S}(x,y)\be {{{} \mu(x,A)\atop \bigwedge}\atop }\lbe \mu(x,A^{\e\prime}\e))(f\e)
\]
and
\[
(\e\uisom_{\e S}(x,y)\be {{{} \mu(x,A)\atop \bigwedge}\atop }\lbe \mu(x,A^{\e\prime}\e))^{f}(1_{T})\overset{(2)}{\lra}(\e\uisom_{\e T}(x^{f},y^{f})\be {{{} \mu(x^{f},A)\atop \bigwedge}\atop }\lbe \mu(x^{f},A^{\e\prime}\e))(1_{T}\e).
\]
We define the map from  $\Hom_{\e S}^{*}(x,y)$ to $\Hom_{\e T}^{*}(G^{\e*}(f\e)(x),G^{\e*}(f\e)(y))$ to be the composition of (1) and (2).
Hereby is the construction of the split $E$-prestack $G^{\e*}$ completed.
\end{itemize}

Next we have the following morphism of split $E$-categories
\[
\lambda^{*}\colon G\to G^{\e*}.
\]
For every object $x\in{\rm Ob}(G_{S})$, we have $\lambda^{*}(x)=x$ and the image of an $S$-morphism $m\colon x\isoto y$ is the element that is locally determined by $(m,1)$.

\subsection{The $(A^{\e\prime},\Pi^{\e\prime})$-gerbe $G^{\e\prime}$}\label{32}

We now move on to the stack associated with $G^{\e*}$.  According to \cite[II, 2.1.3]{14}, we have a {\it bicovering} morphism
\[
G^{\e*}\overset{\!a}{\to} A(G^{\e*}\e).
\]
Since $G^{\e*}$ is already a prestack we know that $a$ is fully faithfull and locally surjective on the objects \cite[II, 1.4.5]{14}.

We define then
\[
G^{\e\prime}=A(G^{\e*}\e).
\]
By composing $a$ with $\lambda^{*}$  we obtain a morphism of split $E$-categories
\[
\lambda\colon G\to G^{\e\prime}.
\]
\begin{itemize}
\item {\it $G^{\e\prime}$ is a gerbe on $E$.}

It is immediate from its construction that $G^{\e\prime}$ is a stack in groupoids on $E$. It remains to check two conditions:
\begin{enumerate}
\item Does there exist a refinement $\mathcal R$ of $E$ such that for every $S\in \mathcal R$ we have ${\rm Ob}(G^{\e\prime}_{\lbe S})\neq\emptyset$?
\item is it the case that every two objects $x^{\e\prime},y^{\e\prime}$ in $G^{\e\prime}_{S}$ are locally isomorphic?
\end{enumerate}

The first condition holds since $G$ is a gerbe and we have the morphism 
$\lambda\colon G\to G^{\e\prime}$. The second condition also holds since the $E$-functor $\lambda$ is locally surjective on the objects and $G$ itself is a gerbe.

\item {\it The structure of $(A^{\e\prime},\Pi^{\e\prime})$-gerbe on $G^{\e\prime}$.}

Do we have a cartesian $E$-functor
\[
\mu^{\e\prime}\colon G^{\e\prime}\to{\rm TORSC}(E;\Pi^{\e\prime}\e)
\]
as well as an isomorphism of morphisms of stacks
\[
j^{\e\prime}\colon \aut(G^{\e\prime}\e)\isoto (-\be {{{} \Pi^{\e\prime}\atop \bigwedge}\atop }\lbe A^{\e\prime}\e)\circ\mu^{\e\prime}
\]
so that the condition in the definition of $(A^{\e\prime},\Pi^{\e\prime})$-gerbe is satisfied?
Since $a$ is bicovering, it suffices to show that there exists a cartesian $E$-functor
\[
\mu^{\e\prime}\colon G^{\e\prime}\to{\rm TORSC}(E;\Pi^{\e\prime}\e)
\]
so that, for every $x\in{\rm Ob}(G_{S}^{\e*}), S\in{\rm Ob}(E)$, a natural isomorphism
\[
j_{x}^{*}\colon\uaut_{\e S}^{*}(x)\isoto\mu^{*}(x)\be {{{} \Pi^{\e\prime}\atop \bigwedge}\atop }\lbe A^{\e\prime}
\]
can be constructed.

We set $\mu^{*}(x)= \mu(x)\be {{{} \Pi\atop \bigwedge}\atop }\lbe \Pi^{\e\prime}$ for every $x\in{\rm Ob}(G_{S}^{\e*})$ and every $S\in{\rm Ob}(E)$. We see that
\[
\mu^{*}(x)\be {{{} \Pi^{\e\prime}\atop \bigwedge}\atop }\lbe A^{\e\prime}=
(\,\mu(x)\be {{{} \Pi\atop \bigwedge}\atop }\lbe \Pi^{\e\prime})\be {{{} \Pi^{\e\prime}\atop \bigwedge}\atop }\lbe A^{\e\prime}\simeq \mu(x)\be {{{} \Pi\atop \bigwedge}\atop }\lbe A^{\e\prime}.
\]
On the other hand we have, by definition,
\[
\begin{array}{rcl}
\uaut_{\e S}^{*}(x)&=&\uaut_{\e S}(x)\be {{{} \mu(x,A)\atop \bigwedge}\atop }\lbe\mu(x, A^{\e\prime}\e)\\
&\simeq& \mu(x,A)\be {{{} \mu(x,A)\atop \bigwedge}\atop }\lbe\mu(x, A^{\e\prime}\e)\quad\text{since}\, \uaut_{\e S}(x)\underset{j_{\le x}}{\isoto}\mu(x,A)\\
&\simeq&\mu(x)\be {{{} \Pi\atop \bigwedge}\atop }\lbe A^{\e\prime}.
\end{array}
\]
In this way we see that a natural isomorphism
\[
j_{x}^{*}\colon\uaut_{\e S}^{*}(x)\isoto\mu^{*}(x)\be {{{} \Pi^{\e\prime}\atop \bigwedge}\atop }\lbe A^{\e\prime}
\]
is established.

Now we look more closely at the effect of $\mu^{*}$ on the morphisms. Let $m^{*}\colon x\to y$ be an element of $\Hom_{\e S}^{*}(x,y)$.

Since $\Hom_{\e S}^{*}(x,y)$ is the set of sections of the sheaf
\[
\uisom_{\e S}(x,y)\be {{{} \mu(x,A)\atop \bigwedge}\atop }\lbe\mu(x, A^{\e\prime}\e),
\]
we see that $m^{*}$ is locally represented by a pair $(m,a_{x}^{\e\prime})\in
\uisom_{\e S}(x,y)\times\mu(x, A^{\e\prime}\e)$. The element $a_{x}^{\e\prime}\in \mu(x)\be {{{} \Pi\atop \bigwedge}\atop }\lbe A^{\e\prime}$ is locally represented by a pair $(\alpha,a^{\e\prime})\in\mu(x)\times A^{\e\prime}$.

Let $\Theta$ be any element of $(\,\mu(x)\be {{{} \Pi\atop \bigwedge}\atop }\lbe \Pi^{\e\prime})(1_{S})$. This element is locally determined by a class $[(\alpha,p^{\e\prime}\e)]$.

$\mu^{*}(m^{*})(\Theta)$ is then the element that is locally determined by the class $[(\mu(m)(\alpha),\rho^{\e\prime}(a^{\e\prime})\cdot p^{\e\prime}\e)]$. Here we note that this construction is independent of the chosen representative 
$(m,a_{x}^{\e\prime})$ of $m^{*}$ and the representative 
$(\alpha,a^{\e\prime})$ of $a^{\e\prime}_{x}$.

\item {\it Is the condition fulfilled?}

Let $a^{\e\prime}_{x}$ be an element of $(\,\mu(x)\be {{{} \Pi\atop \bigwedge}\atop }\lbe A^{\e\prime})(1_{S})$ and $(\alpha,a^{\e\prime}\e)$ a representative of this element. Via the isomorphism
\[
\mu(x)\lbe {{{} \Pi\atop \bigwedge}\atop }\lbe A^{\e\prime}\isoto \mu^{*}(x)\be {{{} \Pi\atop \bigwedge}\atop }\lbe A^{\e\prime},
\]
the element $a^{\e\prime}_{x}$ lies over an element $a^{*}_{x}$ that is locally represented by $([(\alpha,1)],a^{\prime}\e)$. We should now be able to prove that
\begin{equation}\label{moe}
\mu^{*}((\, j_{x}^{*})^{-1}(a^{*}_{x}))((\alpha,1)^{*})=(\alpha,1)^{*}\times \rho^{\e\prime}(a^{\e\prime})
\end{equation}
Here $(\alpha,1)^{*}$ stands for the element that is locally determined by the class $[(\alpha,1)]$.

From the way $j_{x}^{*}$ was defined, it follows that
$(\,j_{x}^{*})^{-1}(a^{*}_{x}))$ is locally determined by $[({\rm id}_{x},a^{\e\prime}_{x})]$. Thus $\mu^{*}((\, j_{x}^{*})^{-1}(a^{*}_{x}))((\alpha,1)^{*})$ is the
element that is locally determined by $[(\e\mu({\rm id}_{x})(\alpha),\rho^{\e\prime}(a^{\e\prime})\cdot 1)]$. The latter is equal to the class $[(\alpha,\rho^{\e\prime}(a^{\e\prime}\e))]$ or even $[(\alpha,1)]\times \rho^{\e\prime}(a^{\e\prime}\e)$.
Thus we have shown that both members of \eqref{moe} are equal locally and this
suffices.  Let now $\mu^{\e\prime}\colon G^{\e\prime}\to{\rm TORSC}(E;\Pi^{\e\prime}\e)$ be the morphism of gerbes that is determined by $\mu^{*}$ and let
\[
j^{\e\prime}\colon \aut(G^{\e\prime}\e)\to (-\be {{{} \Pi^{\e\prime}\atop \bigwedge}\atop }\lbe A^{\e\prime})\circ \mu^{\e\prime}
\]
be the isomorphism of morphisms of stacks that is determined by $j^{*}$. From the previous considerations it follows that $(G^{\e\prime}, \mu^{\e\prime}, j^{\e\prime}\e)$ is an $(A^{\e\prime},\Pi^{\e\prime})$-gerbe on $E$.
\end{itemize}

Now we come to the statement from which the functoriality of $H^{2}$ will later follow.

\begin{theorem}{\rm (The functoriality theorem)} 
Let $(f,\varphi)\colon (A,\Pi)\to(A^{\e\prime},\Pi^{\e\prime})$ be a morphism of sheaves of crossed groups and $(G,\mu, j)$ an $(A, \Pi)$-gerbe. Then there exists an $(A^{\e\prime},\Pi^{\e\prime})$-gerbe $(G^{\e\prime},\mu^{\e\prime}, j^{\e\prime})$ as well as an $(f,\varphi)$-morphism $\lambda\colon G\to  G^{\e\prime}$.	
\end{theorem}
\begin{proof} We have already established the first part of the statement so that now it remains only to prove that the morphism
\[
\lambda\colon G\to  G^{\e\prime}=A(G^{*})
\]
is an $(f,\varphi)$-morphism. For every object $x$ in $G_{S}, S\in{\rm Ob}(E)$, we have a natural isomorphism
\[
i_{x}\colon \mu(x)\be {{{} \Pi\atop \bigwedge}\atop }\lbe \Pi^{\e\prime}
\isoto\mu^{\e\prime}(\lambda(x)).
\]

Indeed, we have
\[
\mu^{\e\prime}(\lambda(x))=\mu^{\e\prime}(a\circ \lambda^{*}(x))=(\e\mu^{\e\prime}\circ a)(\lambda^{*}(x))=\mu^{*}(\lambda^{*}(x)).
\]
Because of the way $m^{*}$ was constructed, we have
\[
\mu^{\e\prime}(\lambda(x))=\mu(x)\be {{{} \Pi\atop \bigwedge}\atop }\lbe \Pi^{\e\prime}.
\]
Thus it suffices to take $i_{x}={\rm id}_{\!\mu(x)\be {{{} \Pi\atop \bigwedge}\atop }\lbe \Pi^{\e\prime}}$.

We must also show that the following square commutes:
\[
\xymatrix{\uaut_{\e S}(x)\ar[d]_(.45){j_{x}}^(.45){\wr}\ar[rr]^(.45){\aut(\lambda)_{x}}&& \uaut_{\e S}(\lambda(x))\ar[d]^(.45){k_{\lambda(x)}}_(.45){\wr}\\
\mu(x)\be {{{} \Pi\atop \bigwedge}\atop }\lbe A\ar[rr]^(.4){\eta_{\e x}}&& \mu^{\e\prime}(\lambda(x)){{{} \Pi^{\e\prime}\atop \bigwedge}\atop }\lbe A^{\e\prime}
}
\]	
with $\eta_{\e x}$ equal to the following composition
\[
\xymatrix{\mu(x)\be {{{} \Pi\atop \bigwedge}\atop }\lbe A\ar[rr]^(.4){p_{x}\le\wedge\e f}
&&(\,\mu(x)\be {{{} \Pi\atop \bigwedge}\atop }\lbe \Pi^{\e\prime}\,)\be {{{} \Pi^{\e\prime}\atop \bigwedge}\atop }\lbe A^{\e\prime}\ar[rr]^{{\rm id}}&&\mu^{\e\prime}(\lambda(x)){{{} \Pi^{\e\prime}\atop \bigwedge}\atop }\lbe A^{\e\prime}\\
}
\]
The above square can be split into two squares
\[
\xymatrix{\ar @{} [drr] |{\rm I}\uaut_{\e S}(x)\ar[d]^(.4){j_{\e x}} _(.4){\wr}\ar[rr]^(.45){\aut(\lambda^{\be *}\lbe(x))_{x}}&& \ar @{} [drr] |{\rm II}\uaut_{\e S}(\lambda^{*}(x))\ar[d]^(.43){j^{*}_{\lambda^{\be *}\lbe\lbe(x)}}_(.43){\wr}\ar[rr]^(.45){\aut(a)_{\lambda^{*}(x)}}&&\uaut_{\e S}(a\circ\lambda^{*}(x))\ar[d]^(.43){j^{\e\prime}_{\lambda(x)}}
_(.4){\wr}\\
\mu(x)\be {{{} \Pi\atop \bigwedge}\atop }\lbe A\ar[rr]^(.32){p_{\le x}\le\wedge\e f}&&(\mu(x)\be {{{} \Pi\atop \bigwedge}\atop }\lbe \Pi^{\e\prime})\be {{{} \Pi^{\e\prime}\atop \bigwedge}\atop }\lbe A^{\e\prime}=\mu^{*}(\lambda^{*}(x))\be {{{} \Pi^{\e\prime}\atop \bigwedge}\atop }\lbe A^{\e\prime} \ar[rr]^(.62){{\rm id}}&& \mu^{\e\prime} (\lambda(x))\be{{{} \Pi^{\e\prime}\atop \bigwedge}\atop }\lbe A^{\e\prime}
}
\]
\begin{itemize}
\item square II commutes because of the way we defined $j^{\e\prime}$.
\item we examine square I further.

We have $\uaut_{\e S}(\lambda^{*}(x))=\uaut_{\e S}(x)\be {{{} \mu(x,A)\atop \bigwedge}\atop }\lbe \mu(x,A^{\e\prime}\e)$ and $j^{*}_{\lambda^{\be *}\lbe\lbe(x)}$  is the following composition
\[
\xymatrix{\uaut_{\e S}(x)\be {{{} \mu(x,A)\atop \bigwedge}\atop }\lbe \mu(x,A^{\e\prime}\e)\ar[r]^(.5){(1)}
&\mu(x,A)\be {{{} \mu(x,A)\atop \bigwedge}\atop }\lbe \mu(x,A^{\e\prime}\e)\ar[r]^(.6){(2)}_(.6){\sim}&\mu(x,A^{\e\prime}\e)\ar[r]^{(3)}&\mu^{*}(x)\be {{{} \Pi^{\e\prime}\atop \bigwedge}\atop }\lbe A^{\e\prime}\\
}
\]
where arrow (1) is equal to $j_{x}\wedge\mu(x,A^{\e\prime}\e)$, arrow (2) is the canonical isomorphism and arrow (3) is defined via the canonical isomorphism $\mu(x)\be {{{} \Pi\atop \bigwedge}\atop }\lbe A^{\e\prime}\isoto\mu(x)\be {{{} \Pi\atop \bigwedge}\atop}(\Pi^{\e\prime}\be {{{} \Pi^{\e\prime}\atop \bigwedge}\atop}\lbe A^{\e\prime})$ and the associativity of the contracted product.

We can verify that this square commutes locally.

Let $\sigma$ be an element of $\uaut_{\e S}(x)$. Let the image element $j_{x}(\sigma)\in
\mu(x)\be {{{} \Pi\atop \bigwedge}\atop }\lbe A$ be locally represented by a pair $(\alpha,a)\in\mu(x)\times A$.

$((p_{x}\wedge f\e)\circ j_{x})(\sigma)$ is locally determined by $((\alpha,1)^{*},f(a))$.

$(\e j^{*}_{\lambda^{\be *}\lbe(x)}\circ\aut(\lambda^{\be *})_{x})(\sigma)=\mu(x,f\e)(\, j_{x}(\sigma))$ and the latter element is locally determined by $(\alpha, f(a))$. The image of $(\alpha, f(a))$ under (3) is equal to $((\alpha,1)^{*},f(a))$ so that we may conclude that
\[
j^{*}_{\lambda^{*}(x)}\circ\aut(\lambda^{*})_{x}=(\e p_{x}\wedge f\e)\circ j_{x}.
\]
Them morphism $\lambda$ is thus an $(f,\varphi)$-morphism.
\end{itemize}
\end{proof} 

\subsection{Some theorems on $(f,\varphi)$-morphisms of gerbes.}

\begin{theorem}\label{331} Let $(f,\varphi)\colon (A,\Pi)\to(A^{\e\prime},\Pi^{\e\prime})$ be a morphism of sheaves of crossed groups on the site $E$, $\lambda\colon (G,\mu, j)\to (G^{\e\prime},\mu^{\e\prime}, j^{\e\prime})$ the $(f,\varphi)$-morphism we constructed in {\rm\ref{32}locally} and let $\gamma\colon (G,\mu, j)\to (G^{\e\prime\prime},\mu^{\e\prime\prime}, j^{\e\prime\prime})$ be any $(f,\varphi)$-morphism.

Then there exists an ${\rm id}_{\e(A^{\e\prime},\Pi^{\e\prime})}$-morphism
\[
\delta\colon (G^{\e\prime},\mu^{\e\prime}, j^{\e\prime})\to (G^{\e\prime\prime},\mu^{\e\prime\prime}, j^{\e\prime\prime})
\]   
such that $\delta\circ\lambda=\gamma$.
\end{theorem}
\begin{proof}
We should try to construct a morphism $\delta\colon G^{\e\prime}\to G^{\e\prime\prime}$. But since $G^{\e\prime}$ is the stack associated to $G^{\e*}$, it will suffice to define a morphism $\delta^{\e*}\colon G^{\e*}\to G^{\e\prime\prime}$.

Since ${\rm Ob}(G^{\e*})={\rm Ob}(G)$ and $\delta^{\e*}\circ\lambda^{\e*}=\gamma$, we must set, for every $x\in{\rm Ob}(G^{\e*}_{S})$,
\[
\delta^{\e*}(x)=\gamma(x).
\]
Now we describe the effect of $\delta^{\e*}$ on the morphisms.

For every object $x$ in $G^{\e*}_{\be S}$, we have a natural isomorphism
\[
\varepsilon_{\e x}\colon \mu^{*}(\lambda^{\be *}(x))\isoto\mu^{\e\prime\prime}(\gamma(x))
\]
because the fact that $\lambda$ and $\gamma$ are $(f,\varphi)$-morphisms implies that both $\mu^{*}(\lambda^{\be *}(x))$ and $\mu^{\e\prime\prime}(\gamma(x))$ are isomorphic to ${\rm TORSC}(E;\varphi)(\mu(x))=\mu(x)\be {{{} \Pi\atop \bigwedge}\atop }\lbe \Pi^{\e\prime}$.

Using $\varepsilon_{\e x}$, we have then the natural isomorphism
\[
(\,j^{\e\prime\prime}_{\gamma(x)})^{-1}\circ (\e \varepsilon_{\e x}\wedge A^{\e\prime}\e)\circ j^{*}_{\lambda^{\be *}\lbe(x)}\colon \uaut_{\e S}^{*}(x)=\uaut_{\e S}(\lambda^{\be *}(x))\to \uaut_{\e S}(\gamma(x))=\uaut_{\e S}(\delta^{\e *}\circ\lambda^{\be *}(x)).
\]

We have thus determined the effect of $\delta^{*}$ on the $S$-automorphisms of the object $x$ in $G_{S}^{*}$. This action can now be extended to $\isom_{S}^{*}(x,y)$ for every pair of $S$-objects $x,y$ in $G_{\be S}^{\e *}$, because  $x$ and $y$ are locally isomorphic. Thus we obtain an $E$-functor
\[
\delta^{\e*}\colon G^{\e*}\to G^{\e\prime\prime}
\]
such that $\delta^{\e*}\circ\lambda^{\e*}=\gamma$.

Further, for every $x\in{\rm Ob}(G_{S}^{*})$, we have the isomorphism
\[
\varepsilon_{\e x}\wedge\Pi^{\e\prime}\colon \mu^{*}(x)\be {{{} \Pi^{\e\prime}\atop \bigwedge}\atop }\lbe \Pi^{\e\prime}\isoto \mu^{\le\prime\prime}(\gamma(x))\be {{{} \Pi^{\e\prime}\atop \bigwedge}\atop }\lbe \Pi^{\e\prime}\simeq \mu^{\le\prime\prime}(\delta^{*}(x))
\]
and, because of the way $\delta^{*}$ was constructed, the following square commutes
\[
\xymatrix{\uaut_{\e S}^{*}(x)\ar[d]_(.45){j_{x}^{*}}^(.45){\wr}\ar[rr]^(.35){\aut(\delta^{*})_{x}}&& \uaut_{\e S}(\delta^{*}(x))=\uaut_{\e S}(\gamma(x))\ar[d]^(.45){j_{\lambda(x)}^{\e\prime\prime}}_(.45){\wr}\\
\mu^{*}(x)\be {{{} \Pi^{\e\prime}\atop \bigwedge}\atop }\lbe A^{\e\prime}\ar[rr]^(.4){\varepsilon_{\e x}\e\wedge\e A^{\e\prime}}&& \mu^{\e\prime\prime}(\gamma(x)){{{} \Pi^{\e\prime}\atop \bigwedge}\atop }\lbe A^{\e\prime}.
}
\]	
If we now let $\delta\colon G^{\e\prime}\to  G^{\e\prime\prime}$ be the unique $E$-functor such that $\delta\circ a=\delta^{*}$, then we can conclude that $\delta$ is an ${\rm id}_{\e(A^{\e\prime},\Pi^{\e\prime})}$-morphism such that $\delta\circ\lambda=\gamma$.
\end{proof}

\begin{theorem}\label{332} Let $(f,\varphi)\colon (A,\Pi)\to(A^{\e\prime},\Pi^{\e\prime})$ be a morphism of sheaves of crossed groups, $\gamma\colon (G,\mu, j)\to (F,\nu, k)$ and $\omega\colon (G,\mu, j)\to (H,\eta, l\e)$ be any two  $(f,\varphi)$-morphisms with a common source. Then we have an {\rm $E$-equivalence}
\[
\delta\colon F\isoto H
\]	
compatible with ${\rm id}_{(A^{\e\prime},\e\Pi^{\e\prime})}$ such that $\delta\circ\gamma\isoto\omega$.	
\end{theorem}
\begin{proof} By Theorem \ref{331} applied twice, we obtain ${\rm id}_{(A^{\e\prime},\Pi^{\e\prime})}$-morphisms $\delta_{F}\colon G^{\e\prime}\to F$ and $\delta_{H}\colon G^{\e\prime}\to H$ such that $\delta_{F}\circ \lambda=\gamma$ and $\delta_{H}\circ \lambda=\omega$.
According to Theorem \ref{234}, $\delta_{F}$ and $\delta_{H}$ are $E$-equivalences.

Let $\delta_{F}^{\e\prime}$ be a quasi-inverse of $\delta_{F}$. The functor $\delta=\delta_{H}\circ \delta_{F}^{\e\prime}$ satisfies all the requirements.
\end{proof}

\chapter{2-cohomology}\label{four}

\section{Definition of $H^{2}$}

\subsection{The $(A,\Pi)$-gerbe ${\rm TORSC}(E;A)$}\label{11}

\begin{theorem} Let $(A,\Pi)$ be a sheaf of crossed groups on the site $E$.
Then ${\rm TORSC}(E;A)$ has a unique canonical structure of $(A,\Pi)$-gerbe.
\end{theorem}
\begin{proof} We define a morphism of gerbes
\[
\mu_{\e\rm TORS}\colon {\rm TORSC}(E;A)\to{\rm TORSC}(E;\Pi)
\]
by setting, for every object $P$ of ${\rm TORSC}(E;A)_{S}$, $S\in{\rm Ob}(E)$,
\[
\mu_{\e\rm TORS}(P\le)=P\be {{{} A\atop \bigwedge}\atop }\lbe \Pi
\]

By the associativity of the contracted product, we have a natural isomorphism
\begin{equation}\label{uni}
\mu_{\e\rm TORS}(P\le)\be {{{} \Pi\atop \bigwedge}\atop }\lbe A=(P\be {{{} A\atop \bigwedge}\atop }\lbe \Pi)\be {{{} \Pi\atop \bigwedge}\atop }\lbe A\isoto
P\be {{{} A\atop \bigwedge}\atop }\lbe A.
\end{equation}
Here we have to take into account that in the formation of the contracted 
product $P\be {{{} A\atop \bigwedge}\atop }\lbe A$ the group $A$ acts on the left on $A$ via inner automorphisms. This implies, in particular, that the sheaf of groups  $P\be {{{} A\atop \bigwedge}\atop }\lbe A$ can be regarded as a twisted object of $A$ via the $A$-torsor $P$. According to \cite[III, 2.3.7]{14}, we have a natural isomorphism
\begin{equation}\label{dui}
\uaut_{\, A}(P\e)\isoto P\be {{{} A\atop \bigwedge}\atop }\lbe A.
\end{equation}
The composition of \eqref{dui} and the inverse of \eqref{uni} yields a natural isomorphism
\[
j_{\e\rm TORS}(P\le)\colon \uaut_{\e A}(P\e)\isoto \mu_{\e\rm TORS}(P\le)\be {{{} \Pi\atop \bigwedge}\atop }\lbe A.
\]
We must now consider whether the condition in the definition of $(A,\Pi)$-bundle
is satisfied. To simplify the notation, we will write $j_{\e\rm TORS}$ in place of $j_{\e\rm TORS}(P\le)$.
	
Let $x$ be any element of $\mu_{\e\rm TORS}(P\le)\be {{{} \Pi\atop \bigwedge}\atop }\lbe A=(P\be{{{} A\atop \bigwedge}\atop }\lbe \Pi\e) {{{} \Pi\atop \bigwedge}\atop }\lbe A$ and $(\e\bar{p},a)$ a local representative of $x$. Is $\mu_{\e\rm TORS}(j_{\e\rm TORS}^{-1}(x))(\bar{p})=\bar{p}\cdot\rho(a)$ ?

It suffices to prove that the above equality holds locally. The
element $\bar{p}$ is locally determined by the class $[(p,\alpha)]$ with $p\in P$ and $\alpha\in\Pi$. It follows from the definition of $j_{\e\rm TORS}$ that $j_{\e\rm TORS}^{-1}(x)$ is locally determined by $\sigma_{p}(\alpha\cdot a)$. Consequently, $\mu_{\e\rm TORS}(j_{\e\rm TORS}^{-1}(x))(\bar{p})$ is locally determined by the class $[(\sigma_{p}(\alpha\cdot  a))(p),\alpha)]$. This further equals
\[
\begin{array}{rcl}
[(\sigma_{p}(\alpha\cdot  a))(\e p),\alpha)]&=&[(\e p\cdot (\alpha\cdot  a),\alpha)]=[(\e p,\rho(\alpha\cdot  a)\cdot \alpha)]\\
&=&[(\e p,\alpha\cdot\rho(a)\cdot\alpha^{-1}\cdot \alpha\le)]=[(\e p,\alpha)]\cdot \rho(a).
\end{array}
\]
Thus $\mu_{\e\rm TORS}(j_{\e\rm TORS}^{-1}(x))(\bar{p})$ and $\bar{p}\cdot\rho(a)$ agree locally. We may thus conclude that
\[
({\rm TORSC}(E;A),\mu_{\e\rm TORS},j_{\e\rm TORS})
\]
is an $(A,\Pi)$-gerbe.	
\end{proof}

\subsection{Definition of $H^{2}$}

Let $\Phi=(A,\rho,\Pi,\phi)$ be a sheaf of crossed groups on $E$.  The cohomology in dimensions 0 and 1 agree with those of Giraud.
Thus we have
\[
\begin{array}{rcl}
H^{0}(E,\Phi)&=& H^{0}(A)=\varprojlim A\\
H^{1}(E,\Phi)&=&H^{1}(A),\text{ i.e., the set of isomorphism classes of $A$-torsors on $E$}
\end{array}
\]
If $(f,\varphi)\colon \Phi=(A,\Pi\e)\to \Phi^{\e\prime}=(A^{\e\prime},\Pi^{\e\prime}\e)$ is a morphism of sheaves of crossed groups, then $(f,\varphi)^{(0)}\colon H^{\e 0}(E,\Phi)\to H^{\e 0}(E,\Phi^{\e\prime}\e)$ maps a section $s$ to $f\circ s$ while $(f,\varphi)^{(1)}\colon H^{1}(E,\Phi)\to H^{1}(E,\Phi^{\e\prime}\e)$ maps a class $[P\e]$ to $[{}^{f}P\e]$, where ${}^{f}P=P\be{{{} A\atop \bigwedge}\atop }\lbe A^{\e\prime}_{d}$\,.

\begin{definition}
Let $\Phi=(A,\rho,\Pi,\phi)$ be a sheaf of crossed groups on $E$. Then
\[
H^{2}(E,\Phi)
\]
is the {\it set of all classes of $(A,\Pi)$-equivalent $(A,\Pi)$-gerbes on $E$.}
\end{definition}

Instead of $H^{2}(E,\Phi)$ we will sometimes also use the notation $H^{2}(A,\Phi)$. A class is called {\it neutral} if a representative has a section. From \ref{11} we know that the gerbe ${\rm TORSC}(E; A)$ of $A$-torsors is canonically endowed with a unique  structure of $(A,\Pi)$-gerbe.

The class containing ${\rm TORSC}(E; A)$ is called the {\it unit class}.

Let $(f,\varphi)\colon \Phi=(A,\Pi)\to \Phi^{\e\prime}=(A^{\e\prime},\Pi^{\e\prime})$ be a morphism of sheaves of crossed groups.  Then we have a {\it map}
\[
(f,\varphi)^{(2)}\colon H^{2}(E,\Phi)\to H^{2}(E,\Phi^{\e\prime}\e)
\]
which maps an element $g=[(G,\mu, j\e)]\in H^{2}(E,\Phi)$ to the element
$(f,\varphi)^{(2)}(g)=[(F,\nu, k\e)]\in H^{2}(E,\Phi^{\e\prime}\e)$. Here $(F,\nu, k\e)$ is an $(A^{\e\prime},\Pi^{\e\prime})$-gerbe such that there exists an $(f,\varphi)$-morphism
\[
\delta\colon (G,\mu, j\e)\to (F,\nu, k\e).
\]
To justify this definition, we can rely on the functoriality theorem and on Theorem \ref{332} from Chapter 3.

The map $(f,\varphi)^{(2)}$ sends a neutral element of $H^{2}(E,\Phi)$ to
a neutral element of $H^{2}(E,\Phi^{\e\prime}\e)$. For every pair of composable morphisms of sheaves of crossed groups
\[
\Phi=(A,\Pi)\overset{(f,\e\varphi)}{\lra}\Phi^{\e\prime}=(A^{\e\prime},\Pi^{\e\prime})\overset{(g,\psi)}{\lra}\Phi^{\e\prime\prime}=(A^{\e\prime\prime},\Pi^{\e\prime\prime})
\]
we have, by \cite[III, 2.3.3]{14},
\[
(g,\psi)^{(2)}\circ (f,\varphi)^{(2)}=(g\circ f,\psi\circ\varphi)^{(2)}.
\]
The functoriality of our $H^{2}$ is clear from the above.

\section{Exactness and naturality of the cohomology}

\subsection{Exactness of the cohomology}

First we reconsider the definition of {\it short exact} sequence of sheaves of crossed groups on $E$. The sequence
\[
e\to \Phi=(A,\rho,\Pi)\overset{(f,\e\varphi)}{\lra}\Phi^{\e\prime}=(A^{\e\prime},
\rho^{\e\prime},\Pi^{\e\prime})\overset{(h,\e\psi)}{\lra}\Phi^{\e\prime\prime}=(A^{\e\prime\prime},\rho^{\e\prime\prime},\Pi^{\e\prime\prime})\to e
\]
is called exact if it induces the following commutative diagram
\[
\xymatrix{&\Pi\ar[r]^(.45){\varphi}& \Pi^{\e\prime}\ar[r]^(.45){\psi}&
\Pi^{\e\prime\prime}\\
1\ar[r]&A\ar[u]^(.4){\rho} \ar[r]^(.4){f}& A^{\e\prime}\ar[u]_(.4){\rho^{\e\prime}} \ar[r]^{h}& A^{\e\prime\prime} \ar[u]_(.4){\rho^{\e\prime\prime}}\ar[r]&1
}
\]
which, in addition, satisfies the following two conditions:
\begin{equation}\label{uno}
1\to A\overset{\!f}{\to}A^{\e\prime}\overset{\!h}{\to}A^{\e\prime\prime}\to 1
\end{equation}
is a short exact sequence of sheaves of groups on the site $E$.
\begin{equation}\label{dos}
\varphi\colon \Pi\to \Pi^{\e\prime}\text{ is an isomorphism and } \psi\colon \Pi^{\e\prime}\to \Pi^{\e\prime\prime}\text{ is an epimorphism.}
\end{equation}
Because of condition \eqref{dos} we can set $\Pi^{\e\prime}=\Pi$ in the above diagram and replace it with
\[
\xymatrix{&\Pi\ar[r]^(.45){{\rm id}_{\le\Pi}}& \Pi\ar[r]^(.45){\psi}&
\Pi^{\e\prime\prime}\\
1\ar[r]&A\ar[u]^(.45){\rho} \ar[r]^(.45){f}& A^{\e\prime}\ar[u]_(.45){\rho^{\e\prime}} \ar[r]^{h}& A^{\e\prime\prime} \ar[u]_(.45){\rho^{\e\prime\prime}}\ar[r]&1
}
\]
\begin{theorem} Let $P$ be an $A^{\e\prime\prime}$-torsor on $E$ and $K(P)$ the gerbe of liftings of $P$ to $A^{\e\prime}$ via $h$. Then $K(P)$ is uniquely endowed with a canonical structure of $(A,\Pi)$-gerbe.
\end{theorem}
\begin{proof}\indent
\begin{itemize}
\item  Let $(Q, \lambda)$ be any $S$-object in $K(P)$, i.e., $Q$ is an $A^{\e\prime}$-torsor on $E/S$ and $\lambda\colon Q\to P^{S}$ is an $h$-morphism.

Then we define
\[
\mu_{\e P}(Q,\lambda)=Q\be{{{} A^{\e\prime}\atop \bigwedge}\atop }\lbe \Pi.
\]
Thus we obtain a morphism of gerbes on $E$
\[
\mu_{\e P}\colon K(P)\to {\rm TORSC}(E;\Pi).
\]
\item For every object $(Q, \lambda)\in K(P)_{S}$ we have a natural isomorphism
\[
j_{\e P}(Q, \lambda)\colon\uaut_{\e S}(Q, \lambda)\isoto \mu_{\e P}(Q,\lambda)\be{{{} \Pi\atop \bigwedge}\atop }\lbe A.
\] 
We describe the effect of $j_{\e P}^{-1}(Q, \lambda)$.

Let $x$ be any element of $\mu_{\e P}(Q,\lambda)\be{{{} \Pi\atop \bigwedge}\atop }\lbe A$. Then we know that $x$ is locally determined by a class of equivalent pairs $[(\bar{q},a)]$ with $\bar{q}\in \mu_{\e P}(Q,\lambda)$ and $a\in A$. The element $\bar{q}$ is, in turn, locally determined by a class of equivalent pairs $[(q,\alpha)]$ with $q\in Q$ and $\alpha\in\Pi$.

We define then $j_{\e P}^{-1}(Q, \lambda)(x)$ as the automorphism of $(Q, \lambda)$ that is locally equal to $\sigma_{q}(\alpha\cdot a)$.

\item Now, in what follows, we use the notation $j_{\e P}^{-1}$ in place of $j_{\e P}^{-1}(Q, \lambda)$.

We wish to show that $\mu_{\e P}(j_{\e P}^{-1}(x))(\e\bar{q}\e)= \bar{q}\cdot\rho(a)$.

Considering that these elements are values of sheaves, it suffices to check that they agree locally. Viewed locally, $\bar{q}$ is  determined by the class
$[(q,\alpha)]$. The image element $\mu_{\e P}(j_{\e P}^{-1}(x))(\e\bar{q}\e)$
is then locally determined by the class $[(\sigma_{q}(\alpha\times a)(q),\alpha)]$. Now
\[
\begin{array}{rcl}
[(\sigma_{q}(\alpha\cdot a)(q),\alpha)]&=&[(q\cdot (\alpha\cdot a),\alpha)]=[(q,\rho(\alpha\cdot a)\cdot\alpha)]\\
&=&[(q,\alpha\cdot\rho(a)\cdot\alpha^{-1}\cdot \alpha)]=[(q,\alpha\cdot\rho(a))]=[(q,\alpha)]\cdot\rho(a).
\end{array}
\]
Locally $\mu_{\e P}(j_{\e P}^{-1}(x))(\e\bar{q}\e)$ and $\bar{q}\cdot\rho(a)$ are determined by the same element.  Thus they are equal.

Thus we have defined on $K(P)$ a structure of $(A,\Pi)$-gerbe.

{\bf Notation}. The gerbe $K(P)$ with this canonical structure of
$(A,\Pi)$-gerbe is denoted by.
\[
(K(P),\mu_{\e P},j_{\e P}).
\]
\end{itemize}	
\end{proof}

\begin{theorem} \label{212} The cartesian $E$-functor
\[
k(P)\colon K(P)\to {\rm TORSC}(E;A^{\e\prime}\e),\,(Q,\lambda)\mapsto Q,
\]
is an $(f,1_{\Pi})$-morphism
\end{theorem}
\begin{proof}\indent
\begin{itemize}
\item Let $(Q,\lambda)$ be an object of the fiber category $K(P)_{S}$.  
Because of the way the $(A^{\e\prime}, \Pi)$-gerbe structure on ${\rm TORSC}(E; A^{\e\prime})$ was defined, we have
\[
\mu_{\e\rm TORS}(k(P)(Q,\lambda))=Q\be{{{} A^{\e\prime}\atop \bigwedge}\atop }\lbe\Pi.
\]
On the other hand, it is clear from the way in which the $(A,\Pi)$-gerbe structure on $K(P)$ was defined that
\[
\mu_{\e P}(Q,\lambda))=Q\be{{{} A^{\e\prime}\atop \bigwedge}\atop }\lbe\Pi.
\]
Because of \ref{pgg} we have a natural isomorphism
\[
i_{(Q,\lambda)}\colon \mu_{\e P}(Q,\lambda))\be{{{} \Pi\atop \bigwedge}\atop }\lbe\Pi\isoto \mu_{\e\rm TORS}(k(P)(Q,\lambda)).
\]

\item Consider the following square:
\[
\xymatrix{\uaut_{\e S}(Q,\lambda)
\ar[d]_(.45){j_{P}(Q,\lambda)}^(.45){\wr}\ar[rr]^(.45){\aut(k(P))_{(Q,\lambda)}}&& \uaut_{\e S}(k(P)(Q,\lambda))\ar[d]^(.4){j_{\e\rm TORS}(k(P)(Q,\lambda))}_(.4){\wr}\\
\mu_{P}(Q,\lambda)\be {{{} \Pi\atop \bigwedge}\atop }\lbe A\ar[rr]^(.4){\eta_{\e (Q,\lambda)}}&& \mu_{\e\rm TORS}(k(P)(Q,\lambda)){{{} \Pi\atop \bigwedge}\atop }\lbe A^{\e\prime}
}
\]	 
where $\eta_{\e (Q,\lambda)}$ is equal to the composition
\[
\xymatrix{\mu_{P}(Q,\lambda)\be {{{} \Pi\atop \bigwedge}\atop }\lbe A\ar[rr]^(.4){p_{(Q,\lambda)}\le\wedge\e f}
&&(\e\mu_{P}(Q,\lambda)\be {{{} \Pi\atop \bigwedge}\atop }\lbe \Pi\,)\be {{{} \Pi\atop \bigwedge}\atop }\lbe A^{\e\prime}\ar[rr]^{i_{(Q,\lambda)}\wedge A^{\e\prime}}&&\mu_{\e\rm TORS}(k(P)(Q,\lambda)){{{} \Pi\atop \bigwedge}\atop }\lbe A^{\e\prime}\\
}
\]
Because it is simpler to describe the effect of $j_{\e P}^{-1}$ and $j_{\e \rm TORS}^{-1}$, we will try to show that
\[
j_{\e \rm TORS}^{-1}\circ \eta_{\e (Q,\lambda)}=\aut(k(P))_{(Q,\lambda)}\circ
j_{\e P}^{-1}. 
\]
Let $x$ be any element of $\mu_{P}(Q,\lambda)\be {{{} \Pi\atop \bigwedge}\atop }\lbe A$, $[(\bar{q},\alpha)]$ a class that determines $x$ locally and $[(q, \alpha)]$ a class that determines the element $\bar{q}$ locally.
Then $(\aut(k(P))_{(Q,\lambda)}\circ
j_{\e P}^{-1})(x)$ is locally equal to $\sigma_{q}(\e f(\alpha\cdot a\e))$.   
On the other hand, $(j_{\e \rm TORS}^{-1}\circ \eta_{\e (Q,\lambda)})(x)$ is locally equal to $\sigma_{q}(\alpha\cdot f(a\e))$.
Since $f$ is compatible with $1_{\Pi}$, we have $\alpha\cdot f(a)=f(\alpha\cdot a)$. The square commutes and thus $k(P)$ is an $(f, 1)$-morphism.
\end{itemize}
\end{proof}

\begin{lemma}\label{213} Let $(f,\varphi)\colon (A,\Pi)\to(A^{\e\prime},\Pi^{\e\prime})$ be a morphism of sheaves of crossed groups and $\gamma\colon (G,\mu, j)\to (F,\nu, k)$ an $(f,\varphi)$-morphism. If $f$ is the unit morphism then there exists a section $s$ of $F$ such that $\gamma\isoto s\circ g$, where $g\colon G\to E$ is the projection from $G$ to $E$.
\end{lemma}
\begin{proof} That $\gamma$ is an $(f,\varphi)$-morphism means that, for every object $x$ of $G_{S}$, $S\in{\rm Ob}(E)$, there exists a natural isomorphism
\[
i_{x}\colon\mu(x)\be {{{} \Pi\atop \bigwedge}\atop }\lbe \Pi^{\e\prime}\isoto\nu(\gamma(x))
\]
such that $k_{\gamma(x)}\circ\aut(\gamma)_{x}=(\e i_{x}\wedge A^{\e\prime})\circ(\e p_{x}\wedge f\e)\circ j_{x}$.
Since $f$ is the unit morphism, i.e., $f(a)=1$ for all $a\in A$, it follows from this equality that
\[
\aut(\gamma)_{x}\colon \uaut_{\e S}(x)\to\uaut_{\e S}(\gamma(x))
\]
is also the unit morphism. Thus $\aut(\gamma)_{x}$ maps every element of
$\uaut_{\e S}(x)$ to ${\rm id}_{\e\gamma(x)}$.

It follows that, for every pair of $S$-morphisms $n, m\in\Hom_{\e S}(y,x)$ in $G$, we have 
\begin{equation}\label{aa}
\gamma(n)=\gamma(m).
\end{equation}
After all, $n=(n\circ m^{-1})\circ m$ and $n\circ m^{-1}\in \uaut_{\e S}(x)$.

We must now determine a cartesian section of $F$, i.e., an object of the category $\underline{\rm Cart}_{\e E}(E,F\e)$. Since $F$ is a gerbe and thus, in particular, a stack, we know that for every refinement $\mathcal V$ of $E$ we have
\begin{equation}\label{veld}
\underline{\rm Cart}_{\e E}(E,F\e)\isoto \underline{\rm Cart}_{\e E}(\mathcal V,F\e), s\mapsto s\be\mid_{\mathcal V},
\end{equation}
Since $G$ is a gerbe there exists a refinement $\mathcal V$ of $E$ such
that ${\rm Ob}(G_{S})\neq\emptyset$ for every $S\in \mathcal V$. Thus for every object $S\in \mathcal V$ we can choose an object in the fiber of $G$ over $S$.
Let us denote this chosen object by $x_{S}$.

Then we define
\[
s(S)=\gamma(x_{S}).
\]
Let $f\colon T\to S$ be any arrow in $\mathcal V$, $x_{S}^{f}$ an inverse $f$-image of $x_{S}$ and $x_{f}\colon x_{S}^{f}\to x_{S}$ the transport morphism.

Since $G$ is a gerbe, $x_{T}$ and $x_{S}^{f}$ are locally isomorphic.

Although these local isomorphisms $x_{T}\isoto x_{S}^{f}$ do not necessarily form a coherent family in $G$, the family formed by
their $\gamma$-images is coherent in $F$ and this is so by \eqref{aa}.
This coherent family of local isomorphisms then determines a $T$-isomorphism
$i_{f}\colon \gamma(x_{T})\isoto \gamma(x_{S}^{f})$.

We define $s(f)$ as follows:
\[
s(f)=\gamma(x_{f})\circ i_{f}.
\]
Thus we have determined an object in 
$\underline{\rm Cart}_{\e E}(\mathcal V,F\e)$. Due to the equivalence \eqref{veld}, this determines (up to isomorphism) an object in $\underline{\rm Cart}_{\e E}(E,F\e)$..	
\end{proof}

\begin{theorem}\label{214} Let $(G,\mu, j\e)$ be an $(A,\Pi)$-gerbe on $E$.

If $\gamma\colon (G,\mu, j\e)\to ({\rm TORSC}(E;A^{\e\prime}),\mu_{\e\rm TORS},j_{\e\rm TORS})$ is an $(f,1_{\Pi})$-morphism, then there exists an $A^{\e\prime\prime}$-torsor $P$ as well as an $(A,\Pi)$-equivalence
\[
\delta\colon (G,\mu,j)\isoto (K(P),\mu_{\e P},j_{\e P}\e).
\]
\end{theorem}
\begin{proof} Since ${\rm TORSC}(E, h)\circ\gamma$ is a $(g,\psi)\circ(f,1_{\Pi})$-morphism and $g\circ f$ is the unit morphism,
the lemma shows that there exists a section $s_{P}$ of ${\rm TORSC}(E, A^{\e\prime\prime})$ such that
\[
{\rm TORSC}(E, h)\circ\gamma\underset{\eta}{\isoto}s_{P}\circ g^{\e\prime},
\]
where $g^{\e\prime}$ is the projection from $G$ to $E$.

Let $(K(P),\mu_{\e P},j_{\e P})$ be the gerbe of liftings of the section $s_{P}$ with its structure of $(A,\Pi)$-gerbe. By Theorem \ref{212}, $k(P)\colon K(P)\to{\rm TORSC}(E, A^{\e\prime})$ is an 
$(f,1_{\Pi})$-morphism. Consequently, there exists a morphism $\delta\colon G\to K(P)$ that is completely determined by the conditions
\[
k(P)\circ\delta=\gamma\quad\text{ and }\quad \kappa(s_{P})\ast\delta=\eta\quad\cite[{\rm IV}, 2.5.2]{14}
\]
It follows that, for every object $x$ in $G_{S}$ the image object $\delta(x)$ is a pair $(Q,\lambda)$ with $Q=\gamma(x)$ and $\lambda=\eta_{\e x}\circ \omega_{\e Q}$, where $\omega_{\e Q}$ is the $h$-morphism from $Q$ to ${}^{h}Q=Q\be {{{} A^{\prime}\atop \bigwedge}\atop }\lbe A^{\e\prime\prime}$.

The morphism $\delta$ is a $(1_{A},1_{\Pi})$-morphism.

Indeed, for every object $x$ in $G_{S}$ we have a natural isomorphism
\[
i_{x}\colon \mu(x){{{}\Pi\atop \bigwedge}\atop }\lbe \Pi\isoto\mu_{\e P}(\delta(x))
\]
because $\gamma$ is an $(f,1_{\Pi})$-morphism and we have
\[
\mu_{\e\rm TORS}(\gamma(x))=\gamma(x){{{}A^{\e\prime}\atop \bigwedge}\atop }\lbe \Pi=(k(P)\circ \delta)(x){{{}A^{\e\prime}\atop \bigwedge}\atop }\lbe \Pi=\mu_{\e P}(\delta(x)).
\]
Further, we have the following equality:
\[
((\e i_{\e x}\circ p_{\e x}\e)\wedge A\e)\circ j_{\e x}=j_{\e P}(\delta(x))\circ \aut(\delta\e)_{x}.
\]	
We consider the following diagram
\[
\xymatrix{\ar @{} [drr] |{\rm I}\uaut_{\e S}(x)\ar[d]^(.4){j_{\e x}} _(.4){\wr}\ar[rr]^(.45){\aut(\delta)_{x}}&& \ar @{} [drr] |{\rm II}\uaut_{\e S}(\delta(x))\ar[d]^(.43){j_{\e P}(\delta(x))}_(.43){\wr}\ar[rr]^(.45){\aut(k(P))_{\delta(x)}}&&\uaut_{\e S}(k(P)(\delta(x)))\ar[d]^(.43){j_{\e \rm TORS}(\gamma(x))}_(.4){\wr}\\
\mu(x)\be {{{} \Pi\atop \bigwedge}\atop }\lbe A\ar[rr]^(.45){(i_{\le x}\le\circ \le p_{\le x}\le)\le\wedge\le 1_{A}}&&\mu_{\le P}(\delta(x))\be {{{} \Pi\atop \bigwedge}\atop }\lbe A\ar[rr]^(.42){{\rm id}\le\wedge\le f}&& \mu_{\e\rm TORS}(k(P)(\delta(x)))\be{{{} \Pi\atop \bigwedge}\atop }\lbe A^{\e\prime}\e.
}
\]
The diagram obtained by combining I and II commutes since $\gamma$ is an $(f,1_{\Pi})$-morphism and II commutes since $k(P)$ is also an $(f,1_{\Pi})$-morphism. From this we may conclude that square I also commutes since ${\rm id}\le\wedge\le f$ is a monomorphism because $f$ is a monomorphism.
\end{proof}

\begin{definition} Let
\[
e\to \Phi=(A,\rho,\Pi)\overset{(f,\e\varphi)}{\lra}\Phi^{\e\prime}=(A^{\e\prime},
\rho^{\e\prime},\Pi^{\e\prime})\overset{(h,\e\psi)}{\lra}\Phi^{\e\prime\prime}=(A^{\e\prime\prime},\rho^{\e\prime\prime},\Pi^{\e\prime\prime})\to e
\]
be a short exact sequence of sheaves of crossed groups.

Then we define a map
\[
d\colon H^{1}(E;\Phi^{\e\prime\prime})\to H^{2}(E;\Phi)
\]
by mapping $p\in H^{1}(E;\Phi^{\e\prime\prime})$ to the class of the gerbe of liftings to $A^{\e\prime}$ of a representative $P$ of $p$.

We call $d$ the second coboundary map.
\end{definition}

\begin{theorem}\label{giv} Assume as given a short exact sequence of sheaves of crossed groups
\[
e\to \Phi=(A,\Pi)\overset{(f,\e\varphi)}{\lra}\Phi^{\e\prime}=(A^{\e\prime},
\Pi^{\e\prime})\overset{(h,\e\psi)}{\lra}\Phi^{\e\prime\prime}=(A^{\e\prime\prime},\Pi^{\e\prime\prime})\to e
\]
and an  $(h,\psi)$-morphism of gerbes
\[
\gamma\colon(F,\nu,k)\to(H,\omega,1).
\]

If $s$ is a section of $H$, then the gerbe $K(s)$ of liftings of $s$ to $F$ has a structure of $(A,\Pi)$-gerbe such that $k(s)\colon K(s)\to F$ is an $(f,1_{\Pi})$-morphism.
\end{theorem}
\begin{proof}\indent
\begin{itemize}
\item The $(A,\Pi)$-gerbe structure on $K(s)$.

We obtain a morphism of gerbes
\[
\mu_{\le s}\colon K(s)\to{\rm TORSC}(E;\Pi)
\]
if, for every $S$-object $(z, m)$ in $K(s)$, we set
\[
\mu_{\le s}(z,m)=\nu(z).
\]
Next, we need to show that there exists a natural isomorphism
\[
j_{\e s}(z,m)\colon\uaut_{\e S}(z,m)\to \mu_{\le s}(z,m)\be{{{} \Pi\atop \bigwedge}\atop }\lbe A.
\]
Since $F$ is an $(A^{\e\prime},\Pi)$-gerbe and $H$ is an $(A^{\e\prime\prime},\Pi^{\e\prime\prime})$-gerbe we have the following diagram:
\[
\xymatrix{\uaut_{\e S}(z,m)\ar[rrr]^(.48){\aut(k(s))_{(z,m)}}&&&
\uaut_{\e S}(z)\ar[d]^(.43){k_{\e z}}_(.43){\wr}\ar[rrr]^(.45){\aut(\gamma)_{z}}&&&\uaut_{\e S}(\gamma(z))\ar[d]^(.43){1_{\le\gamma(z)}}_(.4){\wr}\\
\nu(z)\be {{{} \Pi\atop \bigwedge}\atop }\lbe A\ar[rrr]^(.4){1_{\nu(z)}\le\wedge\e f}&&&\nu(z)\be {{{} \Pi\atop \bigwedge}\atop }\lbe A^{\e\prime}\ar[rrr]^(.45){\beta_{z}=(i_{z}\le\wedge\le h\le)\circ(p_{z}\le
\wedge
\le 1_{A^{\le\prime}})}&&& \omega(\e\gamma(z))\be{{{} \Pi^{\e\prime\prime}\atop \bigwedge}\atop }\lbe A^{\e\prime\prime}
}
\]	
where the second square commutes since $\gamma$ is an $(h,\psi)$-morphism. An automorphism $\sigma\colon z\isoto z$ belongs to $\aut_{\e S}(z,m)$ if, and only if, $m\circ\gamma(\sigma)=m$ or if, and only if, $\gamma(\sigma)={\rm id}_{\e\gamma(z)}$. This shows that $\uaut_{\e S}(z,m)$ is a kernel of $\aut(\gamma)_{z}
\colon \uaut_{\e S}(z)\to \uaut_{\e S}(\gamma(z))$ and thus also of $1_{\gamma(z)}\circ\aut(\gamma)_{z}$ using the commutation with $\beta_{z}$. Since $\nu(z)\be {{{} \Pi\atop \bigwedge}\atop }\lbe A$ is likewise a kernel of $\beta_{z}$, there exists a unique isomorphism
\[
j_{\e s}(z,m)\colon\uaut_{\e S}(z,m)\isoto \nu(z)\be{{{} \Pi\atop \bigwedge}\atop }\lbe A=\mu_{\le s}(z,m)\be{{{} \Pi\atop \bigwedge}\atop }\lbe A
\]
such that
\[
(\e 1_{\nu(z)}\e\wedge\e f\e)\circ j_{\e s}(z,m)=k_{z}\circ\aut(k(s))_{(z,m)}.
\]
Let $x$ be an element of $(\e\mu_{\le s}(z,m)\be{{{} \Pi\atop \bigwedge}\atop }\lbe A)(1_{S})$ and $(\alpha,a)$ a local representative of $x$.

We set $y=j_{\e s}^{-1}(z,m)(x)$, whence $y=k_{z}^{-1}((1_{\nu(z)}\e\wedge\e f\e)(x))$.

Now $\mu_{s}(y)(\alpha)=\nu(y)(\alpha)$ and $\nu(y)(\alpha)=\alpha\cdot\rho^{\e\prime}(f(a))$ since $(\alpha,f(a))$ is a local representative of $(1_{\nu(z)}\e\wedge\e f\e)(x)$ and $(F,\nu,k)$ is an $(A^{\e\prime},\Pi)$-gerbe.

Since $(f,1_{\Pi})$ is a morphism of sheaves of crossed groups we have $\rho^{\e\prime}(f(a))=\rho(a)$.

Thus $\mu_{s}(y)(\alpha)=\alpha\cdot \rho(a)$ and, therefore, $(K(s),\mu_{s},j_{s})$ is an $(A,\Pi)$-gerbe.

\item Is $k(s)$ an  $(f,1_{\Pi})$-morphism?
	
For every object $(z, m)$ of $K(s)_{S}$ we have a natural isomorphism
\[
i_{\le s}(z,m)\colon\mu_{s}(z,m)\be {{{} \Pi\atop \bigwedge}\atop }\lbe \Pi\isoto \nu(k(s)(z,m))
\]
because $\mu_{s}(z,m)\be{{{} \Pi\atop \bigwedge}\atop }\lbe \Pi=\nu(z)\be{{{} \Pi\atop \bigwedge}\atop }\lbe \Pi$, $\nu(k(s)(z,m))=\nu(z)$ and we have a canonical isomorphism $\nu(z)\be{{{} \Pi\atop \bigwedge}\atop }\lbe \Pi\isoto\nu(z)$.

The diagram
\[
\xymatrix{\uaut_{\e S}(z,m)\ar[d]_(.45){j_{s}(z,m)}^(.4){\wr}\ar[rr]^(.45){\aut(k(s))_{(z,m)}}&& \uaut_{\e S}(k(s)(z,m))\ar[d]^(.4){k_{z}}_(.4){\wr}\\
\mu_{s}(z,m)\be {{{} \Pi\atop \bigwedge}\atop }\lbe A\ar[rr]^(.45){1_{\nu(z)}\le\wedge\le f}&& \nu(z){{{} \Pi\atop \bigwedge}\atop }\lbe A^{\e\prime}
}
\]	
commutes because of the way $j_{\le s}$ was defined. Thus $k(s)\colon K(s)\to F$ is an $(f,1_{\Pi})$-morphism.

\end{itemize}
\end{proof}

The previous theorems will now allow us to show that to a short exact sequence of sheaves of crossed groups there corresponds an exact cohomology sequence. The next theorem explains in what sense this exactness must be understood.

\begin{theorem} Let
\[
e\to \Phi=(A,\Pi)\overset{(f,\e 1_{\Pi})}{\lra}\Phi^{\e\prime}=(A^{\e\prime},
\Pi)\overset{(h,\e\psi)}{\lra}\Phi^{\e\prime\prime}=(A^{\e\prime\prime},\Pi^{\e\prime\prime})\to e
\]
be a short exact sequence of sheaves of crossed groups. Then we have a cohomology sequence
\[
\begin{array}{rcl}
1&\to& H^{0}(E;\Phi)\to\dots\to H^{1}(E;\Phi^{\e\prime})\overset{(h,\psi)^{(1)}}{\lra}H^{1}(E;\Phi
^{\e\prime\prime})
\overset{d}{\lra} H^{\le 2}(E;\Phi)\\
&\overset{(f,1_{\Pi})^{(2)}}{\lra}& H^{\e 2}(E;\Phi^{\e\prime})
\overset{(h,\psi)^{(2)}}{\lra} H^{\e 2}(E;\Phi^{\e\prime\prime})
\end{array}
\]
which is {\rm exact} in the following sense:
\begin{enumerate}
\item[\rm (i)] in order for $p\in H^{1}(E;\Phi
^{\e\prime\prime})$ to belong to ${\rm Im}(h,\psi)^{(1)}$, it is necessary and 
sufficient that $d(p)$ be a {\rm neutral} element.
\item[\rm (ii)] in order for $x\in H^{2}(E;\Phi)$ to belong to the image of $d$, it is necessary and sufficient that $(f,1_{\Pi})^{(2)}(x)$ be the {\rm unit element}.
\item[\rm (iii)] in order for $x\in H^{2}(E;\Phi^{\e\prime})$ to belong to
${\rm Im}(f,1_{\Pi})^{(2)}$, it is necessary and 
sufficient that $(h,\psi)^{(2)}(x)$ be a {\rm neutral} element.
\end{enumerate}
\end{theorem}
\begin{proof}\indent
\begin{enumerate}
\item[\rm (i)]  Let $p=[P\e]$ be an element of $H^{1}(E;\Phi
^{\e\prime\prime})$. The element $p$ belongs to ${\rm Im}(h,\psi)^{(1)}$ if, and only if, there exists an $A^{\e\prime}$-torsor $Q$ and an $h$-morphism $\lambda\colon Q\to P$. But this means that $P$ can be lifted $A^{\e\prime}$, in other words, The gerbe $K(P)$ has a section. Thus $d(p)=[K(P)]$ is a neutral element.
\item[\rm (ii)] holds by Theorems \ref{212} and \ref{214}.
\item[\rm (iii)] holds by Lemma \ref{213} and Theorem \ref{giv}.
\end{enumerate}
\end{proof}

\subsection{Naturality of the cohomology}

We wish to show that to a given commutative diagram of short exact sequences of sheaves of crossed groups there corresponds a commutative diagram of exact cohomology sequences.

First we prove a theorem.

Consider the following {\it commutative} diagram of short exact sequences 
of sheaves of crossed groups:
\begin{equation}\label{big}
\xymatrix{e\ar[r]&\Phi=(A,\Pi)\ar[d]_{(m,\,\omega)}\ar[r]^(.45){(f,\e 1_{\Pi})}& \Phi^{\e\prime}=(A^{\e\prime},\Pi)\ar[d]_{(m^{\le\prime},\,\omega)}
\ar[r]^(.5){(h,\e\psi)}&
\Phi^{\e\prime\prime}=(A^{\e\prime\prime},\Pi^{\e\prime\prime})\ar[d]^(.45){(m^{\le\prime\prime},\,\omega^{\e\prime\prime})}\ar[r]& e\\
e\ar[r]&\Psi=(B,\Sigma)\ar[r]^(.5){(g,1_{\Sigma})}& \Psi^{\e\prime}=(B^{\e\prime},\Sigma)\ar[r]^{(1,\eta)}& \Psi^{\e\prime\prime}=(B^{\e\prime\prime},\Sigma^{\e\prime\prime}) \ar[r]& e
}
\end{equation}
Further, let $P$ be an $A^{\e\prime\prime}$-torsor, $s_{P}$ the corresponding section of ${\rm TORSC}(E;A^{\e\prime\prime})$ and $K(P)$ the gerbe of liftings of $s_{P}$ relative to the morphism ${\rm TORSC}(E;h)$. From the $A^{\e\prime\prime}$-torsor $P$ we obtain, by expanding its structural
group via $m^{\e\prime\prime}$, a $B^{\e\prime\prime}$-torsor $P^{\e\prime\prime}$. We denote the corresponding section of ${\rm TORSC}(E;B^{\e\prime\prime})$ by $s_{P^{\le\prime\prime}}$ and let $K(P^{\e\prime\prime})$ be the gerbe of liftings of $s_{P^{\le\prime\prime}}$ with respect to the morphism ${\rm TORSC}(E;1)$.

\begin{theorem} There exists a unique $(m,\omega)$-morphism
\[
k\colon K(P)\to K(P^{\le\prime\prime})
\]
such that $k(P^{\le\prime\prime})\circ k={\rm TORSC}(E;m^{\e\prime}\e)\circ k(P)$.
\end{theorem}
\begin{proof} From \ref{pgg} and \ref{assoc} we conclude that there exists an isomorphism of morphisms of stacks
\[
\alpha\colon {\rm TORSC}(E;1)\circ {\rm TORSC}(E;m^{\e\prime}\e)\isoto {\rm TORSC}(E;m^{\e\prime\prime}\e)\circ {\rm TORSC}(E;h).
\]
Now we define an $E$-functor $k\colon K(P)\to K(P^{\le\prime\prime})$. Let $(Q,v)$ be an $S$-object from $K(P)$, i.e., $Q\in {\rm Ob}({\rm TORSC}(E;A^{\e\prime})_{S})$ and $v\colon{\rm TORSC}(E;h)(Q)\isoto s_{P}(S)=P^{S}$ is an $A^{\le\prime\prime}$-morphism in ${\rm TORSC}(E, A^{\le\prime\prime})$.

Since $k$ must satisfy ${\rm TORSC}(E;m^{\e\prime}\e)\circ k(P)=k(P^{\le\prime\prime})\circ k$, we must define the effect of $k$ on the objects as follows:
\[
k(Q,v)=({\rm TORSC}(E;m^{\e\prime}\e)(Q), {\rm TORSC}(E;m^{\e\prime\prime})(v)\circ \alpha_{\e Q}).
\]
Let $\varphi\colon(R,w)\to(Q,v)$ be an $(f\colon T\to S)$-morphism of $K(P)$.
We set
\[
k(\varphi)={\rm TORSC}(E;m^{\e\prime}\e)(\varphi).
\]
The $E$-functor $k$ thus defined also satisfies the following condition:
\[
\kappa(s_{P^{\le\prime\prime}})\ast k=({\rm TORSC}(E;m^{\e\prime\prime}\e)\ast\kappa(s_{P}))\circ(\alpha\ast k(P)).
\]
We must now consider whether $k$ is an $(m,\omega)$-morphism.

Let $(Q, v)$ be an object from $K(P)_{S}$. From the definition of $\mu_{\e P}$ we have $\mu_{\le P}(Q,v)=Q\be {{{} A^{\e\prime}\atop \bigwedge}\atop }\lbe \Pi$.
Further, we have
\[
\begin{array}{rcl}
\mu_{\le P^{\e\prime\prime}}(k(Q,v))&=&\mu_{P^{\e\prime\prime}}({\rm TORSC}(E;m^{\e\prime}\e)(Q),{\rm TORSC}(E;m^{\e\prime\prime}\e)(v)\circ \alpha_{\e Q})\\
&=&{\rm TORSC}(E;m^{\e\prime}\e)(Q)\be {{{} B^{\e\prime}\atop \bigwedge}\atop }\lbe \Sigma=(Q\be {{{} A^{\e\prime}\atop \bigwedge}\atop }\lbe B^{\e\prime}\e)
\be {{{} B^{\e\prime}\atop \bigwedge}\atop }\lbe \Sigma.
\end{array}
\]
By \ref{pgg} and \ref{assoc}, we have canonical isomorphisms
\[
(Q\be {{{} A^{\e\prime}\atop \bigwedge}\atop }\lbe B^{\e\prime}\e)
\be {{{} B^{\e\prime}\atop \bigwedge}\atop }\lbe \Sigma\isoto
Q\be {{{} A^{\e\prime}\atop \bigwedge}\atop }\lbe \Sigma
\]
and
\[
(Q\be {{{} A^{\e\prime}\atop \bigwedge}\atop }\lbe\Pi)
\be {{{} \Pi\atop \bigwedge}\atop }\lbe \Sigma\isoto
Q\be {{{} A^{\e\prime}\atop \bigwedge}\atop }\lbe \Sigma.
\]
By composition we obtain a natural isomorphism
\[
i_{(Q,v)}\colon \mu_{\e P}(Q,v)\be {{{} \Pi\atop \bigwedge}\atop }\lbe\Sigma\isoto\mu_{P^{\e\prime\prime}}(k(Q,v)).
\]
               
Finally, we show that
\[
(i_{(Q,v)}\wedge B)\circ(\e p_{(Q,v)}\wedge m)\circ j_{P}(Q,v)=j_{P^{\e\prime\prime}}(k(Q,v))\circ \aut(k)_{(Q,v)}.
\]
For simplicity, let us denote the left-hand side by $l$ and the right-hand side by $r$ and let $t\colon \mu_{P^{\e\prime\prime}}(k(Q,v)){{{} \Sigma\atop \bigwedge}\atop }\lbe B\to \mu_{\rm TORS}((k(P^{\e\prime\prime})\circ k)(Q,v)\be {{{} \Sigma\atop \bigwedge}\atop }\lbe B^{\e\prime}\e)$ be the composition $(i^{\e\prime\prime}_{(k(P^{\e\prime\prime})\le\circ\le k\e)(Q,v)}\wedge
B^{\e\prime}\e)\circ (\e p_{\le k(Q,v)}\wedge g\e)$.

Since $k(P)$ is an $(f,1_{\Pi})$-morphism, ${\rm TORSC}(E;m^{\e\prime}\e)$ is an $(h,\psi)$-morphism and $k(P^{\e\prime\prime})$ is a $(g,1_{\Sigma})$-morphism, we have $t\circ l=t\circ r$. Now $t$ is a monomorphism since
\[
i^{\e\prime\prime}_{(k(P^{\e\prime\prime})\e\circ\e k\e)(Q,v)}\circ p_{\e k(Q,v)}={\rm id}_{\be(Q\be {{{} A^{\e\prime}\atop \bigwedge}\atop }\lbe B^{\e\prime})
\be {{{} B^{\e\prime}\atop \bigwedge}\atop }\lbe \Sigma}
\]
and $g$ is a monomorphism. Consequently, we have $l=r$ which immediately yields that $k$ is a $(m,\omega)$-morphism.   
\end{proof}

\begin{corollary} {\rm (Naturality of the cohomology)} 
Diagram \eqref{big} induces the following {\it commutative} diagram of exact cohomology sequences:
\[
\xymatrix{1\ar[r]& H^{0}(E;\Phi)\ar[d]\ar[r]&\dots\ar[r]& H^{1}(E;\Phi^{\e\prime})\ar[d]\ar[r]\ar @{} [dr] |{\rm I}
& H^{\le 2}(E;\Phi) \ar[d]\ar[r] \ar @{} [dr] |{\rm II}& H^{\e 2}(E;\Phi^{\e\prime})\ar[d]\ar @{} [dr] |{\rm III}\ar[r]& H^{\e 2}(E;\Phi^{\e\prime\prime})\ar[d]\\
1\ar[r]& H^{0}(E;\Psi)\ar[r]&\dots\ar[r]& H^{1}(E;\Psi^{\e\prime})\ar[r]& H^{\le 2}(E;\Psi) \ar[r]& H^{\e 2}(E;\Psi^{\e\prime})\ar[r]& H^{\e 2}(E;\Psi^{\e\prime\prime}).
}
\]
\end{corollary}
\begin{proof}\indent
\begin{itemize}
\item The commutativity of the squares to the left of (I) follows from
the fact that our cohomology in dimensions 0 and 1 is the same as Giraud's.
\item  square I commutes by the previous theorem.
\item squares II and III  commute by Theorem \ref{332}.
\end{itemize}	
\end{proof}

\chapter{Relation to Giraud's $H^{2}$}\label{rel}

Let $A$ be a sheaf of groups on the site $E$.

According to Giraud, we have $H^{2}_{\rm g}(A)=H^{2}_{\rm g}(L)$ with $L={\rm lien}(A)$, i.e., $H^{2}_{\rm g}(A)$ is the set of all $L$-equivalence classes of ${\rm lien}(A)$-gerbes on $E$.

We recall that a ${\rm lien}(A)$-gerbe $F$ is a gerbe $F$ equipped with an isomorphism of bands
\[
a\colon L\circ f\isoto {\rm liau}(F).
\]
This means that for every object $x$ in $F_{S}$, $S\in{\rm Ob}(E)$, we have an isomorphism of bands on $S$
\[
a(x)\colon L(S)={\rm lien}(A)^{S}\isoto {\rm lien}(\uaut_{\e S}(x))
\]
such that the family of these isomorphisms is compatible 
with localization and, in addition, satisfies the condition that, for every $S$-isomorphism $i\colon x\to y$ in $F$, the morphism of sheaves of groups
$\uinn(i)\colon \uaut_{\e S}(x)\to \uaut_{\e S}(y)$ represents the identity morphism of $L(S)$.

\begin{theorem}\label{one} If $(F, a)$ is a ${\rm lien}(A)$-gerbe, then $F$also has a structure of $(A,\uinn(A)$-gerbe.
\end{theorem}
\begin{proof}\indent
\begin{itemize}
\item Let $x$ be an object of the fiber category $F_{S}$. Then we consider the following sheaf of sets on $E/S$
\[
\uisom_{\, a(x)}(A,\uaut_{\e S}(x)).
\]                                  
For every object $t\colon T\to S$ in $E/S$ we define $\uisom_{\, a(x)}(A,\uaut_{\e S}(x))(t)$ as the set of all the isomorphisms $\sigma\colon A^{T}\to\uaut_{\e S}(x)^{T}$ for which  ${\rm lien}(\sigma)=a(x)^{t}$. It is clear that $\uisom_{\, a(x)}(A,\uaut_{\e S}(x))$ is a right $\uinn(A)$-torsor on $E/S$. If we now set
\[
\mu_{\e a}(x)=\uisom_{\, a(x)}(A,\uaut_{\e S}(x)),
\]
then a morphism of gerbes is thus defined
\[
\mu_{\e a}\colon F\to {\rm TORSC}(E;\uinn(A))
\]
\item We obtain a natural isomorphism
\[
j_{\e a}(x)\colon\uaut_{\e S}(x)\isoto \mu_{a}(x)\! {{{} \uinn(A)\atop \bigwedge}\atop }\! A
\]   
by defining $j_{a}(x)(m)$, for every $S$-automorphism $m\colon x\to x$, as that element of $\mu_{\e a}(x)\! {{{} \uinn(A)\atop \bigwedge}\atop }\! A$ that is locally determined by $(\sigma,a)\in \mu_{\e a}(x)\times A$, where $\sigma$ is an
arbitrary element of $\mu_{\e a}(x)$ and $a$ is the unique element of $A$ such that $\sigma(a)= m$. Now suppose conversely that $\Theta$ is any element of
$\mu_{\e a}(x)\be {{{} \uinn(A)\atop \bigwedge}\atop }\lbe A$  that is locally determined by a pair $(\sigma,a)$. Then $j_{a}^{-1}(x)(\Theta)$ is that  $S$-automorphism of $x$ that is locally equal to $\sigma(a)$
\item Working locally, we thus have
\[
\mu_{\e a}(\e j_{a}^{-1}(x)(\Theta))(\sigma)=\mu_{\e a}(\sigma(a))(\sigma)=\inn(\sigma(a))\circ
\sigma=\sigma\circ \inn(a).
\]
We conclude that $(F,\mu_{\e a},j_{a})$ is an $(A,\uinn(A))$-gerbe.
\end{itemize}
\end{proof}

\begin{theorem} Let $(F, a)$ and $(G, b)$ be two ${\rm lien}(A)$-gerbes on $E$.
If $\delta\colon F\to G$ is a ${\rm lien}(A)$-equivalence, then $\delta$ is an $(A,\uinn(A))$-equivalence from $(F,\mu_{\e a},j_{a})$ to $(G,\mu_{\e b},j_{b})$.
\end{theorem}
\begin{proof} By composition with $\aut(\delta)_{x}\colon\uaut_{\e S}(x)\isoto  \uaut_{\e S}(\delta(x))$ we obtain an isomorphism $\mu_{\e a}(x)\isoto\mu_{\e b}(\delta(x))$. The latter map composed with the canonical isomorphism 
\[
\mu_{\e a}(x)\be(Q\be {{{} \uinn(A)\atop \bigwedge}\atop }\lbe \uinn(A)\isoto \mu_{\e a}(x)
\]	
yields a natural isomorphism
\[
i_{x}\colon \mu_{\e a}(x)\be{{{} \uinn(A)\atop \bigwedge}\atop }\lbe \uinn(A)\isoto \mu_{\e b}(\delta(x)).
\]  
In addition, the following equality holds:
\[
(\e i_{x}\wedge A\e)\circ(\e p_{\le x}\wedge A\e)\circ j_{a}(x)=j_{\e b}(\delta(x))\circ\aut(\delta)_{x},
\]
which can be easily checked locally. Thus $\delta$ is an $(A,\uinn(A))$-equivalence.	
\end{proof}

\begin{corollary} We have a map
\begin{equation}\label{comp}
H^{2}_{g}(A)\to H^{2}(A,\uinn(A))
\end{equation}
which maps the class $[F]$ of the ${\rm lien}(A)$-gerbe $F$ to the class 
$[[F]]$ of the $(A,\uinn(A))$-gerbe $F$.
\end{corollary}

\begin{theorem}\label{bij} The map \eqref{comp} is bijective.
\end{theorem}
\begin{proof}\indent
\begin{enumerate}
\item Injectivity.

Let $(F, a)$ and $(G, b)$ be two ${\rm lien}(A)$-gerbes on $E$.
Assume that $[[F]]=[[G]]$, i.e., there exists an $(A,\uinn(A))$-equivalence
\[
\delta\colon F\to G
\]
This means that, for every object $x$ of $F_{S}$, $S\in{\rm Ob}$, the following equality holds:
\[
(\e i_{x}\wedge A\e)\circ(\e p_{\le x}\wedge A\e)\circ j_{\e a}(x)=j_{\e b}(\delta(x))\circ\aut(\delta)_{x}.
\]
By interpreting this locally we find that, if $\sigma\in\mu_{\e a}(x)$, then $\aut(\delta)_{x}\circ\sigma\in\mu_{\e b}(\delta(x))$ and this expresses that $\delta\colon F\to G$ is an equivalence of ${\rm lien}(A)$-gerbes. Thus we have $[F]=[G]$.
\item Surjectivity.

Let $(G,\mu, j)$ be an $(A,\uinn(A))$-gerbe on the site $E$.

From this we can deduce a ${\rm lien}(A)$-gerbe structure on $G$.
For every $S\in{\rm Ob}(E)$ and every object $x$ of $G_{S}$ we need to find an isomorphism of bands $a_{\e\mu,\e j}\colon {\rm lien}(A)^{S}\isoto{\rm lien}(\uaut_{\e S}(x))$ such that some compatibility conditions are satisfied. In other words, we should seek to define {\it local} isomorphisms of groups
$A\to \uaut_{\e S}(x)$	that form a coherent family up to inner automorphisms of $A$.

Since $\mu(x)$ is an $\uinn(A)$-torsor, it has a local section. Consequently,
we can consider locally an element $\alpha$ in $\mu(x)$.  Using this element we define then a local isomorphism
\[
\sigma_{a}\colon A\isoto \uaut_{\e S}(x)
\]
by $\sigma_{a}=j_{x}^{-1}\circ l_{\alpha}^{-1}$, where $l_{\alpha}$ is the following isomorphism
\[
\mu(x)\be {{{} \uinn(A)\atop \bigwedge}\atop }\lbe A\to A,\,\Theta=(\alpha,a)^{*}\mapsto a.
\]   
Here $(\alpha,a)^{*}$ is the image of $(\alpha,a)$ under the morphism
\[
\mu(x)\times A\to\mu(x)\be {{{} \uinn(A)\atop \bigwedge}\atop }\lbe A.
\]   
If $\beta$ is another element of $\mu(x)$ then we have $\beta=\alpha\cdot\inn(a_{\le 0})$ whence it follows that  $\sigma_{\beta}=\sigma_{\alpha}\circ\inn(a_{\le 0}^{-1})$.  The family of isomorphisms of {\it bands} induced by these group isomorphisms is coherent and therefore determines an isomorphism of bands on $S$
\[
a_{\mu,j}\colon {\rm lien}(A)^{S}\isoto{\rm lien}(\uaut_{\e S}(x))
\]
Thus we obtain the ${\rm lien}(A)$-gerbe $(G,a_{\mu,j})$.

According to Theorem \ref{one}, we can then associate an $(A,\uinn(A))$-gerbe which we denote by $(G,\mu_{\e a},j_{a})$. For every object $x$ of $G_{S}$, we have
$\mu_{\e a}(x)=\uisom_{\, a_{\mu,\e j}(x)}(A,\uaut_{\e S}(x))$ and $j_{a}(x)\colon 
\uaut_{\e S}(x)\isoto \mu_{\e a}(x)\be {{{} \uinn(A)\atop \bigwedge}\atop }\lbe A$ is the isomorphism that to an element $m\in\uaut_{\e S}(x)$ associates that element of $\mu_{\e a}(x)\be {{{} \uinn(A)\atop \bigwedge}\atop }\lbe A$ that is locally determined by $(\sigma_{\alpha},a)$ with $\sigma_{a}=j_{x}^{-1}\circ l_{\alpha}^{-1}$ and $a$ is the unique element of $A$ such that $\sigma_{\alpha}(a)=m$.

This $(A,\uinn(A))$-gerbe $(G,\mu_{\e a},j_{a})$ is $(A,\uinn(A))$-equivalent to the original gerbe $(G,\mu, j)$  because ${\rm id}_{\e G}\colon G\to G$ is an $(A,\uinn(A))$-morphism.

The map $H^{2}_{g}(A)\to H^{2}(A,\uinn(A))$ is thus surjective.

\end{enumerate}
	
\end{proof}

\begin{remark} From the previous theorem we see that the ${\rm lien}(A)$-gerbes of Giraud agree with our $(A,\uinn(A))$-gerbes.
\end{remark}

\begin{corollary} In the {\rm abelian case}, i.e., $A$ is abelian and $\Pi$ is the trivial group, our $H^{2}$ coincides with the $H^{2}$ in abelian cohomology.	
\end{corollary}
\begin{proof} If $A$ is abelian and $\Pi$ is trivial, then $\Pi=\uinn(A)$ and Theorem \ref{bij} yields a bijection $H^{2}(A,\Pi)\isoto H^{2}_{\rm g}(A)$. If $H^{2}_{\bullet}(A)$ denotes the cohomology group defined in homological algebra, then \cite[IV, 3.4.2]{14} yields an isomorphism $\alpha_{2}\colon 
H^{2}_{\bullet}(A)\isoto H^{2}_{\rm g}(A)$. The corollary follows.
\end{proof}

\subsubsection{Some final considerations}

\begin{enumerate}
\item It is now clear why the cohomology of Giraud is {\it not} functorial. 
In the investigation of the obstruction to the lifting of an $A^{\e\prime\prime}$-torsor $P$ to $A^{\e\prime}$, Giraud found that this is a gerbe $K(P)$ to which one cannot associate a sheaf of groups but a new object ${\rm lien}(A)$ that he called a band. However, it was not noted that in the obstruction also a morphism of gerbes $\mu\colon K(P)\to{\rm TORSC}(E;\uinn(A^{\e\prime}))$ is involved.  It is therefore not surprising that to a given morphism of sheaves of groups
$f\colon A\to A^{\e\prime}$ no map $H^{2}_{\rm g}(A)\to H^{2}_{\rm g}(A^{\e\prime})$, in general, can be defined. After all, we have seen that $H^{2}_{\rm g}(A)\isoto H^{2}(A,\uinn(A))$ and $H^{2}_{\rm g}(A^{\e\prime})\isoto H^{2}(A^{\e\prime},\uinn(A^{\e\prime}))$ so, in addition to the morphism $f\colon A\to A^{\e\prime}$, one still needs to construct a morphism $\varphi\colon\uinn(A)\to \uinn(A^{\e\prime})$ in order to construct a map from $H^{2}_{\rm g}(A)$ to $H^{2}_{\rm g}(A^{\e\prime})$. Moreover, the case in which he does have a map from
$H^{2}_{\rm g}(A)$ to $H^{2}_{\rm g}(A^{\e\prime})$ occurs when the condition $C_{f}=C_{A^{\e\prime}}$ holds. But if this condition is fulfilled then, just by using $f\colon A\to A^{\e\prime}$, we can define a morphism
\[
\varphi\colon \uinn(A)\to\uinn(A^{\e\prime}), \inn(a)\mapsto \inn(f(a)),
\]
so that $(f,\varphi)$ is a morphism of sheaves of crossed groups.
If the condition $C_{f}=C_{A^{\e\prime}}$ does not hold, then Giraud can only define a {\it relation} $H^{2}_{\rm g}(A)\ggto{} H^{2}_{\rm g}(A^{\e\prime})$.

\item When a short exact sequence 
\begin{equation}\label{tres}
1\to A\overset{\!u}{\to}B\overset{\!v}{\to}C\to 1
\end{equation}
is given, then we have seen that Giraud's $H^{2}_{\rm g}(A)$ is not sufficient to account for all the obstructions to lifting $C$-torsors to $B$.
He therefore had to define a set $O(v)$ expressly so that it could contain all the obstructions.  Why this fails is now also clear.
To the short exact sequence \eqref{tres} one can associate the following short exact sequence of sheaves of crossed groups:
\[
\xymatrix{&\uinn(B)\ar[r]^(.45){\rm id}&\uinn(B)\ar[r]&\uinn(C)\\
1\ar[r]&A\ar[u] \ar[r]^(.4){u}& B\ar[u]\ar[r]^{v}& C\ar[u]\ar[r]&1
}
\]
Consequently, the obstructions end up in $H^{2}(A,\uinn(B))$ and {\it not}
in $H^{2}(A,\uinn(A))$, which is the set that is in bijective correspondence with Giraud's $H^{2}(A)$, as already noted.

\item Finally, we mention two questions that still await a satisfactory answer.
\begin{enumerate}
\item[\rm (i)] what connection does there exist between the $H^{2}$ defined here and the $H^{2}$ defined by Dedecker when $E$ is equal to the topos $B_{H}$ of left $H$-sets?

\item[\rm (ii)] what is the connection with the cohomology of a topological space defined using hypercoverings on that space? More should be examined to determine how an $(A,\Pi)$-bundle itself can be interpreted in relation to a hypercovering. 
\end{enumerate}
\end{enumerate}

\end{document}